\documentclass[twoside]{amsart}
 \usepackage[colorlinks=false]{hyperref}
\usepackage{latexsym}
\usepackage{mathrsfs}
\usepackage{pstricks}
\usepackage{amssymb}
\usepackage{amsmath}
\usepackage{amsfonts}
\usepackage{amsthm,array}
\usepackage{mathtools}
\usepackage{epsfig}
\usepackage{verbatim}
\usepackage{graphicx}
\usepackage{float}

  \usepackage{cleveref}
  \usepackage{graphicx}
  \usepackage{subfigure}
  \usepackage{tikz}
  \usepackage{tabularx}
  \usepackage{booktabs}
  \usepackage{algorithm}
  \usepackage{algorithmic}
  \DeclareMathOperator{\id}{Id}
  \DeclareMathOperator*{\myspan}{span}
\newcommand{\amal}[5]{#1\prescript{#4}{#5}\times_{#3}#2}
\newcommand{\eqdeg}[1]{#1\mbox{\rm -}\deg}

  \newcommand{\abs}[1]{\left\lvert#1\right\rvert}
  
  \newcommand{\set}[1]{\left\{#1\right\}}

  \newcommand{\sgnrep}[1]{\mathfrak{1}^-_{#1}}

\def\bc{{\mathbb{C}}}

\def\bn{{\mathbb{N}}}

\def\br{{\mathbb{R}}}

\def\bz{{\mathbb{Z}}}

\def\br{\mathbb R}

\def\id{{\text{\rm Id}}}

\def\wt{\widetilde }

\def\vs{\vskip.3cm}
\def\noi{\noindent}
\def\wt{\widetilde}

\def\Ker{\text{\rm Ker\,}}
\def\Im{\text{\rm Im\,}}
\def\vp{\varphi}
\def\ve{\varepsilon}
\def\id{\text{\rm Id\,}}

  \def\tD{{\widetilde D}}

  \def\td{{\tilde d}}

\newcommand\cU{\ensuremath{\mathcal U}}
\newcommand\cV{\ensuremath{\mathcal V}}

\newrgbcolor{violet}{.6 .1 .8}
\newrgbcolor{lightyellow}{1 1 .8}
\newrgbcolor{lightblue}{.80 1 1}
\newrgbcolor{mygreen}{0 .66 .05}
\definecolor{mygreen}{rgb}{0,.66,.05}
\definecolor{lightyellow}{rgb}{1,1,.80}
\newrgbcolor{orange}{1 .6 0}
\newrgbcolor{GreenYellow}{.85 1 .31}
\newrgbcolor{Yellow}{1  1  0}
\newrgbcolor{Goldenrod}{1  .90  .16}
\newrgbcolor{Dandelion}{1  .71  .16}
\newrgbcolor{Apricot}{1  .68  .48}
\newrgbcolor{Peach}{1  .50  .30}
\newrgbcolor{Melon}{1  .54  .50}
\newrgbcolor{YellowOrange}{1  .58  0}
\newrgbcolor{Orange}{1  .39  .13}
\newrgbcolor{BurntOrange}{1  .49  0}
\newrgbcolor{Bittersweet}{1.  .4300  .24}
\newrgbcolor{RedOrange}{1  .23  .13}
\newrgbcolor{Mahogany}{1.  .4475  .4345}
\newrgbcolor{Maroon}{1.  .4084  .5376}
\newrgbcolor{BrickRed}{1.  .3592  .3232}
\newrgbcolor{Red}{1  0  0}
\newrgbcolor{OrangeRed}{1  0  .50}
\newrgbcolor{RubineRed}{1  0  .87}
\newrgbcolor{WildStrawberry}{1  .04  .61}
\newrgbcolor{Salmon}{1  .47  .62}
\newrgbcolor{CarnationPink}{1  .37  1}
\newrgbcolor{Magenta}{1  0  1}
\newrgbcolor{VioletRed}{1  .19  1}
\newrgbcolor{Rhodamine}{1  .18  1}
\newrgbcolor{Mulberry}{.6668  .1180  1.}
\newrgbcolor{RedViolet}{.9538  .4060  1.}
\newrgbcolor{Fuchsia}{.5676  .1628  1.}
\newrgbcolor{Lavender}{1  .52  1}
\newrgbcolor{Thistle}{.88  .41  1}
\newrgbcolor{Orchid}{.68  .36  1}
\newrgbcolor{DarkOrchid}{.60  .20  .80}
\newrgbcolor{Purple}{.55  .14  1}
\newrgbcolor{Plum}{.50  0  1}
\newrgbcolor{Violet}{.98 .15 .95}
\newrgbcolor{RoyalPurple}{.25  .10  1}
\newrgbcolor{BlueViolet}{.84  .38  .98}
\newrgbcolor{Periwinkle}{.43  .45  1}
\newrgbcolor{CadetBlue}{.38  .43  .77}
\newrgbcolor{CornflowerBlue}{.35  .87  1}
\newrgbcolor{MidnightBlue}{.4414  .9259  1.}
\newrgbcolor{NavyBlue}{.06  .46  1}
\newrgbcolor{RoyalBlue}{0  .50  1}
\newrgbcolor{Blue}{0  0  1}
\newrgbcolor{Cerulean}{.06  .89  1}
\newrgbcolor{Cyan}{0  1  1}
\newrgbcolor{ProcessBlue}{.04  1  1}
\newrgbcolor{SkyBlue}{.38  1  .88}
\newrgbcolor{Turquoise}{.15  1  .80}
\newrgbcolor{TealBlue}{.1572  1.  .6668}
\newrgbcolor{Aquamarine}{.18  1  .70}
\newrgbcolor{BlueGreen}{.15  1  .67}
\newrgbcolor{Emerald}{0  1  .50}
\newrgbcolor{JungleGreen}{.01  1  .48}
\newrgbcolor{SeaGreen}{.31  1  .50}
\newrgbcolor{Green}{0  1  0}
\newrgbcolor{ForestGreen}{.1992  1.  .2256}
\newrgbcolor{PineGreen}{.3100  1.  .5575}
\newrgbcolor{LimeGreen}{.50  1  0}
\newrgbcolor{YellowGreen}{.56  1  .26}
\newrgbcolor{SpringGreen}{.74  1  .24}
\newrgbcolor{OliveGreen}{.6160  1.  .4300}
\newrgbcolor{RawSienna}{.53  .28  .16}
\newrgbcolor{Sepia}{1.  .7510  .70}
\newrgbcolor{Brown}{.41  .25  .18}
\newrgbcolor{TAN}{.86  .58  .44}
\newrgbcolor{Gray}{1.  1.  1.}
\newrgbcolor{Black}{1  1  1}
\newrgbcolor{White}{1  1  1}
\newtheorem{theorem}{Theorem}[section]
\newtheorem{proposition}{Proposition}[section]
\newtheorem{lemma}{Lemma}[section]
\newtheorem{corollary}{Corollary}[section]

\newtheorem{definition}{Definition}[section]

\newtheorem{remark}{Remark}[section]
\newtheorem{example}{Example}[section]

\newtheorem{remark-definition}{Remark and Definition}[section]
\newtheorem{rem-not}{Remark and Notation}[section]

\begin{document}

\title{Multiple Periodic Solutions for $\Gamma$-Symmetric Newtonian
Systems}
\author{Mieczyslaw Dabkowski}
\address{Department of Mathematical Sciences, the University of Texas at
Dallas,
 Richardson, TX, 75080-3021, U.S.A.}
 \email{mdab@utdallas.edu}
 \author{ Wieslaw Krawcewicz} 
 \address{Department of Mathematical Sciences, the University of Texas at
Dallas,
 Richardson, TX, 75080-3021, U.S.A. \\
 and\\
 College of Mathematics and Information Sciences, Guangzhou University,
Guangzhou, 510006
China}
 \email{wieslaw@utdallas.edu}
\author{Yanli Lv}
\address{Department of Mathematical Sciences, the University of Texas at
Dallas,
 Richardson, TX, 75080-3021, U.S.A.\\ and\\
 College of Science, China Three Gorges University,
Yichang, 443002, P.R.C.}
\email{yxl103720@utdallas.edu}
\author{Hao-Pin Wu}
\address{Department of Mathematical Sciences, the University of Texas at
Dallas,
 Richardson, TX, 75080-3021, U.S.A.}
 \email{hxw132130@utdallas.edu}
\date{}
\maketitle

\begin{abstract} The existence of periodic solutions in $\Gamma$-symmetric Newtonian systems $\ddot x=-\nabla f(x)$ can be effectively studied by means of the $\Gamma\times O(2)$-equivariant gradient degree with values in the Euler ring $U(\Gamma\times O(2))$. In this paper we show that in the case of $\Gamma$ being a finite group, the Euler ring $U(\Gamma\times O(2))$ and the  related basic degrees are effectively computable using Euler ring homomorphisms, the Burnside ring $A(\Gamma\times O(2)) $, and the reduced   $\Gamma\times O(2)$-degree with no free parameters.  We present several examples of Newtonian systems with various symmetries, for which we show existence of  multiple periodic solutions. We also provide exact value of the equivariant topological invariant for those problems. 
\end{abstract}

\footnote{\hskip.1cm 2010 \textit{AMS Mathematics Subject Classification:}
37J45, 37C25, 37C80, 37K05, 47H11, 55M25} \footnote{\hskip.1cm $^{\dagger }$ Corresponding author,
E-mail: wieslaw@utdallas.edu.} \footnote{\hskip.1cm\textbf{Key Words and
Phrases:} Newtonian systems, symmetries, periodic solutions, Euler ring, Burnside ring, equivariant degree, variational problem}

\section{Introduction} In this paper we are  studying the existence of non-constant $p$-periodic solutions in $\Gamma$-symmetric  Newtonian systems (here  $\Gamma$ is assumed to be  a finite group) of the type
\begin{equation}\label{eq:int1}
\ddot x(t)=-\nabla f(x(t)), \quad x(t)\in V,
\end{equation}
where $V:=\br^n$ is an orthogonal $\Gamma$-representation and  $f:V\to \br$ a $C^2$-differentiable $\Gamma$-invariant function. By rescaling the time, this problem  can be reduced  to the following one-parameter system
\begin{equation}\label{eq:New-L}
\begin{cases}
\ddot x(t)=-\lambda^2 \nabla f(x(t)), \quad x(t)\in V,\\
x(0)=x(2\pi),\;\; \dot x(0)=\dot x(2\pi).
\end{cases}
\end{equation}
Many methods applied to  \eqref{eq:int1} were inspired   by the  Hamiltonian systems  of type
\begin{equation}\label{eq:int2}
\dot x(t)=J\nabla f(x(t)), \quad x(t)\in V
\end{equation}
where $V=\br^{2n}$ and $J=\left(
           \begin{array}{cc}
             0 & -\id\\
             \id & 0 \\
           \end{array}
         \right)$
is the symplectic matrix, for which the existence of $2\pi$-periodic solutions was intensively studied 
(see for example \cite{AB,Amb,Bar,IE1,IE2,IE3,GF1,GF2,MG1,GM,YL1,YL3,YL4,MP,MW,RP,Rab,Ryba}). Similar methods were also developed for  the system \eqref{eq:int1} (see \cite{AZ,BF,GI1,GI2,GI3,GI4,GI5,Wang}).  After K.~Geba introduced the concept of the gradient equivariant degree (cf. \cite{G}) leading to a development of  new equivariant-theoretical methods, several interesting papers  by A.~Golebiewska,  J.~Fura, A.~Ratajczak, W.~Radzki, H.~Ruan and  S.~Rybicki were published (cf. \cite{FRR,FRRuan,GolRyb,RR,RY2,RY4,RY3,RybSurvey}).  These authors applied the   $\Gamma\times SO(2)$-equivariant degree to study the existence of multiple $2\pi$-periodic solutions to \eqref{eq:int1}. However, there is a significant difference between the systems \eqref{eq:int1} and \eqref{eq:int2}: the system \eqref{eq:int1} is time-reversible, so it leads to a variational problem with $\Gamma\times O(2)$-symmetries. As the equivariant gradient degree provides {\it full} equivariant topological classification of the critical set for the related to \eqref{eq:int1} variational functional $\mathscr J:H^1(S^1;V)\to \br$, it is important to consider the {\it full} symmetry group for $\mathscr J$. In addition, if the function $f$ is even, then the associated with \eqref{eq:int1} functional $\mathscr J$ becomes $\Gamma\times \bz_2\times O(2)$-invariant. It is our strong conviction that in order to make the most efficient use of the equivariant degree methods, one cannot ignore symmetric properties of $\mathscr J$ for the sake of computational simplifications. 
\vs
It seems that the main difficulty related to the usage of the $G$-equivariant gradient degree  method  is related to the topological sophistication of its definition. The $G$-equivariant gradient degree takes its values  in the Euler ring $U(G)$, which was introduced by T. tom Dieck. In his monograph (cf. \cite{tD}) tom Dieck presented several cohomological formulae used for the computations of the ring multiplication in $U(G)$. However, in the case of an arbitrary compact Lie group $G$, these computations may  constitute extraordinary challenges. Very often,  the users of the $G$-equivariant degree theory are restricted  by these computational obstacles only  to certain  (simpler or well worked out) types of symmetry groups $G$, leading inevitable to  ignoring the full symmetry group for the related problem (cf. \cite{GR,RR,BKR}). Notice that  $U(G)$ is a generalization of the so-called Burnside ring $A(G)$, which is also a part of the ring $U(G)$.  While the computations of $U(G)$ remain a difficult task, the Burnside ring $A(G)$  is relatively easy to describe.  One should point out, that other algebraic structures were also used for the ranges of various versions of equivariant degrees. For instant the twisted $\Gamma\times S^1$-equivariant degree with one parameter is taking its values in the $A(\Gamma)$-module $A_1(\Gamma\times S^1)$ (see \cite{AED}) and the twisted $\Gamma\times \mathbb T^n$-equivariant degree with $n$-parameters is taking its values in $A(\Gamma)$-module $A_1(\Gamma\times \mathbb T^n)$ (see \cite{DKL}). These additional structures, for which there is a well established computational base, can be used in order to simplify the computations of the Euler ring $U(G)$.

\vs
Consider the Euler ring homomorphism $\Psi: U(\Gamma\times O(2))\to U(\Gamma\times SO(2))$  induced by the inclusion $i:\Gamma\times SO(2)\hookrightarrow \Gamma\times O(2)$ (cf. \cite{tD,BKR}.
In this paper we will show that in the case $G:=\Gamma\times O(2)$ (where $\Gamma$ is a finite group) the Euler ring $U(G)$ ring can be fully described by the means of the multiplication in the Burnside ring $A(G)$ and the homomorphism $\Psi$,  using the known structure of the Euler ring $U(\Gamma\times SO(2))$ and the Burnside ring $A(\Gamma\times O(2))$ (see \cite{RR,BKR} for more details). More precisely,   one can obtain a complete set of relations for the multiplication in the Euler ring $U(\Gamma\times O(2))$. These relations can be effectively implemented in a special computer software designed for this type of groups. Since several computer programs are presently developed for the Euler ring $U(\Gamma\times SO(2))$ and for the Burnside ring $A(\Gamma\times O(2))$ (see  \url{https://bitbucket.org/psistwu/gammao2} for the most recent updates), we believe that one can expect such programs for the computation of $U(\Gamma\times O(2))$ and other objects related  to the equivariant degree theory.
\vs
Clearly, efficient applications of the gradient equivariant degree depend on more than just the structure of the Euler ring. For instance, in order to compute the gradient equivariant degrees of linear maps, one also needs a database composed of the so-called gradient basic degrees (which are simply the degrees of $-\id$ on irreducible representations). Also it is possible to use a topological definition of the equivariant gradient degree  (cf. \cite{GR}) one can also apply the Euler homomorphism for such computations (cf. \cite{RR,BKR}). 
However, for a given compact Lie group $G$, there is also the so-called $G$-equivariant degree with no free parameter taking its values in the Burnside ring $A(G)$  (see \cite{AED,survey}) which  can be considered as a part of the $G$-equivariant gradient degree. Consequently, all the algorithmic computational formulae for the $G$-equivariant degree with no free parameter can be used in order to obtain partial values of the $G$-equivariant gradient degree, in particular the gradient basic degrees. In this paper we will show that this information is sufficient to fully reconstruct the complete values of the gradient basic degrees in the case of $G=\Gamma\times O(2)$. 
\vs
Finally, let us discuss the applications of the $\Gamma\times O(2)$-equivariant degree to the systems \eqref{eq:int1} and \eqref{eq:New-L}. In the case when it is possible to obtain {\it a priori} bounds for the $2\pi$-periodic solutions of these equations, one can compute the $\Gamma\times O(2)$-equivariant invariant $\omega\in U(G)$ (which could be for example a difference of the equivariant gradient degrees on large and small ball). Such cases are for example when the system \eqref{eq:int1} is asymptotically linear or satisfy a Nagumo-type growth condition. Then clearly the existence of non-trivial solutions (i.e. outside the small ball) can be concluded by the fact that $\omega\not=0$. However, one can be also interested to predict the existence of   multiple $2\pi$-periodic solutions with  different types of symmetry. In such a case, the coefficients of $\omega$ corresponding to the so-called {\it maximal orbit type} can provide the crucial information in order to formulate such results. But the {\it maximality} of  such  obit types implies that it is a generator of the Burnside ring $A(G)$, therefore it can actually be detected by the $\Gamma\times O(2)$-equivariant degree with no free parameter, which can be much easier computed than the equivariant gradient degree.
Similar arguments apply to the system \eqref{eq:New-L}, which we can consider as a bifurcation problem with a parameter $\lambda$. More precisely, in this case we are looking for critical values $\lambda_o$ of the parameter $\lambda>0$, to which we can associate the $\Gamma\times O(2)$-equivariant gradient bifurcation invariants $\omega(\lambda_o)\in U(\Gamma\times O(2))$ classifying the bifurcation of $2\pi$-periodic solutions from the zero solution. The existence and multiplicity of such bifurcating branches of $2\pi$-periodic solutions can be described from the information contained in the invariants  $\omega(\lambda_o)$.
 Consequently,  all the essential information needed to establish the existence and multiplicity results for the systems \eqref{eq:int1} and \eqref{eq:New-L} can be extracted from the $\Gamma\times O(2)$-equivariant degree (with no free parameter) of $\mathscr J$  which takes values in the Burnside ring $A(G)$.   It is clear that the $\Gamma\times O(2)$-equivariant degree without free parameter can be easily  computed (without getting entangled in  complicated technical details), has similar properties and provides enough information for analyzing these problems.

 \vs
 Nevertheless, let us emphasize that only the equivariant invariants $\omega\in U(G)$ (without truncation of its coefficients) provide a complete equivariant topological classification for the related solution sets to \eqref{eq:int1} or \eqref{eq:New-L}. 
 
\vs
To illustrate the usage and the computations of the associated with the systems \eqref{eq:int1} and \eqref{eq:New-L} equivariant invariants, in section 7 we present several examples of symmetric Newtonian systems, for which the exact values of the associated invariants can be effectively computed. These computations are possible only with the assistance of computer programs  in GAP (developed by H-P. Wu), which were developed for several different groups $\Gamma$  (cf. \cite{Pin}). 

\vs \paragraph{\bf Acknowledgment}The two last authors acknowledge the support from National Science Foundation through grant DMS-1413223. The third author was also supported by National Natural Science Foundation of China (no. 11301102).
\vs

\section{Preliminaries}

In this section, we present all the preliminary notions that will be used
in the remaining parts of the paper.  In this section we recall some basic notations and results
relevant to this paper. In what follows $G$ always stands for a compact Lie
group and whenever we refer to $H$ as a subgroup of $G$, i.e. 
 $H\leq G$, we will assume that $H$ is a closed subgroup of $G$.

\subsection{$G$-Actions, Sets $N(L,H)$, $N(L,H)/H$ and $N(L,H)/N(H)$}
Let $H\le G$.  We denote by $N(H)$ the
normalizer of $H$ in $G$ and by $W(H)=N(H)/H$ the Weyl group of $H$.  The symbol $(H)$ will stand for the conjugacy class of $H$ in $G$. We also will write $N_{G}(H)$, $W_{G}(H)$ and $(H)_G$ in order to clarify that 
$H$ is considered as a subgroup of $ G$.  We
also use the following notations: 
\begin{equation*}
\Phi (G)=\{(H):H\;\;\text{is a subgroup of }\;G\}\text{ and }\Phi
_{n}(G)=\{(H)\in \Phi (G):\text{\textrm{dim\thinspace }}W(H)=n\}.
\end{equation*}%
The set $\Phi (G)$ has a natural partial order defined by 
$(H)\leq (K)$ iff $\exists g\in G\;\;gHg^{-1}\leq K$. 
\vs

A topological space $X$ equipped with a left (resp. right) $G$-action is
called a $G$\textit{-space} (resp. {space-}$G$) and if an action is not
specified we assume that $G$ acts from the left. For a $G$-space $X$ and $x\in X$, we denote by: 
$G_{x} :=\{g\in G:gx=x\}$  -- the {\it isotropy group}  of $x$, $G(x) :=\{gx:g\in G\}$ -- the {\it orbit} of $x$,
$(G_{x})$ ---  the {\it orbit type} of $x\in X$
For a subgroup $H\leq G$, we also put:
\begin{align*}
X_{H}& :=\{x\in X:G_{x}=H\}; \\
X^{H}& :=\{x\in X:G_{x}\geq H\}; \\
X_{(H)}& :=\{x\in X:\;(G_{x})=(H)\}; \\
X^{(H)}& :=\{x\in X:\;(G_{x})\geq (H)\}
\end{align*}
and we will denote
\[
\Phi(G;X):=\{(G_x): x\in X\} \;\;\; \text{ and }\;\;\; \Phi_n(G;X):=\Phi(G;X)\cap \Phi_n(G).
\]
It is well-known, $W(H)$ acts on $X^{H}$ and it acts
freely on $X_{H}$. The orbit space of $G$ acting on $X$ will be denoted by $X/G$ (we use $G\backslash X$ for $G$ acting on the right).
\vs 
 Let $G_{1}$ and $G_{2}$ be compact Lie groups. Assume that $X$ is
a $G_{1}$-space and space-$G_{2}$ simultaneously  and the following condition
is satisfied:
\begin{equation*}
g_{1}(xg_{2})=(g_{1}x)g_{2},
\end{equation*}
for all $x\in X$, $g_{i}\in G_{i}$, $i=1,$ $2$ (i.e. we assume that the double action $G_1\times X\times G_2\to X$ is continuous). In this case, we say that $X$
is a $G_{1}$-space-$G_{2}$. Clearly, $X/G_{1}$ is a space-$G_{2}$ while $G_{2}\backslash X$ is a $G_{1}$-space. For the double orbit spaces $G_{2}\backslash (X/G_{1})$ and $(G_{2}\backslash X)/G_{1}$, we use the
notation $G_{2}\backslash X/G_{1}$ since both double orbit spaces are
homeomorphic. In particular, if $X=G$ and $H\leq G$ (resp. $L\leq G$) acts
on $G$ by left (resp. right) translations, $G$ is an $H$-space-$L$.
Moreover, $G/H$ (resp. $L\backslash G$) can be identified with the set of
left cosets $\{Hg:g\in G\}$ (resp. right cosets $\{gL:g\in G\}$), and $L$
(resp. $H$) acts on $G/H$ (resp. $L\backslash G$) as follows
\begin{equation*}
(Hg)l=H(gl),l\in L(\text{resp}.h(gL)=(hg)L,(h\in H)).
\end{equation*}
In addition, $L\backslash G/H$ can be identified with the set of the
corresponding double cosets. If we replace $G$ by a $G$-invariant subset of $G$ analogous observations can be applied\footnote{%
For the equivariant topology background used in this paper, we refer to \cite{tD,BtD,Kawa,Bred,AED}.}.\medskip
\vs
Let $L\leq H\leq G$. Define 
\begin{equation}
N(L,H)=\{g\in G:gLg^{-1}\leq H\}.  \label{eq:NLH}
\end{equation}
One can show ( see \cite{AED}, Lemma 2.55) that \textrm{dim}\thinspace $W(H)\leq \text{\textrm{dim\thinspace }}W(L)$ and $N(L,H)$ is a $N(H)$-space-$N(L)$ (cf. \cite{BKR,Kawa}). Furthermore, $N(L,H)/H$ is space-$N(L)$, hence
it is also a space-$W(L)$. If one takes a $G$-space $G/H$ then:
\begin{itemize}
\item[(i)] $(G/H)^{L}$ is $W(L)$-equivariantly diffeomorphic to $N(L,H)/H$;\medskip

\item[(ii)] $(G/H)^{L}$ contains finitely many $W(L)$-orbits
(cf. \cite{Bred}, Corollary 5.7).\medskip
\end{itemize} 
 Consider set $N(L,H)/N(H)$ and let 
\begin{equation*}
n(L,H)=|N(L,H)/N(H)|,
\end{equation*}
where $|X|$ stands for the cardinality of $X$ (cf. \cite{GolSchSt}). As it
was shown in \cite{AED} (see Proposition 2.52), if \textrm{dim}\thinspace $W(L)=\text{\textrm{dim\thinspace }}W(H)$, then $n(L,H)$ is finite. Moreover,
one can show (cf. \cite{BKR}) that there is finite sequence $g_{1},g_{2},$\dots ,$g_{n}$, of elements of $G$ (where $g_{1}=e$ is the identity of $G$), such that $N(L,H)$ is a disjoint union of $N(H)g_{j}N(L)$, $j=1,2,\dots,n $ (here $N(H)g_{j}N(L)$ denotes a double coset), i.e. 
\begin{equation*}
N(L,H)=N(H)g_{1}N(L)\sqcup N(H)g_{2}N(L)\sqcup \dots \sqcup N(H)g_{n}N(L),
\end{equation*}
For more information about the numbers $n(L,H)$ and their properties we
refer to \cite{AED}.
\vs
\subsection{Isotypical Decomposition 
\label{subsec:G-represent}}
As is well-known, any compact Lie group $G$ admits countably many
non-equivalent real (resp. complex) irreducible representations. Therefore,
we will assume that a complete list of all real (resp. complex) irreducible $G$-representations $\mathcal{V}_{i}$, $i=0,$ $1,$ $\ldots $ (resp. $\mathcal{U}_{j}$, $j=0,1,\dots $) is given. Let $V$ (resp. $U$) be a finite-dimensional real (resp. complex) $G$- representation and (without loss of generality) we may assume that $V$ (resp. $U$) is an orthogonal
(resp. unitary)) representation. Then, $V$ (resp. $U$) decomposes into a direct sum 
\begin{equation}
V=V_{0}\oplus V_{1}\oplus \dots \oplus V_{r},  \label{eq:Giso}
\end{equation}
\begin{equation}
\text{(resp. }U=U_{0}\oplus U_{1}\oplus \dots \oplus U_{s}\;),
\label{eq:Giso-comp}
\end{equation}
where each component $V_{i}$ (resp. $U_{j}$) is {\it modeled} on the
irreducible $G$-representation $\mathcal{V}_{i}$, $i=0,1,2,\dots ,r$, (resp. 
$\mathcal{U}_{j}$, $j=0,$ $1,$ $\dots ,$ $s$), that is, $V_{i}$ (resp. $U_{j} $) contains all the irreducible subrepresentations of $V$ (resp. $U$)
equivalent to $\mathcal{V}_{i}$ (resp. $\mathcal{U}_{j}$). We call such
decomposition  \eqref{eq:Giso} (resp.  \eqref{eq:Giso-comp}) a $G$\textit{-isotypical decomposition of }$V$ (resp. $U$).
\vs

Given an orthogonal $G$-representation $V$, denote by $\text{L}^{G}(V)$
(resp. $\text{GL}^{G}(V)$) the $\mathbb{R}$-algebra (resp. group) of all $G$-equivariant linear (resp. invertible) operators on $V$. Clearly, the
isotypical decomposition \eqref{eq:Giso} induces the following direct sum
decomposition of $\text{GL}^{G}(V):$ 
\begin{equation*}
\text{GL}^{G}(V)=\bigoplus_{i=0}^{r}\text{GL}^{G}(V_{i}),
\label{eq:GLG-decomp}
\end{equation*}
where for every isotypical component $V_{i}$
\begin{equation*}
\text{GL}^{G}(V_{i})\simeq \text{GL}(m_{i},\mathbb{F}),\quad m_{i}=\text{\rm dim\,}V_{i}/\text{\rm dim\,}\mathcal{V}_{i}
\end{equation*}
and depending on the type of the irreducible representation $\mathcal{V}_{i}, $ $\mathbb{F}$ ($=\mathbb{R}$, $\mathbb{C}$ or $\mathbb{H}$) is a
finite-dimensional division algebra $\text{L}^{G}(\mathcal{V}_{i})$.
\vs
Finally, for a linear operator $A\in \text{L}^{G}(V)$ we will denote by $\sigma(A)$ the spectrum of $A$ and for each $\mu\in \sigma(A)$ and $i=0,1,2,\dots, r$ we put 
\[
m_i(\mu):=\text{dim} \, (E(\mu)\cap V_i)/\text{dim}\, \cV_i,
\]
where $V_i$ is given by \eqref{eq:Giso} and $E(\mu)$ denotes the generalized eigenspace of $\mu$. The number $m_i(\mu)$ will be called $\cV_i$-multiplicity of $\mu$. 
\vs

\section{Euler Ring and Related Algebraic Structures}

\subsection{Euler Characteristic}

For a topological space $Y$, denote by $H_{c}^{\ast }(Y)$ the ring of
Alexander-Spanier cohomology with compact support (see \cite{Spa}). If $%
H_{c}^{\ast }(Y)$ is finitely generated then the Euler characteristic $\chi
_{c}(Y)$ is correctly defined. For a compact $CW$-complex $X$ and its closed
subcomplex $A$, as it is very well-known $H_{c}^{\ast }(X\setminus A))\cong
H^{\ast }(X,A;\mathbb{R})$, where $H^{\ast }(\cdot )$ stands for a usual
cellular cohomology ring (see \cite{Spa}, Chap. 6, Sect. 6, Lemma 11).
Therefore, $\chi _{c}(X\setminus A)$ is correctly defined and 
\begin{equation*}
\chi (X)=\chi _{c}(X\setminus A)+\chi (A)=\chi (X,A)+\chi (A).
\end{equation*}%
where $\chi (\cdot )$ stands for the Euler characteristic with respect to
the cellular cohomology groups. Moreover, if $Y$ a compact $CW$-complex, $%
B\subseteq Y$ its closed subcomplex and $p:X\setminus A\rightarrow
Y\setminus B$ a locally trivial fibre bundle with path-connected base and
fibre $F$ which is a compact manifold, then (cf. \cite{Spa}, Chap. 9, Sect.
3, Theorem 1) 
\begin{equation*}
\chi _{c}(X\setminus A)=\chi (F)\chi _{c}(Y\setminus B).
\end{equation*}%
Furthermore, for a compact $G$-ENR-space $X$, where $G=\mathbb{T}^{n}$ an $n$%
-dimensional torus ($n>0$) one shows 
\begin{equation}\label{eq:EchTn}
\chi (X)=\chi (X^{G}),
\end{equation}%
and $X^{G}=\emptyset $, so $\chi (X)=0$ (see \cite{Kawa,Kom}).\medskip

\vs
\subsection{Euler and Burnside Rings\label{subsect:Euler}}

In this section we recall some basic
properties of the  Euler and Burnside Rings (cf. \cite{tD}).
\vs 
\begin{definition}\rm 
\label{def:EulerRing} (cf. \cite{tD}) Let $U(G)={\mathbb{Z}}[\Phi
(G)]$ be a free ${\mathbb{Z}}$-module with basis $\Phi (G)$ and $\chi _{c}$ denotes the Euler characteristic defined in terms of the
Alexander-Spanier cohomology with compact support (cf. \cite{Spa}). Define the {\it 
multiplication}  on $U(G)$ as follows:  for generators $(H)$, $(K)\in
\Phi (G)$ put: 
\begin{equation}
(H)\ast (K)=\sum_{(L)\in \Phi (G)}n_{L}(L),  \label{eq:Euler-mult}
\end{equation}
where 
\begin{equation}
n_{L}:=\chi _{c}((G/H\times G/K)_{L}/N(L)),  \label{eq:Euler-coeff}
\end{equation}%
and we extend it linearly to the multiplication on entire $U(G)$. We call
the free ${\mathbb{Z}}$-module $U(G)$ equipped with multiplication (\ref{eq:Euler-mult}) the {\it Euler ring} of the group $G$ (cf. \cite{BtD}).
\end{definition}

Let $\Phi _{0}(G)=\{(H)\in \Phi (G):$ \textrm{dim\thinspace }$W(H)=0\}$ and
denote by $A(G)={\mathbb{Z}}[\Phi _{0}(G)]$ a free ${\mathbb{Z}}$-module
with basis $\Phi _{0}(G).$ Define multiplication on $A(G)$ by restricting
multiplication from $U(G)$ to $A(G)$, i.e. for $(H)$, $(K)\in \Phi _{0}(G)$
let
\begin{equation}
(H)\cdot (K)=\sum_{(L)}n_{L}(L),\qquad (H),(K),(L)\in \Phi _{0}(G),\text{
where}  \label{eq:multBurnside}
\end{equation}
\begin{equation}
n_{L}=\chi ((G/H\times G/K)_{L}/N(L))=|(G/H\times G/K)_{L}/N(L)|
\label{eq:CoeffBurnside}
\end{equation}
($\chi $ here denotes the usual Euler characteristic). Then $A(G)$ with
multiplication  \eqref{eq:multBurnside} becomes a ring which is called the {\it Burnside ring} of $G$. As it can
be shown, the coefficients \eqref{eq:CoeffBurnside}
can be found (cf. \cite{AED}) using the following recursive formula:
\begin{equation}
n_{L}=\frac{n(L,K)|W(K)|n(L,H)|W(H)|-\sum_{(\widetilde{L})>(L)}n(L,\widetilde{L})n_{\widetilde{L}}|W(\widetilde{L})|}{|W(L)|},
\label{eq:rec-coef}
\end{equation}
where $(H),$ $(K),$ $(L)$ and $(\widetilde{L})$ are taken from $\Phi _{0}(G)$.\medskip
\vs

Observe that although $A(G)$ is clearly a ${\mathbb{Z}}$-submodule of $U(G)$, in general, it may {\bf not} be a subring of $U(G)$.
\vs

Define $\pi _{0}:U(G)\rightarrow A(G)$ as follows: for $(H)\in \Phi (G)$ let 
\begin{equation}
\pi _{0}((H))=
\begin{cases}
(H) & \text{ if }\;(H)\in \Phi _{0}(G), \\ 
0 & \text{ otherwise.}
\end{cases}
\label{eq:pi_0-homomorphism}
\end{equation}
The map $\pi _{0}$ defined by (\ref{eq:pi_0-homomorphism}) is a ring
homomorphism {(cf. \cite{BKR})}, i.e. 
\begin{equation*}
\pi _{0}((H)\ast (K))=\pi _{0}((H))\cdot \pi _{0}((K)),\qquad (H),(K)\in
\Phi (G).
\end{equation*}
The following  well-known result (cf. \cite{tD}, Proposition 1.14, page
231), shows a difference between the generators $(H)$ of $U(G)$ and $A(G)$..

\begin{proposition}
\label{pro:nilp-elements} Let $(H)\in \Phi _{n}(G)$.
\begin{itemize}
\item[(i)] If  $n>0$,  then,  $(H)^{k}=0$ in $U(G)$ for some $k\in \mathbb{N}$, i.e.  $(H)$
is a nilpotent element in $U(G)$;
\item[(ii)] If $n=0$, then $(H)^k\not=0$  for all $k\in \bn$.
\end{itemize}
\end{proposition}

The multiplicative structure of $A(G)$ can be only partially used to
determine the multiplication structure of $U(G)$. In subsection \ref{sec:UGO2}, we will discuss this idea in
more details  in the case of the group $G:=\Gamma\times O(2)$ (with $\Gamma$ being a finite group). .

\subsection{Euler Ring Homomorphism}

\label{sec:Euler-hom}Let $\psi :{G^{\prime }}\rightarrow G$ be a
homomorphism of compact Lie groups. Then, the ${G^{\prime }}$ acts on the
left on $G$ by ${g^{\prime }}x:=\psi ({g^{\prime }})x$ (a similarly $x{%
g^{\prime }}:=x\psi ({g^{\prime }})$ defines the right action). In
particular, for any subgroup $H\leq G$, the $\psi $ induces a $G^{\prime }$%
-action on $G/H$ and the stabilizer of $gH\in G/H$ is given by%
\begin{equation}
G_{gH}^{\prime }=\psi ^{-1}(gHg^{-1}).  \label{eq:grad-iso-rel}
\end{equation}%
Therefore, $\psi $ induces a map $\Psi :U(G)\rightarrow U({G^{\prime }})$
defined by 
\begin{equation}
\Psi ((H))=\sum_{(\tilde{H})\in \Phi ({G^{\prime }})}\chi _{c}((G/H)_{({%
H^{\prime }})}/{G^{\prime }})({H^{\prime }}).  \label{eq:RingHomomorphism}
\end{equation}
\vs

\begin{proposition}
\label{prop:RingHomomorphism} \textrm{(cf. \cite{BKR,tD})} The map $\Psi $
defined by (\ref{eq:RingHomomorphism}) is a homomorphism of Euler rings.
\end{proposition}
\vs

Let us recall the following  fact about Euler Ring Homomorphism $\Psi
:U(G)\rightarrow U(\mathbb{T}^{n})$ where $\mathbb{T}^{n}$ is a maximal
torus in $G$ and $\psi :\mathbb{T}^{n}\rightarrow G$ is the natural
embedding. Then we have (cf. \cite{BKR}):
\vs

\begin{proposition}
\label{pro:torus-coefficient} Let $\mathbb{T}^{n}$ be a maximal torus in $G$. Then, 
\begin{equation*}
\Psi (\mathbb{T}^{n})=|W(\mathbb{T}^{n})|(\mathbb{T}^{n})+\sum_{(\mathbb{T}^{\prime })}n_{\mathbb{T}^{\prime }}(\mathbb{T}^{\prime }),
\end{equation*}%
where $\mathbb{T}^{\prime }=g\mathbb{T}^{n}g^{-1}\cap \mathbb{T}^{n}$ for
some $g\in G$ and $(\mathbb{T}^{\prime })\not=(\mathbb{T}^{n})$.
\end{proposition}

\vs
The above Euler ring homomorphism  $\Psi: U(G)\to U(\mathbb T^n)$ can sometimes be used to compute the multiplication tabled for the Euler ring $U(G)$.
\vs
\begin{example}\rm 
Let us consider the group $G=O(2)$. Then we have 
\begin{align*}
\Phi _{0}(O(2))& =\{(O(2)),(SO(2)),(D_{k});k=1,2,3,\dots \}, \\
\Phi _{1}(O(2))& =\{(\mathbb{Z}_{k}):k=1,2,3,\dots \}.
\end{align*}
The structure of the Burnside ring $A(O(2))$ is well known (see for example 
\cite{AED}), therefore we can use it to obtain the following partial
multiplication table in $U(O(2))$: 

\begin{center}
\begin{tabular}{|c||c|c|c|c|}
\hline
& $(O(2))$ & $(SO(2))$ & $(D_{n})$ & $(\mathbb{Z}_{n})$ \\ \hline\hline
$(O(2))$ & $(O(2))$ & $(SO(2))$ & $(D_{n})$ & $(\mathbb{Z}_{n})$ \\ 
$(SO(2))$ & $(SO(2))$ & $2(SO(2))$ & $a(\mathbb{Z}_{n})$ & $b(\mathbb{Z}%
_{n}) $ \\ 
$(D_{k})$ & $(D_{k})$ & $a(\mathbb{Z}_{k})$ & $2(D_{l})+c(\mathbb{Z}_{l})$ & 
$x(\mathbb{Z}_{l})$ \\ 
$(\mathbb{Z}_{k})$ & $(\mathbb{Z}_{k})$ & $b(\mathbb{Z}_{k})$ & $x(\mathbb{Z}%
_{l})$ & $c(\mathbb{Z}_{l})$ \\ \hline
\end{tabular}%
\end{center}
\noindent where $l:=gcd(n,k)$ and the coefficients $a,b,c,d,x,$ are
unknown. In order to determine these coefficients, we will use the
homomorphism $\Psi :U(O(2))\rightarrow U(S^{1})$. Notice that 
\begin{equation*}
\Psi (SO(2))=2(S^{1}),\quad \Psi (D_{n})=(\mathbb{Z}_{n}),\quad \Psi (\mathbb{Z}_{n})=2(\mathbb{Z}_{n}).
\end{equation*}
Indeed, notice for example that 
\begin{equation*}
\frac{N(\mathbb{Z}_{n},D_{n})}{D_{n}}/S^{1}=\frac{O(2)}{D_{n}}/S^{1}=pt,\quad \text{ thus }\;\chi \left( \frac{N(\mathbb{Z}_{n},D_{n})}{D_{n}}/S^{1}\right) =1.
\end{equation*}
The structure of the Euler ring $U(S^{1})$ is very simple: by \eqref{eq:EchTn} we have  $(\mathbb{Z}_{k})\ast (\mathbb{Z}_{n})=0$, $n,k\in \mathbb{N}$. Therefore, by applying
the homomorphism $\Psi $ to the partial multiplication table we obtain: 
\begin{align*}
2(\mathbb{Z}_{n})& =\Psi (SO(2))\ast \Psi (D_{n})=\Psi (a(\mathbb{Z}_{n}))=a2(\mathbb{Z}_{n}), \\
4(\mathbb{Z}_{n})& =\Psi (SO(2))\ast \Psi (\mathbb{Z}_{n})=\Psi (b(\mathbb{Z}_{n}))=b2(\mathbb{Z}_{n}), \\
0& =(\mathbb{Z}_{k})\ast (\mathbb{Z}_{n})=\Psi (D_{k})\ast \Psi (D_{n})=\Psi
(2(D_{l})+c(\mathbb{Z}_{l}))=2\mathbb{(}\mathbb{Z}_{l})+2c(\mathbb{Z}_{l}),\\
0&=2(\mathbb{Z}_{k})\ast (\mathbb{Z}_{n})=\Psi (\bz_{k})\ast \Psi (D_{n})=\Psi
(x(\mathbb{Z}_{l}))=2x(\mathbb{Z}_{l}),
\end{align*}%
and consequently we get the following multiplication table for $U(O(2))$:
\begin{center}
\begin{tabular}{|c||c|c|c|c|}
\hline
& $(O(2))$ & $(SO(2))$ & $(D_{n})$ & $(\mathbb{Z}_{n})$ \\ \hline\hline
$(O(2))$ & $(O(2))$ & $(SO(2))$ & $(D_{n})$ & $(\mathbb{Z}_{n})$ \\ 
$(SO(2))$ & $(SO(2))$ & $2(SO(2))$ & $(\mathbb{Z}_{n})$ & $2(\mathbb{Z}_{n})$
\\ 
$(D_{k})$ & $(D_{k})$ & $(\mathbb{Z}_{k})$ & $2(D_{l})-(\mathbb{Z}_{l})$ & $0$\\ 
$(\mathbb{Z}_{k})$ & $(\mathbb{Z}_{k})$ & $2(\mathbb{Z}_{k})$ & $0$ & $0$ \\ \hline
\end{tabular}
\end{center}

\end{example}
\vs
\section{Product Group $G_{1}\times G_{2}$}

In order to proceed further with computations of $U\left( G\right) ,$ where $G=\Gamma \times \mathrm{O}(2)$ and $\Gamma $ is a finite group, we need to
develop a description of subgroups its subgroups and conjugacy classes. For
completeness of this paper, following \cite{BSZ}, we will discuss a
description of subgroups of $G\times G^{\prime }$. The classical result,
Goursat's Lemma allows to characterize all subgroups $H$ of the direct
product $G\times G^{\prime }$ of groups $G$, $G^{\prime }$ in terms of the
isomorphisms between their quotients. Namely, given $H\leq G\times G^{\prime
}$ there are two subgroups $K\leq G$, $K^{\prime }\leq G^{\prime }$, a group 
$L,$ and a pair of epimorphisms $\varphi :H\rightarrow L,$ $\varphi
:H^{\prime }\rightarrow L$ such that
\begin{equation*}
H=\left\{ \left( a,b\right) \in K\times K^{\prime }\text{ }|\text{ }\varphi
\left( a\right) =\psi \left( b\right) \right\} .
\end{equation*}
\vs
\subsection{Description of Subgroups in the Product Group $G_{1}\times G_{2}$}

We start this section by recalling classical result known as Goursat's Lemma
which characterizes subgroups of a direct product of groups.
\vs
\begin{lemma}[Goursat's Lemma]
Let $G_{1},$ $G_{2}$ be groups, and let $H$ be a subgroup of $G_{1}\times
G_{2}$ such that the two projections $p_{1}:H\rightarrow G_{1}$ and $%
p_{2}:H\rightarrow G_{2}$ are surjective. Let $N_{1}\trianglelefteq G_{1}$
be the kernel of $p_{2}$ and $N_{2}\trianglelefteq G_{2}$ the kernel of $%
p_{1}$. Then the image of $H$ in $G_{1}/N_{1}\times G_{2}/N_{2}$ is the
graph of an isomorphism $G_{1}/N_{1}\simeq G_{2}/N_{2}.$
\end{lemma}
\vs
The above statement gives descriptions of subgroups $H\leq G_{1}\times G_{2}$
in terms of graphs of isomorphisms $G_{1}/N_{1}\simeq G_{2}/N_{2}.$ Since
description of subgroups above is a rather difficult to use practically, we
will use reformulation of Goursat's Lemma. Let $H\leq G_{1}\times G_{2}$ and 
$p_{i}:G_{1}\times G_{2}\rightarrow G_{i},$ $i=1,$ $2$ be natural
projections,
\begin{equation}
K_{i}=p_{i}\left( H\right) \leq G_{i},\text{ }i=1,\text{ }2.  \label{Eq_1}
\end{equation}
and
$\pi _{i} :H\rightarrow K_{i}\leq G_{i},$  given by  
$\pi _{i}\left( g\right) =p_{i}\left( g\right)$
be
the be restriction of $p_{i}$ to $H,$ $i=1,2$. Let
\begin{equation}
H_{1}=\pi _{1}(\text{\textrm{Ker\thinspace }}\pi _{2})\text{ and }H_{2}=\pi
_{2}(\text{\textrm{Ker\thinspace }}\pi _{1}).  \label{Eq_2}
\end{equation}
\vs 
\begin{theorem}
If $H\leq G_{1}\times G_{2}$ there exist a group $L$ and epimorphisms $\varphi _{i}:K_{i}\rightarrow L$, where $K_{1}$ and $K_{2}$ are given by $\left( \text{\textrm{\ref{Eq_1}}}\right) $ and $\left( \text{\textrm{\ref{Eq_2}}}\right) $ such that 
\begin{equation}
H=\{(k_{1},k_{2})\in K_{1}\times K_{2}\mid \varphi _{1}(k_{1})=\varphi
_{2}(k_{2})\}.  \label{Eq_A1}
\end{equation}
\end{theorem}
\vs 
\begin{definition}\rm
Let $L$ be a group, $K_{i}\leq G_{i}$, $\varphi _{i}:K_{i}\rightarrow L$ be
epimorphisms, $i=1,$ $2$. The subdirectpoduct of subgroups $K_{1}$ and $K_{2} $ and subgroup $L$ is defined as follows 
\begin{equation}
K_{1}{^{\varphi _{1}}\times^{\varphi _{2}}_{L}}K_{2}=\left\{
\left( k_{1},\text{ }k_{2}\right) \in K_{1}\times K_{2}\mid \varphi
_{1}(k_{1})=\varphi _{2}(k_{2})\right\}  \label{Eq_A2}
\end{equation}
\end{definition}
\vs 
As we mentioned it before, for subgroups $K_{i}\leq G_{i}$ $i=1$, $2$ and a given group $L$ the description of the group
\begin{equation*}
H=\{(k_{1},k_{2})\in K_{1}\times K_{2}\mid \varphi _{1}(k_{1})=\varphi
_{2}(k_{2})\}
\end{equation*}
in the form $K_{1}{^{\varphi _{1}}\times^{\varphi _{2}}_{L}}K_{2}$ is not unique. In fact, as we observed we have the following.
\vs 

\begin{proposition}
Let $L$ be a group, $K_{i}\leq G_{i}$ subgroups given by given by \eqref{Eq_1} and \eqref{Eq_2} and 
\begin{equation*}
\varphi _{1},\text{ }\varphi _{2}:K_{1}\rightarrow L\text{ and }\psi _{1},\text{ }\psi _{2}:K_{2}\rightarrow L
\end{equation*}
be epimorphisms. Then 
\begin{equation*}
K_{1}{^{\varphi _{1}}\times^{\psi _{1}}_L}K_{2}=K_{1}{^{\varphi _{2}}\times^{\psi _{2}}_{L}}K_{2}
\end{equation*}
if and only if there exists an automorphism $\sigma \in \mathrm{Aut}\left(L\right) ,$ such that $\varphi _{1}=\sigma \circ \varphi _{2}$ and $\psi
_{1}=\sigma \circ \psi _{2}.$
\end{proposition}
\vs 
Let $G_{1}$ and $G_{2}$ be two subgroups and consider the product group and $G_{1}\times G_{2}$. We define the following projection homomorphisms: 
\begin{align*}
\pi _{1}& :G_{1}\times G_{2}\rightarrow G_{1};\;\;\;\pi
_{1}(g_{1},g_{2})=g_{1}, \\
\pi _{2}& :G_{1}\times G_{2}\rightarrow G_{2}\;\;\;\pi
_{2}(g_{1},g_{2})=g_{2},
\end{align*}
and the embedding homomorphisms: 
\begin{align*}
i_{1}& :G_{1}\rightarrow G_{1}\times G_{2}\;\;\;i_{1}(g_{1})=(g_{1},e) \\
i_{2}& :G_{2}\rightarrow G_{1}\times G_{2}\;\;\;i_{2}(g_{2})=(e,g_{2}).
\end{align*}
Then we have the following result describing all the subgroups $H$ of the
product $G_{1}\times G_{2}$.
\vs

\begin{theorem}
\label{th:prod1}There is a bijection correspondence between subgroups $H$ of 
$G_{1}\times G_{2}$ and $5$-tuples $\left( H_{1},H_{0},K_{1},K_{0},\theta \right) $, where $H$ is a subgroup of $H_{o}$
(resp. $K$ is a subgroup of $\mathscr G_{2}$), $H_{o}$ is a normal subgroup
of $H$ (resp. $K_{o}$ is a normal subgroup of $K$) and $\theta
:H/H_{o}\rightarrow K/K_{o}$ is an isomorphism, given by the relation 
\begin{equation}
\mathscr H:=\{(h,k)\in H\times K:\theta \circ p_{1}(h)=p_{2}(k)\},
\label{eq:pr1}
\end{equation}
where $p_{1}:H\rightarrow H/H_{o}$ and $p_{2}:K\rightarrow K/K_{o}$ are
natural projections.
\end{theorem}

\vs

Suppose that $\mathscr H$ is a subgroup of $\mathscr G_1\times\mathscr G_2$.
Put 
\begin{align*}
H&:=\pi _{1}( \mathscr H), \quad H_o:=i_1^{-1}(\mathscr H) \\
K&:=\pi_2(\mathscr H), \quad K_o:=i_2^{-1}(\mathscr H).
\end{align*}
Clearly, $H_o$ is normal in $H$. Indeed, Notice that $H_o$ is normal in $H$.
Indeed, let $h^{\prime }\in H_o$ and $h\in H$, thus by the definition $(
h^{\prime },e) \in \mathscr H$ and $( h,k) \in \mathscr H$ for some $k\in \mathscr G_1$. Thus $(h,k) ( h^{\prime },e) ( h,k) ^{-1}\in \mathscr H$ and
since 
\begin{equation*}
( h,k) ( h^{\prime },e) ( h,k) ^{-1}=( h,k) \left( h^{\prime },e\right) (
h^{-1},k^{-1}) =( hh^{\prime }h^{-1},kek^{-1}) =( hh^{\prime }h^{-1},e) \in \mathscr H,
\end{equation*}
which implies $( h,k) ( h^{\prime },e) ( h,k) ^{-1}\in H_o$. Similarly, one
can show that $K_o$ is normal in $K$. The situation is illustrated on Figure \ref{fig:productG}. 

\begin{figure}[tbp]
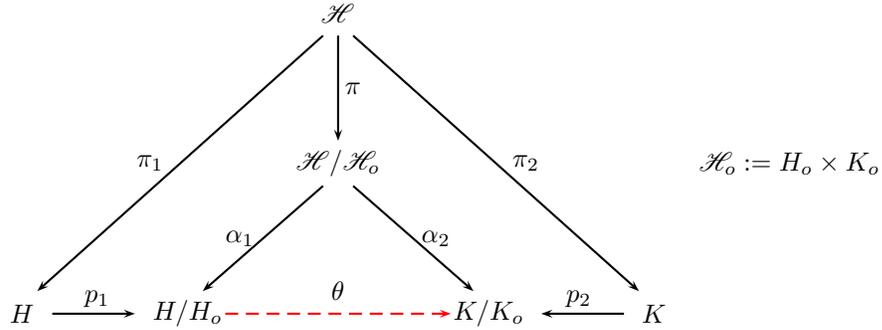

\label{fig:productG}
 \vskip4.5cm 
 \rput(-2,0){$H/H_o$} \rput(2,0){$K/K_o$} 
 \rput(0,2){$\mathscr H/\mathscr H_o$} 
 \rput(0,4){$\mathscr H$} 
 \rput(-4.2,0){$H$} 
 \rput(4.2,0){$K$} 
 \psline{->}(0,3.7)(0,2.3) 
 \psline{->}(-.2,1.7)(-1.8,.3) 
 \psline{->}(.2,1.7)(1.8,.3) 
 \psline{->}(-.2,3.7)(-4,.3) 
\psline{->}(.2,3.7)(4,.3) 
\psline{->}(3.8,0)(2.7,0) 
\psline{->}(-3.8,0)(-2.7,0) 
\rput(2.5,2){$\pi_2$} 
\rput(-2.5,2){$\pi_1$} 
\rput(-1.3,1){$\alpha_1$} 
\rput(1.3,1){$\alpha_2$} 
\rput(0.2,3){$\pi$} 
\rput(3.2,.2){$p_2$} 
\rput(-3.2,.2){$p_1$} 
\rput(6,2){$\mathscr H_o:=H_o\times K_o$} 
\psline[linecolor=red,linestyle=dashed]{->}(-1.5,0)(1.5,0) \rput(0,.3){$\theta$} 
\vskip.3cm
\caption{Commutative diagram used for the construction of the isomorphism $\protect\theta: H/H_o\to K/K_o$.}
\end{figure}

Since $H_o \trianglelefteq H$ and $K_o \trianglelefteq K$, we can consider
the quotient subgroups $H/H_o$ and $K/K_o$. We claim that there exists a
natural isomorphism $\theta : H_o \trianglelefteq H\to K_o \trianglelefteq K 
$. Indeed, For $\overline h:=H_oh\in H/H_o$ choose $k\in K$ such that $(h,k)\in \mathscr H$ and put $\theta(\overline h)=\overline k$, We need to
show that this definition doesn't depend on the choice of $k$ and that $\theta$ is an isomorphism. Indeed, put 
\begin{equation*}
\overline{h}=H_oh,\quad \overline{k}=K_ok,\quad \overline{h^{\prime } }=H_oh^{\prime },\quad \overline{k^{\prime }}=K_ok^{\prime },
\end{equation*}
then if $(h,k)\in \mathscr H$ and $(h,k^{\prime })\in \mathscr H$, this $\widetilde k={k^{\prime }}^{-1}k\in K_o$, which implies that the above
definition doesn't depend on the choice of $k$. Similarly, if $\overline h=\overline{h^{\prime }}$, then $h^{\prime }=\widetilde h h$ for some $%
\widetilde h\in H_o$ and since $(\widetilde h,e)\in \mathscr H$ and $(h,k)\in \mathscr H$ implies $(h^{\prime },k)\in \mathscr H$, it follows
that $\theta(\overline h)=\theta(\overline{h^{\prime }}$. It can be easily
verified that $\theta$ is a homomorphism. Suppose now that $\theta(\overline h)=K_o$, i.e. there is $\widetilde k\in K_o$ such that $(h,\widetilde k)\in \mathscr H$, then since $(h,e)=(h,\widetilde k)(e,{\widetilde k}^{-1})\in \mathscr H$, it follows that $\pi(h,\widetilde k)\in H_o$, which implies $\theta(\overline h)=H_o$, and consequently $\theta$ is injective. On the
other hand, for for every $\overline k\in K/K_o$, there is $h\in \mathscr G_1 $ such that $(h,k)\in \mathscr H$, it is clear that $\theta(\overline h)=\overline k$, which concludes the argument that $\theta$ is an isomorphism.
\vs
 It is also clear from the construction that the 5-tuples $\left(H.H_o,K,K_o, \theta\right)$ is uniquely determined and $\mathscr H:=\{(h,k)\in H\times K: \theta\circ p_1(h)=p_2(k)\}$.
\vs

\begin{corollary}
There is a bijection correspondence between subgroups $\mathscr H$ of $\mathscr G_{1}\times \mathscr G_{2}$ and 4-tuples $\left( H,K,K_{o},\varphi
\right) $, where $H$is a subgroup of $\mathscr H_{o}$ (resp. $K$ is a
subgroup of $\mathscr G_{2}$), $H_{o}$ is a normal subgroup of $H$ and $\varphi :H\rightarrow K/K_{o}$ is an epimorphism, given by the relation 
\begin{equation*}
\mathscr H=\{(h,k)\in H\times K:\varphi (h)=p_{2}(k)\},
\end{equation*}
where $p_{2}:K\rightarrow K/K_{o}$ is a natural projection.
\end{corollary}
\vs

It is clear that one can assign to a 5-tuples $\left( H,H_o,K,K_o,\theta\right)$ (described in Theorem \ref{th:prod1}) the 4-tuples $\left(H,K,K_o,\varphi\right)$, where $\varphi:=\theta\circ \pi_1$. Conversely, for
an arbitrary$\left( H,K,K_o,\varphi\right)$ where $H$ is a subgroup of $\mathscr H_o$ (resp. $K$ is a subgroup of $\mathscr G_2$), $H_o$ is a normal
subgroup of $H$ and $\varphi :H \to K/K_o$ is an epimorphism, we have the
following subgroup of $\mathscr G_1\times \mathscr G_2$ 
\begin{equation*}
\mathscr H:=\{(h,k)\in H\times K: \varphi(h)= p_2(k)\},
\end{equation*}
Notice that $\pi_1(\mathscr H)=H$ and $H_o:=i_1(H)=\text{\textrm{Ker\,}}
\varphi$ (since $i_1(h)=p_2(e)=e$), and clearly, according to the definition
of $\theta$, we have by isomorphism theorem 
\begin{equation*}
\theta(\overline h)=\overline k, \;\; (h,k)\in \mathscr H\;\;
\Leftrightarrow\;\; \overline h=\text{\textrm{Ker\,}}\varphi\, h \;\text{
and }\; \varphi(h)=\overline k.
\end{equation*}
where $\overline h:=H_o h$ and $\overline k:=K_ok$.
\vs 
\begin{lemma}
\label{lem:product2} Let $\left( H,H_o,K,K_o, \theta\right)$ be 5-tuples
such that $H$is a subgroup of $\mathscr H_o$ (resp. $K$ is a subgroup of $\mathscr G_2$), $H_o$ is a normal subgroup of $H$ (resp. $K_o$ is a normal
subgroup of $K$) and $\theta :H/H_o \to K/K_o$ is an isomorphism. Put $\mathscr H_o=H_o\times K_o$ and denote by $\pi:\mathscr H\to \mathscr H/\mathscr H_o$ the natural projection. Then there exist natural epimorphisms $\alpha_1:\mathscr H/\mathscr H_o\to H/H_o$ and $\alpha_2:\mathscr H/\mathscr H_o\to K/K_o$ such that for all $(h,k)\in \mathscr H$ 
\begin{align*}
p_1\circ \pi_1(h,k)=\alpha_1\circ \pi(h,k),\quad \text{ and }\quad p_2\circ
\pi_1(h,k)=\alpha_2\circ \pi(h,k).
\end{align*}
(See Figure \ref{fig:productG})
\end{lemma}
\vs

One can easily notice from the following diagram 

\vglue3cm 
\hskip4cm 
\rput(0,2.5){$\mathscr H$} 
\rput(5,2.5){$H$} 
\rput(8,2.5){$H/H_o$} 
\rput(0,0){$\mathscr H/\mathscr H_o$} 
\psline{->}(.5,2.5)(4.7,2.5)
\rput(2.5,2.8){$\pi_1$}
\psline{->}(5.3,2.5)(7.5,2.5)
\rput(6.3,2.8){$p_1$} 
\psline[linecolor=red,linestyle=dashed]{->}(.6,0.21)(7.8,2.3)
\rput(4,1.5){\color{red}$\alpha_1$} 
\psline{->}(0,2.2)(0,.5)
\rput(-.3,1.5){$\pi$}
 \vskip.4cm 
 
 \noindent that $\text{\textrm{Ker\,}} (p_1\circ \pi_1)=\mathscr H_o$,
thus it follows from the isomorphism theorem that there exists a unique $\alpha_1:\mathscr H/\mathscr H_o\to H/H_o$ such that $p_1\circ
\pi_1=\alpha_1\circ \pi$. Similarly, one shows the existence of $\alpha_2$.

\vs
 It will be convenient to characterize the subgroups $\mathscr H$
of the product $\mathscr G_1\times \mathscr G_2$ using the following result:
\vs

\begin{theorem}
\label{th:product-1} For every subgroup $\mathscr H$ of the product group $\mathscr G_1\times \mathscr G_2$ there exist a group $L$ and two
epimorphisms $\varphi :\pi_1(\mathscr H)\rightarrow L$ and $\psi :\pi_2(\mathscr H)\rightarrow L$, such that 
\begin{equation}  \label{eq:product-rep}
\mathscr H=\{(h,k)\in H\times K: \varphi(h)=\psi(k)\}, \quad H:=\pi_1(\mathscr H),\; K:=\pi_2(\mathscr H).
\end{equation}
\end{theorem}

\vs

Let $\mathscr H$ be a subgroup of $\mathscr G_1\times \mathscr G_2$ and
suppose that the 5-tuple $(H,H_o,K,K_o,\theta)$ is associated with $\mathscr H$ (see Theorem \ref{th:prod1}). Then put $L=\overline{K}/K_o$, $%
\varphi:=\theta\circ p_1$ and $\psi:=p_2$. Then, by \eqref{eq:pr1}, we have 
\begin{equation*}
\mathscr H=\mathscr H:=\{(h,k)\in H\times K: \theta\circ p_1(h)=
p_2(k)\}=\{(h,k)\in H\times K: \varphi(h)=\psi(k)\}.
\end{equation*}

\vs

\begin{definition}\rm 
Let $H$ be a subgroup of $\mathscr G_1$ (resp. $K$ be a subgroup of $\mathscr G_2$) and $L$ a given group. Assume also that $\varphi:H\to L$ and $\psi:K\to L$ are two given homomorphisms. Then the subgroup of $\mathscr G_1\times \mathscr G_2$, given by 
\begin{equation}  \label{eq:amalg1}
H{^\varphi\times _L^\psi} K:=\{(h,k)\in H\times K: \varphi(h)=\psi(k)\}.
\end{equation}
will be called an {\it $(\varphi,\psi)$-amalgamated subgroup} for $L$   (or just simply an {\it amalgamated subgroup}). 
\end{definition}

\vs

\begin{remark}\rm 
Notice that by Theorem \ref{th:product-1}, every subgroup $\mathscr H$ of $\mathscr G_1\times \mathscr G_2$ is amalgamated subgroup $H{^\varphi\times
_L^\psi} K$ for certain $L$ and $(\varphi,\psi)$. Clearly, this
representation $H{^\varphi\times _L^\psi} K$  is not unique.
\end{remark}

\vs
The following result allows us to distinguish between different
amalgamated subgroups. 
\vs

\begin{proposition}
\label{pro_Equal_amal}
 Let $L$ be a group, $H\subset \mathscr G_1$ and $K\subset \mathscr G_2$ two subgroups and 
\begin{equation*}
\varphi,\,\varphi^{\prime }:H\to L,\quad \psi,\, \psi^{\prime }:K\to L
\end{equation*}
be epimorphisms. Then 
\begin{equation*}
H{^\varphi\times _{L}^\psi}K=H{^{\varphi^{\prime }}\times_L^{\psi^{\prime }}}K
\end{equation*}
if and only if there exists an automorphism $\gamma \in\text{\textrm{Aut}}(
L) ,$ such that $\varphi =\gamma \circ \varphi^{\prime }$ and $\psi =\gamma
\circ \psi ^{\prime }$.
\end{proposition}

\begin{proof}
Clearly, if  $\gamma \in \text{Aut}( L) $ is such that $\vp =\gamma \circ
\vp'$ and $\psi =\gamma \circ \psi '$, then it
\begin{align*}
H{^\vp\times _{L}^\psi}K&=\{(h,k)\in H\times K: \vp(h)=\psi(k)\}=\{(h,k)\in H\times K: \gamma\circ\vp'(h)=\gamma\circ\psi'(k)\}\\
&=\{(h,k)\in H\times K: \vp'(h)=\psi'(k)\}=H{^{\vp'}\times_L^{\psi'}}K.
\end{align*}
Suppose therefore, 
\begin{equation*}
H{^\vp\times _{L}^\psi}K=H{^{\vp'}\times_L^{\psi'}}K\;\;\text{ i.e. }\;\; \{(h,k)\in H\times K: \vp(h)=\psi(k)\}=\{(h,k)\in H\times K: \vp'(h)=\psi'(k)\}.
\end{equation*}
Since $\vp':H\to  L$ is surjective, for
every $l\in L$ there is $h\in H$ such that $\vp'(h)=l$. Then we define $\gamma \in \text{Aut}\left( L\right) $ by
$\gamma (l)=\varphi (h)$.
Clearly, $\gamma$ doesn't depend on the choice of $h$, and  we have
\begin{align*}
\vp(h) &=\gamma(\vp'(h))=( \gamma\circ \vp')(h)) , \\
\psi (k) &=\vp(h)=(\gamma\circ \vp')(h)
= (\gamma (\vp'(h))=(\gamma (\psi'(k))=(\gamma\circ \psi')(k).
\end{align*}
Therefore, we have
\begin{equation*}
\varphi =\gamma \circ \varphi ^{\prime }\text{ and }\psi =\gamma \circ \psi
^{\prime }
\end{equation*}
\end{proof}
\vs
We have also the following result:
\vs

\begin{proposition}
\label{prop_Numb_Elem_Amalgamated} Let $L$ be a group, $H\subset \mathscr G_1 $ and $K\subset \mathscr G_2$ two subgroups and $\varphi:H\to L$, $\psi:K\to L$ two epimorphisms. Put $H_o:=\text{\textrm{Ker\,}} \varphi$ and $K_o:=\text{\textrm{Ker\,}} \psi$ Then we have 
\begin{equation*}
\left| H{^\varphi\times _{L}^\psi}K \right|=|H_o||L||K_o|.
\end{equation*}
\end{proposition}

\begin{proof}
 By the definition
\begin{equation*}
H{^\vp\times _{L}^\psi}K=\{(h,k)\in H\times K: \vp(h)=\psi(k)\}=\bigcup_{l\in L}\mathcal N_l,
\end{equation*}
where $\mathcal N_l:=\{(h,k)\in H\times K: \vp(h)=l=\psi(k)\}$. Observe that for all $l,$ $l'\in L,$ if $l\neq l'$, we have $\mathcal N_l\cap \mathcal N_{l'}=\emptyset$.
Moreover, $\vp(h')=\vp(x)=l$
and $\psi(k')=\psi(k) =l$ if and only if $h^{-1}h'\in H_o$ and $k^{-1}k\in K_o$.
Therefore, for every $l\in L$, $\left|\mathcal N_l\right|=|H_o||K_o|$.,
which completes the proof.
\end{proof} 
\vs

\subsection{Conjugacy Classes of Subgroups in the Product $\mathscr G_1\times \mathscr G_2$}

Assume that $L$ is a fixed group and consider two subgroups $H\subset \mathscr G_1$ and $K\subset \mathscr G_2$. We are interested in describing
all possible different amalgamated subgroup $H{^\varphi\times _{L}^\psi}K$.
One can easily notice that every epimorphism $\varphi: H\to L$ is uniquely
determined by the kernel $H_o:=\text{\textrm{Ker\,}} \varphi$ and an
isomorphism $\theta: H/H_o\to L$. Since we can always identify $L$ with $K/K_o$, we can always take as $\psi$ the natural projection $p_2:K\to K/K_o$. On the other hand, we have that all possible epimorphisms $\varphi:H\to L$
are determined by normal subgroups $H_o$ of $H$ such that $\eta:H/H_o\simeq
L $ and automorphisms $\gamma : L\to L$ such that $\varphi=\eta\circ
\overline p_1$, here $\overline p_1:=\eta\circ \pi_1$, i.e. the following
diagram commutes:

 \vskip3cm
 \hskip6cm \rput(0,3){$H$} \rput(0.5,0){$\eta:H/H_o\simeq L$}
 \rput(4,0){$L$} 
 \psline{->}(0,2.7)(0,.3) 
 \psline{->}(1.7,0)(3.7,0) 
 \psline{->}(0.3,2.7)(3.8,.3) 
 \rput(2.5,.2){$\theta$} 
 \rput(-.3,1.5){$p_1 $} 
 \rput(2.2,1.8){$\varphi$} 
 \vskip.5cm

Therefore, by Proposition \ref{pro_Equal_amal} that there are exactly $|\mathcal{G}_{L}|,$ where $\mathcal{G}_{L}=\text{Aut}\left( L\right) $
different amalgamated subgroup $H{^\varphi\times _{L}^\psi}K$ for fixed $H$
and $K$.
\vs
Consider next for a fixed group $L$ and two subgroups $H\subset \mathscr G_1$ and $K\subset \mathscr G_2$ and their fixed normal divisors $H_o$ and $K_o$ satisfying $H/H_o\simeq L\simeq K/K_o$. We are interested to
find all different conjugacy classes of such amalgamated subgroup $H{^\varphi\times _{L}^\psi}K$ with $\text{\textrm{Ker\,}} \varphi=H_o$ and $\text{\textrm{Ker\,}} \psi = K_o$. \vskip.3cm We need the following lemma
\vs
\begin{lemma}
\label{lem_amal-con} Let $L$ be a fixed group, $H\subset \mathscr G_1$, $K\subset \mathscr G_2$ two subgroups and $H_o$ and $K_o$ their normal
divisors satisfying 
\begin{equation*}
H/H_o\simeq L\simeq K/K_o.
\end{equation*}
Then for two pairs of epimorphisms $(\varphi ,\psi )$, $(\varphi ^{\prime
},\psi ^{\prime })$, where 
\begin{equation*}
\varphi,\,\varphi ^{\prime }:H\rightarrow L,\;\; \text{\textrm{Ker\,}}
\varphi =\text{\textrm{Ker\,}} \varphi ^{\prime}=H_o,
\end{equation*}
and 
\begin{equation*}
\psi , \psi ^{\prime }:K\rightarrow L, \;\; \text{\textrm{Ker\,}} \psi =\text{\textrm{Ker\,}} \psi ^{\prime }=K_o,
\end{equation*}
the subgroups $H{^\varphi\times _{L}^\psi }K$ and $H{^{\varphi^{\prime
}}\times _{L}^{\psi ^{\prime }}} K$ are conjugate if and only if there
exists 
\begin{equation*}
(h,k)\in N_{\mathscr G_1}(H)\times N_{\mathscr G_2}(K)
\end{equation*}
satisfying for all $(h^{\prime },k^{\prime })\in H\times K$ 
\begin{equation*}
\varphi (h^{\prime })=\varphi ^{\prime }(hh^{\prime }h^{-1})\;\;\text{ and }\;\;\psi (k^{\prime })=\psi ^{\prime }(kk^{\prime }k^{-1}).
\end{equation*}
\end{lemma}

\begin{proof}
 Suppose that there is $\left( a,b\right) \in \mathscr G_1\times \mathscr G_2$
such that
\begin{equation*}
\left( a,b\right) H{^{\vp}\times _{L}^{\psi}}
K\left( a^{-1},b^{-1}\right) =H{^{\vp'}\times _{L}^{\psi ^{\prime }}}
K
\end{equation*}
Therefore, we have
\begin{eqnarray*}
&&\left( a,b\right) \left\{ \left( x,y\right) \in H\times K:\varphi \left( x\right) =\psi \left(y\right) \right\} \left( a^{-1},b^{-1}\right)  \\
&=&\left\{ \left( a,b\right) \left( x,y\right) \left( a^{-1},b^{-1}\right)\in H\times K\varphi \left(x\right) =\psi \left( y\right) \right\}  \\
&=&\left\{ \left( axa^{-1},byb^{-1}\right) \in H\times K : \varphi \left( x\right) =\psi \left(y\right) \right\}  \\
&=&\left\{ \left( x,y\right) \in H\times K : \varphi ^{\prime }\left( x\right) =\psi ^{\prime }\left(y\right) \right\} =H{^{\vp'}\times _{L}^{\psi^{\prime }}}K
\end{eqnarray*}
Therefore, for all $\left( x,y\right) \in H\times K$, such that $\varphi \left( x\right) =\psi \left( y\right) ,$ we have   $\left( axa^{-1},byb^{-1}\right) \in H{^{\vp'}\times _{L}^{\psi ^{\prime }}}K$, which implies 
\begin{equation*}
\varphi ^{\prime }\left( axa^{-1}\right) =\psi ^{\prime }\left(byb^{-1}\right) 
\end{equation*}
In particular, for all $x\in H$ and $y\in 
K$ we have
\begin{equation*}
axa^{-1}\in H\text{ and }byb^{-1}\in K
\end{equation*}
which implies that
\begin{equation*}
a\in N_{\mathscr G_1}\left( H\right) \text{ and }b\in N_{\mathscr G_2}\left( K\right) .
\end{equation*}
On the other hand suppose that there are $a\in N_{\mathscr G_1}\left( 
H\right) $ and $b\in N_{\mathscr G_2}\left( K\right) $ such that for all $ (x,y)\in H\times K$ we have
\begin{equation*}
\varphi(x)=\varphi ^{\prime }(axa^{-1})\;\;\text{ and }\;\;\psi (y)=\psi ^{\prime}(byb^{-1})
\end{equation*}
Then, 
\begin{eqnarray*}
&&\left( a,b\right) \left\{ \left( x,y\right) \in H\times K : \varphi \left( x\right) =\psi \left(y\right) \right\} \left( a^{-1},b^{-1}\right)  \\
&=&\left\{ \left( a,b\right) \left( x,y\right) \left( a^{-1},b^{-1}\right)\in H\times K : \varphi \left(x\right) =\psi \left( y\right) \right\}  \\
&=&\left\{ \left( axa^{-1},byb^{-1}\right) \in H\times K :\varphi \left( x\right) =\psi \left(y\right) \right\}  \\
&=&\left\{ \left( axa^{-1},byb^{-1}\right) \in H\times K : \varphi ^{\prime }(axa^{-1})=\psi ^{\prime}(byb^{-1})\right\}  \\
&=&\left\{ \left( x,y\right) \in H\times K :\varphi ^{\prime }\left( x\right) =\psi ^{\prime }\left(y\right) \right\} =H{^{\vp'}\times _{L}^{\psi ^{\prime }}}K
\end{eqnarray*}
It follows that
\begin{equation*}
\left( a,b\right) H{^\vp\times _{L}^\psi}K\left( a^{-1},b^{-1}\right) =H{^{\vp'}\times _{L}^{\psi ^{\prime }}}K
\end{equation*}
as we claimed.
 
\end{proof} 
\vs

Notice that for fixed subgroup $H\subset \mathscr G_1$, and a normal
subgroup $H_o$ such that $H/H_o=L$ we have an action of 
\begin{equation*}
\mathcal{G}_{1}=N(H_o)\cap N(H)
\end{equation*}
on the set \textrm{Epi}$_{H_o}(H,L)$ of all epimorphisms $\varphi $ from $H$
to $L$ with $\text{\textrm{Ker\,}} \varphi=H_o$, given by, for $h\in
N(H_o)\cap N(H)=\mathcal{G}_{1} $ 
\begin{equation*}
(h\cdot \varphi )(h^{\prime })=\varphi (hh^{\prime }h^{-1})\text{ for }\;h^{\prime }\in H.
\end{equation*}
Of course, similar properties can be also established for $K$ and $K_o$,
with 
\begin{equation*}
\mathcal{G}_{2}=N(K_o)\cap N(K)
\end{equation*}
acting on the set \textrm{Epi}$_{K_o}(K,L)$. Then the group 
\begin{equation*}
\mathcal{G}=\mathcal{G}_{1}\times \mathcal{G}_{2}\times \mathcal{G}_{L}
\end{equation*}
acts on the set \textrm{Epi}$_{H_o}(H,L)\times $\textrm{Epi} $_{K_o}(K,L)$
by the formula 
\begin{equation*}
\forall_{(h,k,\gamma )\in \mathcal{G}}\;\;\; \;(h, k, \gamma )\cdot
(\varphi,\psi )=(\gamma \circ (h\cdot \varphi ),\gamma \circ (k\cdot \psi ).
\end{equation*}
Then there are exactly 
\begin{equation*}
\left\vert \mathrm{Epi}_{H_o}(H,L)\times \mathrm{Epi}_{K_o}(K,L)/\mathcal{G}\right\vert
\end{equation*}
different conjugacy classes among such amalgamated subgroups. \vskip.3cm In
the case, we have 
\begin{equation*}
\left\vert \mathrm{Epi}_{H_o}(H,L)\right\vert =|\mathcal{G}_{L}|\text{ or }\left\vert \mathrm{Epi}_{K_o}(K,L)\right\vert =|\mathcal{G}_{L}|
\end{equation*}
this formula can be written in a more transparent form: 
\begin{equation*}
\frac{\left\vert \mathrm{Epi}_{H_o}(H,L)/\mathcal{G}_{1}\right\vert \cdot
\left\vert \mathrm{Epi}_{K_o}(K,L)/\mathcal{G}_{2}\right\vert }{|\mathcal{G}_{L}|}.
\end{equation*}
\vs

\begin{theorem}
\label{th:conj} Let $L$ be a fixed group, $H,\, H^{\prime }\subset \mathscr G_1$, $K,\, K^{\prime }\subset \mathscr G_2$ be subgroups and assume that $%
\varphi: H\to L$, $\varphi^{\prime }:H^{\prime }\to L$ and $\psi:H\to L$, $\psi^{\prime }:H^{\prime }\to L$ are epimorphisms. Then two amalgamated
subgroups $H{^\varphi \times _L^\psi}K$ and $H^{\prime }{^{\varphi^{\prime
}} \times _L^{\psi^{\prime }}}K^{\prime }$ are conjugate if and only if
there exist $(a,b)\in \mathscr G_1\times \mathscr G_2$ and $\alpha\in \text{\rm Aut\,}(L)$ such that $H'=aHa^{ -1}$, $K'=bKb^{ -1}$ and for all $h\in H$, $k\in K$ we have 
\begin{equation}  \label{eq:con2}
\alpha(\varphi'(a^{-1}ha))=\varphi(h)\quad \text{ and } \quad \alpha(\psi'(b^{
-1}kb))=\psi(k).
\end{equation}
\end{theorem}

\begin{proof} Assume that the condition \eqref{eq:con2} is satisfied then we have
\begin{align*}
(a,b)H{^{\vp} \times _L^{\psi}}K(a^{-1},b^{-1})&= (a,b)\{(h,k)\in H\times K: \vp(h)=\psi(k)\}(a^{-1},b^{-1})\\
&=\{(aha^{-1},bkb^{-1})\in aHa^{-1}\times bKb^{-1}: \vp(h)=\psi(k)\}\\
&=\{(h',k')\in H'\times K': \vp(a^{-1}h'a)=\psi(b^{-1}k'b)\}\\
&=\{(h',k')\in H'\times K': \alpha\vp'(h')=\alpha\psi'(k'))\}\\
&=\{(h',k')\in H'\times K': \vp'(h')=\psi'(k'))\}\\
&={H'}{^{\vp'} \times _L^{\psi'}}K'.
\end{align*}
Conversely, assume that $(a,b)H{^{\vp} \times _L^{\psi}}K(a^{-1},b^{-1})=H'{^{\vp' }\times _L^{\psi'}}K'$, then clearly  $H'=aHa^{-1}$, $K'=bKb^{-1}$ and 
\[
\{(aha^{-1},bkb^{-1})\in aHa^{-1}\times bKb^{-1}: \vp(h)=\psi(k)\}=\{(h',k')\in H'\times K': \vp'(h')=\psi'(k'))\},
\]
which implies that for $(aha^{-1},bkb^{-1})=(h,k)$ we have that there exists $\alpha\in \text{\rm Aut\,}(L)$  such that $\alpha\vp'(a^{-1}ha)=\vp(h)$ and $\alpha\psi'(b^{-1}kb)=\psi(k)$.
\end{proof}
\vs

\begin{remark}\rm 
Let $L$ be a fixed group, $H\le \mathscr G_1$, $K\le \mathscr G_2$ be subgroups. Assume that $\varphi:H\to L$ and $\psi:H\to L$ are epimorphisms. Then by Theorem \ref{th:conj}, all the conjugated to  $H{^\varphi \times _L^\psi}K$ subgroups are the amalgamated subgroups $H^{\prime }{^{\varphi^{\prime }} \times _L^{\psi^{\prime }}}K^{\prime }$ such that are there exist inner automorphisms $\mu_a: \mathscr G_1\to \mathscr G_1$ and $\nu_b:\mathscr G_2\to \mathscr G_2$ satisfying
\begin{equation*}
\mu_a(g_1)=a^{-1}g_1a, \quad \nu_b(g_2)=b^{-1}g_2b,\quad g_1\in \mathscr G_1, \; g_2\in \mathscr G_2,
\end{equation*}
for some $(a,b)\in \mathscr G_1\times \mathscr G_2$ and  the following
diagrams commute 
\vglue3.3cm\hskip1.5cm 
\rput(0,3){$H^{\prime }$} 
\rput(4,3){$H$} 
\rput(2,0){$L$} 
\psline{->}(0.5,3)(3.5,3) 
\psline{->}(0,2.7)(1.8,.3) 
\psline{->}(4,2.7)(2.2,0.3) 
\rput(2,3.3){$\mu_a$} 
\rput(1.3,1.5){$\varphi^{\prime }$} 
\rput(2.7,1.5){$\varphi$} 
\rput(7,0) {\rput(0,3){$K^{\prime }$} \rput(4,3){$K$} \rput(2,0){$L$} \psline{->}(0.5,3)(3.5,3) \psline{->}(0,2.7)(1.8,.3) \psline{->}(4,2.7)(2.2,0.3) \rput(2,3.3){$\nu_b$} \rput(1.3,1.5){$\psi^{\prime }$} \rput(2.7,1.5){$\psi$} } 
\vskip.5cm 
\noi In particular if $H=H^{\prime }$ and $K=K^{\prime }$, then $(a,b)\in
N_{\mathscr G_1}(H)\times N_{\mathscr G_2}(K)$. 
\end{remark}

\vs

\begin{table}[tbp]
\label{tab: Tab1-s4} 
\scalebox{.93} {\setlength\extrarowheight{5pt}\begin{tabular}{|rcc|rcc|rcc|}
    \hline
    ID & $(\mathcal H)$ & $\left\lvert W(\mathcal H)\right\rvert$ &ID & $(\mathcal H)$ & $\left\lvert W(\mathcal H)\right\rvert$ &ID & $(\mathcal H)$ & $\left\lvert W(\mathcal H)\right\rvert$ \\
    \hline
    1 & $(\amal{\mathbb{Z}_1}{\mathbb{Z}_{n}}{\mathbb{Z}_{1}}{\mathbb{Z}_1}{})$ & $\infty$&  35 & $(\amal{D_4}{D_{n}}{D_{1}}{\mathbb{Z}_4}{})$ & $4$ &    69 & $(\amal{\mathbb{Z}_2}{SO(2)}{\mathbb{Z}_{1}}{\mathbb{Z}_2}{})$ & $8$ \\
    2 & $(\amal{\mathbb{Z}_2}{\mathbb{Z}_{n}}{\mathbb{Z}_{1}}{\mathbb{Z}_2}{})$ & $\infty$&     36 & $(\amal{S_4}{D_{n}}{D_{1}}{A_4}{})$ & $4$ &    70 & $(\amal{D_1}{SO(2)}{\mathbb{Z}_{1}}{D_1}{})$ & $4$ \\
    3 & $(\amal{D_1}{\mathbb{Z}_{n}}{\mathbb{Z}_{1}}{D_1}{})$ & $\infty$ &  37 & $(\amal{V_4}{D_{2n}}{D_{2}}{\mathbb{Z}_1}{})$ & $8$ & 71 & $(\amal{\mathbb{Z}_3}{SO(2)}{\mathbb{Z}_{1}}{\mathbb{Z}_3}{})$ & $4$ \\
    4 & $(\amal{\mathbb{Z}_3}{\mathbb{Z}_{n}}{\mathbb{Z}_{1}}{\mathbb{Z}_3}{})$ & $\infty$&     38 & $(\amal{D_2}{D_{2n}}{D_{2}}{\mathbb{Z}_1}{\mathbb{Z}_2})$ & $8$ &     72 & $(\amal{V_4}{SO(2)}{\mathbb{Z}_{1}}{V_4}{})$ & $12$ \\
    5 & $(\amal{V_4}{\mathbb{Z}_{n}}{\mathbb{Z}_{1}}{V_4}{})$ & $\infty$&  39 & $(\amal{D_2}{D_{2n}}{D_{2}}{\mathbb{Z}_1}{D_1})$ & $4$ &  73 & $(\amal{D_2}{SO(2)}{\mathbb{Z}_{1}}{D_2}{})$ & $4$ \\
    6 & $(\amal{D_2}{\mathbb{Z}_{n}}{\mathbb{Z}_{1}}{D_2}{})$ & $\infty$&    40 & $(\amal{D_4}{D_{2n}}{D_{2}}{\mathbb{Z}_2}{V_4})$ & $4$ &  74 & $(\amal{\mathbb{Z}_4}{SO(2)}{\mathbb{Z}_{1}}{\mathbb{Z}_4}{})$ & $4$ \\
    7 & $(\amal{\mathbb{Z}_4}{\mathbb{Z}_{n}}{\mathbb{Z}_{1}}{\mathbb{Z}_4}{})$ & $\infty$ & 41 & $(\amal{D_4}{D_{2n}}{D_{2}}{\mathbb{Z}_2}{D_2})$ & $4$ &    75 & $(\amal{D_3}{SO(2)}{\mathbb{Z}_{1}}{D_3}{})$ & $2$ \\
    8 & $(\amal{D_3}{\mathbb{Z}_{n}}{\mathbb{Z}_{1}}{D_3}{})$ & $\infty$&  42 & $(\amal{D_4}{D_{2n}}{D_{2}}{\mathbb{Z}_2}{\mathbb{Z}_4})$ & $4$ &    76 & $(\amal{D_4}{SO(2)}{\mathbb{Z}_{1}}{D_4}{})$ & $2$ \\
    9 & $(\amal{D_4}{\mathbb{Z}_{n}}{\mathbb{Z}_{1}}{D_4}{})$ & $\infty$ & 43 & $(\amal{D_3}{D_{3n}}{D_{3}}{\mathbb{Z}_1}{})$ & $2$ &    77 & $(\amal{A_4}{SO(2)}{\mathbb{Z}_{1}}{A_4}{})$ & $4$ \\
    10 & $(\amal{A_4}{\mathbb{Z}_{n}}{\mathbb{Z}_{1}}{A_4}{})$ & $\infty$ & 44 & $(\amal{S_4}{D_{3n}}{D_{3}}{V_4}{})$ & $2$ &    78 & $(\amal{S_4}{SO(2)}{\mathbb{Z}_{1}}{S_4}{})$ & $2$ \\
    11 & $(\amal{S_4}{\mathbb{Z}_{n}}{\mathbb{Z}_{1}}{S_4}{})$ & $\infty$&     45 & $(\amal{D_4}{D_{4n}}{D_{4}}{\mathbb{Z}_1}{})$ & $2$ &   79 & $(\amal{\mathbb{Z}_2}{O(2)}{D_{1}}{\mathbb{Z}_1}{})$ & $8$ \\
    12 & $(\amal{\mathbb{Z}_2}{\mathbb{Z}_{2n}}{\mathbb{Z}_{2}}{\mathbb{Z}_1}{})$ & $\infty$ &46 & $(\amal{\mathbb{Z}_1}{D_{n}}{\mathbb{Z}_{1}}{\mathbb{Z}_1}{})$ & $48$ &     80 & $(\amal{D_1}{O(2)}{D_{1}}{\mathbb{Z}_1}{})$ & $4$ \\
    13 & $(\amal{D_1}{\mathbb{Z}_{2n}}{\mathbb{Z}_{2}}{\mathbb{Z}_1}{})$ & $\infty$ &    47 & $(\amal{\mathbb{Z}_2}{D_{n}}{\mathbb{Z}_{1}}{\mathbb{Z}_2}{})$ & $8$ &    81 & $(\amal{V_4}{O(2)}{D_{1}}{\mathbb{Z}_2}{})$ & $4$ \\
    14 & $(\amal{V_4}{\mathbb{Z}_{2n}}{\mathbb{Z}_{2}}{\mathbb{Z}_2}{})$ & $\infty$ &     48 & $(\amal{D_1}{D_{n}}{\mathbb{Z}_{1}}{D_1}{})$ & $4$ &     82 & $(\amal{D_2}{O(2)}{D_{1}}{\mathbb{Z}_2}{})$ & $4$ \\
    15 & $(\amal{D_2}{\mathbb{Z}_{2n}}{\mathbb{Z}_{2}}{\mathbb{Z}_2}{})$ & $\infty$ &    49 & $(\amal{\mathbb{Z}_3}{D_{n}}{\mathbb{Z}_{1}}{\mathbb{Z}_3}{})$ & $4$ &   83 & $(\amal{D_2}{O(2)}{D_{1}}{D_1}{})$ & $2$ \\
    16 & $(\amal{D_2}{\mathbb{Z}_{2n}}{\mathbb{Z}_{2}}{D_1}{})$ & $\infty$&     50 & $(\amal{V_4}{D_{n}}{\mathbb{Z}_{1}}{V_4}{})$ & $12$ &     84 & $(\amal{\mathbb{Z}_4}{O(2)}{D_{1}}{\mathbb{Z}_2}{})$ & $4$ \\
    17 & $(\amal{\mathbb{Z}_4}{\mathbb{Z}_{2n}}{\mathbb{Z}_{2}}{\mathbb{Z}_2}{})$ & $\infty$ &    51 & $(\amal{D_2}{D_{n}}{\mathbb{Z}_{1}}{D_2}{})$ & $4$   &  85 & $(\amal{D_3}{O(2)}{D_{1}}{\mathbb{Z}_3}{})$ & $2$ \\
    18 & $(\amal{D_3}{\mathbb{Z}_{2n}}{\mathbb{Z}_{2}}{\mathbb{Z}_3}{})$ & $\infty$ &    52 & $(\amal{\mathbb{Z}_4}{D_{n}}{\mathbb{Z}_{1}}{\mathbb{Z}_4}{})$ & $4$   &  86 & $(\amal{D_4}{O(2)}{D_{1}}{V_4}{})$ & $2$ \\
    19 & $(\amal{D_4}{\mathbb{Z}_{2n}}{\mathbb{Z}_{2}}{V_4}{})$ & $\infty$ &  53 & $(\amal{D_3}{D_{n}}{\mathbb{Z}_{1}}{D_3}{})$ & $2$ &    87 & $(\amal{D_4}{O(2)}{D_{1}}{D_2}{})$ & $2$ \\
    20 & $(\amal{D_4}{\mathbb{Z}_{2n}}{\mathbb{Z}_{2}}{D_2}{})$ & $\infty$ &   54 & $(\amal{D_4}{D_{n}}{\mathbb{Z}_{1}}{D_4}{})$ & $2$ &  88 & $(\amal{D_4}{O(2)}{D_{1}}{\mathbb{Z}_4}{})$ & $2$ \\
    21 & $(\amal{D_4}{\mathbb{Z}_{2n}}{\mathbb{Z}_{2}}{\mathbb{Z}_4}{})$ & $\infty$ &    55 & $(\amal{A_4}{D_{n}}{\mathbb{Z}_{1}}{A_4}{})$ & $4$ &     89 & $(\amal{S_4}{O(2)}{D_{1}}{A_4}{})$ & $2$ \\
    22 & $(\amal{S_4}{\mathbb{Z}_{2n}}{\mathbb{Z}_{2}}{A_4}{})$ & $\infty$ &  56 & $(\amal{S_4}{D_{n}}{\mathbb{Z}_{1}}{S_4}{})$ & $2$ &     90 & $(\amal{\mathbb{Z}_1}{O(2)}{\mathbb{Z}_{1}}{\mathbb{Z}_1}{})$ & $24$ \\
    23 & $(\amal{\mathbb{Z}_3}{\mathbb{Z}_{3n}}{\mathbb{Z}_{3}}{\mathbb{Z}_1}{})$ & $\infty$ &    57 & $(\amal{\mathbb{Z}_2}{D_{2n}}{\mathbb{Z}_{2}}{\mathbb{Z}_1}{})$ & $8$&   91 & $(\amal{\mathbb{Z}_2}{O(2)}{\mathbb{Z}_{1}}{\mathbb{Z}_2}{})$ & $4$ \\
    24 & $(\amal{A_4}{\mathbb{Z}_{3n}}{\mathbb{Z}_{3}}{V_4}{})$ & $\infty$ &    58 & $(\amal{D_1}{D_{2n}}{\mathbb{Z}_{2}}{\mathbb{Z}_1}{})$ & $4$ &    92 & $(\amal{D_1}{O(2)}{\mathbb{Z}_{1}}{D_1}{})$ & $2$ \\
    25 & $(\amal{\mathbb{Z}_4}{\mathbb{Z}_{4n}}{\mathbb{Z}_{4}}{\mathbb{Z}_1}{})$ & $\infty$ &     59 & $(\amal{V_4}{D_{2n}}{\mathbb{Z}_{2}}{\mathbb{Z}_2}{})$ & $4$ &   93 & $(\amal{\mathbb{Z}_3}{O(2)}{\mathbb{Z}_{1}}{\mathbb{Z}_3}{})$ & $2$  \\
    26 & $(\amal{\mathbb{Z}_2}{D_{n}}{D_{1}}{\mathbb{Z}_1}{})$ & $16$ &    60 & $(\amal{D_2}{D_{2n}}{\mathbb{Z}_{2}}{\mathbb{Z}_2}{})$ & $4$ &   94 & $(\amal{V_4}{O(2)}{\mathbb{Z}_{1}}{V_4}{})$ & $6$  \\
    27 & $(\amal{D_1}{D_{n}}{D_{1}}{\mathbb{Z}_1}{})$ & $8$&   61 & $(\amal{D_2}{D_{2n}}{\mathbb{Z}_{2}}{D_1}{})$ & $2$ &   95 & $(\amal{D_2}{O(2)}{\mathbb{Z}_{1}}{D_2}{})$ & $2$  \\
    28 & $(\amal{V_4}{D_{n}}{D_{1}}{\mathbb{Z}_2}{})$ & $8$&     62 & $(\amal{\mathbb{Z}_4}{D_{2n}}{\mathbb{Z}_{2}}{\mathbb{Z}_2}{})$ & $4$ &  96 & $(\amal{\mathbb{Z}_4}{O(2)}{\mathbb{Z}_{1}}{\mathbb{Z}_4}{})$ & $2$\\
    29 & $(\amal{D_2}{D_{n}}{D_{1}}{\mathbb{Z}_2}{})$ & $8$ &    63 & $(\amal{D_3}{D_{2n}}{\mathbb{Z}_{2}}{\mathbb{Z}_3}{})$ & $2$ &   97 & $(\amal{D_3}{O(2)}{\mathbb{Z}_{1}}{D_3}{})$ & $1$   \\
    30 & $(\amal{D_2}{D_{n}}{D_{1}}{D_1}{})$ & $4$ &  64 & $(\amal{D_4}{D_{2n}}{\mathbb{Z}_{2}}{V_4}{})$ & $2$ &  98 & $(\amal{D_4}{O(2)}{\mathbb{Z}_{1}}{D_4}{})$ & $1$   \\
    31 & $(\amal{\mathbb{Z}_4}{D_{n}}{D_{1}}{\mathbb{Z}_2}{})$ & $8$&     65 & $(\amal{D_4}{D_{2n}}{\mathbb{Z}_{2}}{D_2}{})$ & $2$ &  99 & $(\amal{A_4}{O(2)}{\mathbb{Z}_{1}}{A_4}{})$ & $2$ \\
    32 & $(\amal{D_3}{D_{n}}{D_{1}}{\mathbb{Z}_3}{})$ & $4$ &    66 & $(\amal{D_4}{D_{2n}}{\mathbb{Z}_{2}}{\mathbb{Z}_4}{})$ & $2$ &   100 & $(\amal{S_4}{O(2)}{\mathbb{Z}_{1}}{S_4}{})$ & $1$ \\
    33 & $(\amal{D_4}{D_{n}}{D_{1}}{V_4}{})$ & $4$&   67 & $(\amal{S_4}{D_{2n}}{\mathbb{Z}_{2}}{A_4}{})$ & $2$ & &&\\
    34 & $(\amal{D_4}{D_{n}}{D_{1}}{D_2}{})$ & $4$&     68 & $(\amal{\mathbb{Z}_1}{SO(2)}{\mathbb{Z}_{1}}{\mathbb{Z}_1}{})$ & $48$ &&&\\
    \hline
  \end{tabular}  }
\caption{Conjugacy classes in $S_4\times O(2)$}
\end{table}

 \begin{table}[H]
    \centering
    \caption{Conjugacy Classes of Subgroups in $D_4\times\bz_2\times O(2)$ (part 1)}\label{tab1}
    \vskip 1em
    \begin{tabular}{|rcc|rcc|rcc|}
      \toprule
      ID & $(S)$ & $\abs{W(S)}$ &
        ID & $(S)$ & $\abs{W(S)}$ &
        ID & $(S)$ & $\abs{W(S)}$ \\
      \midrule
      1 & $(\amal{\bz_1}{\bz_{n}}{}{}{})$ & $\infty$ &
        36 & $(\amal{\bz_2^2}{\bz_{2n}}{\bz_{2}}{\bz_2}{})$ & $\infty$ &
        71 & $(\amal{D_2^2}{\bz_{2n}}{\bz_{2}}{D_2^d}{})$ & $\infty$ \\
      2 & $(\amal{\bz_1^2}{\bz_{n}}{}{}{})$ & $\infty$ &
        37 & $(\amal{\bz_2^2}{\bz_{2n}}{\bz_{2}}{\bz_2^z}{})$ & $\infty$ &
        72 & $(\amal{D_2^2}{\bz_{2n}}{\bz_{2}}{D_1^2}{})$ & $\infty$ \\
      3 & $(\amal{\bz_2}{\bz_{n}}{}{}{})$ & $\infty$ &
        38 & $(\amal{D_2}{\bz_{2n}}{\bz_{2}}{\bz_2}{})$ & $\infty$ &
        73 & $(\amal{D_4^\td}{\bz_{2n}}{\bz_{2}}{D_2^z}{})$ & $\infty$ \\
      4 & $(\amal{\bz_2^z}{\bz_{n}}{}{}{})$ & $\infty$ &
        39 & $(\amal{D_2}{\bz_{2n}}{\bz_{2}}{D_1}{})$ & $\infty$ &
        74 & $(\amal{D_4^\td}{\bz_{2n}}{\bz_{2}}{\bz_4^z}{})$ & $\infty$ \\
      5 & $(\amal{\tD_1}{\bz_{n}}{}{}{})$ & $\infty$ &
        40 & $(\amal{D_2^z}{\bz_{2n}}{\bz_{2}}{\bz_2}{})$ & $\infty$ &
        75 & $(\amal{D_4^\td}{\bz_{2n}}{\bz_{2}}{\tD_2}{})$ & $\infty$ \\
      6 & $(\amal{\tD_1^\td}{\bz_{n}}{}{}{})$ & $\infty$ &
        41 & $(\amal{D_2^z}{\bz_{2n}}{\bz_{2}}{D_1^d}{})$ & $\infty$ &
        76 & $(\amal{D_4}{\bz_{2n}}{\bz_{2}}{D_2}{})$ & $\infty$ \\
      7 & $(\amal{D_1}{\bz_{n}}{}{}{})$ & $\infty$ &
        42 & $(\amal{D_2^d}{\bz_{2n}}{\bz_{2}}{\bz_2^z}{})$ & $\infty$ &
        77 & $(\amal{D_4}{\bz_{2n}}{\bz_{2}}{\bz_4}{})$ & $\infty$ \\
      8 & $(\amal{D_1^d}{\bz_{n}}{}{}{})$ & $\infty$ &
        43 & $(\amal{D_2^d}{\bz_{2n}}{\bz_{2}}{D_1}{})$ & $\infty$ &
        78 & $(\amal{D_4}{\bz_{2n}}{\bz_{2}}{\tD_2}{})$ & $\infty$ \\
      9 & $(\amal{\bz_2^2}{\bz_{n}}{}{}{})$ & $\infty$ &
        44 & $(\amal{D_2^d}{\bz_{2n}}{\bz_{2}}{D_1^d}{})$ & $\infty$ &
        79 & $(\amal{D_4^d}{\bz_{2n}}{\bz_{2}}{D_2}{})$ & $\infty$ \\
      10 & $(\amal{D_2}{\bz_{n}}{}{}{})$ & $\infty$ &
        45 & $(\amal{D_1^2}{\bz_{2n}}{\bz_{2}}{\bz_1^2}{})$ & $\infty$ &
        80 & $(\amal{D_4^d}{\bz_{2n}}{\bz_{2}}{\bz_4^z}{})$ & $\infty$ \\
      11 & $(\amal{D_2^z}{\bz_{n}}{}{}{})$ & $\infty$ &
        46 & $(\amal{D_1^2}{\bz_{2n}}{\bz_{2}}{D_1}{})$ & $\infty$ &
        81 & $(\amal{D_4^d}{\bz_{2n}}{\bz_{2}}{\tD_2^z}{})$ & $\infty$ \\
      12 & $(\amal{D_2^d}{\bz_{n}}{}{}{})$ & $\infty$ &
        47 & $(\amal{D_1^2}{\bz_{2n}}{\bz_{2}}{D_1^d}{})$ & $\infty$ &
        82 & $(\amal{D_4^z}{\bz_{2n}}{\bz_{2}}{D_2^z}{})$ & $\infty$ \\
      13 & $(\amal{D_1^2}{\bz_{n}}{}{}{})$ & $\infty$ &
        48 & $(\amal{\bz_4}{\bz_{2n}}{\bz_{2}}{\bz_2}{})$ & $\infty$ &
        83 & $(\amal{D_4^z}{\bz_{2n}}{\bz_{2}}{\bz_4}{})$ & $\infty$ \\
      14 & $(\amal{\bz_4}{\bz_{n}}{}{}{})$ & $\infty$ &
        49 & $(\amal{\bz_4^z}{\bz_{2n}}{\bz_{2}}{\bz_2}{})$ & $\infty$ &
        84 & $(\amal{D_4^z}{\bz_{2n}}{\bz_{2}}{\tD_2^z}{})$ & $\infty$ \\
      15 & $(\amal{\bz_4^z}{\bz_{n}}{}{}{})$ & $\infty$ &
        50 & $(\amal{\tD_2}{\bz_{2n}}{\bz_{2}}{\bz_2}{})$ & $\infty$ &
        85 & $(\amal{D_4^2}{\bz_{2n}}{\bz_{2}}{\tD_2^2}{})$ & $\infty$ \\
      16 & $(\amal{\tD_2}{\bz_{n}}{}{}{})$ & $\infty$ &
        51 & $(\amal{\tD_2}{\bz_{2n}}{\bz_{2}}{\tD_1}{})$ & $\infty$ &
        86 & $(\amal{D_4^2}{\bz_{2n}}{\bz_{2}}{\bz_4^2}{})$ & $\infty$ \\
      17 & $(\amal{\tD_2^z}{\bz_{n}}{}{}{})$ & $\infty$ &
        52 & $(\amal{\tD_2^z}{\bz_{2n}}{\bz_{2}}{\bz_2}{})$ & $\infty$ &
        87 & $(\amal{D_4^2}{\bz_{2n}}{\bz_{2}}{D_2^2}{})$ & $\infty$ \\
      18 & $(\amal{\tD_2^\td}{\bz_{n}}{}{}{})$ & $\infty$ &
        53 & $(\amal{\tD_2^z}{\bz_{2n}}{\bz_{2}}{\tD_1^\td}{})$ & $\infty$ &
        88 & $(\amal{D_4^2}{\bz_{2n}}{\bz_{2}}{D_4^\td}{})$ & $\infty$ \\
      19 & $(\amal{\tD_1^2}{\bz_{n}}{}{}{})$ & $\infty$ &
        54 & $(\amal{\tD_2^\td}{\bz_{2n}}{\bz_{2}}{\bz_2^z}{})$ & $\infty$ &
        89 & $(\amal{D_4^2}{\bz_{2n}}{\bz_{2}}{D_4}{})$ & $\infty$ \\
      20 & $(\amal{\tD_2^2}{\bz_{n}}{}{}{})$ & $\infty$ &
        55 & $(\amal{\tD_2^\td}{\bz_{2n}}{\bz_{2}}{\tD_1}{})$ & $\infty$ &
        90 & $(\amal{D_4^2}{\bz_{2n}}{\bz_{2}}{D_4^d}{})$ & $\infty$ \\
      21 & $(\amal{\bz_4^2}{\bz_{n}}{}{}{})$ & $\infty$ &
        56 & $(\amal{\tD_2^\td}{\bz_{2n}}{\bz_{2}}{\tD_1^\td}{})$ & $\infty$ &
        91 & $(\amal{D_4^2}{\bz_{2n}}{\bz_{2}}{D_4^z}{})$ & $\infty$ \\
      22 & $(\amal{D_2^2}{\bz_{n}}{}{}{})$ & $\infty$ &
        57 & $(\amal{\tD_1^2}{\bz_{2n}}{\bz_{2}}{\bz_1^2}{})$ & $\infty$ &
        92 & $(\amal{\bz_4}{\bz_{4n}}{\bz_{4}}{\bz_1}{})$ & $\infty$ \\
      23 & $(\amal{D_4^\td}{\bz_{n}}{}{}{})$ & $\infty$ &
        58 & $(\amal{\tD_1^2}{\bz_{2n}}{\bz_{2}}{\tD_1}{})$ & $\infty$ &
        93 & $(\amal{\bz_4^z}{\bz_{4n}}{\bz_{4}}{\bz_1}{})$ & $\infty$ \\
      24 & $(\amal{D_4}{\bz_{n}}{}{}{})$ & $\infty$ &
        59 & $(\amal{\tD_1^2}{\bz_{2n}}{\bz_{2}}{\tD_1^\td}{})$ & $\infty$ &
        94 & $(\amal{\bz_4^2}{\bz_{4n}}{\bz_{4}}{\bz_1^2}{})$ & $\infty$ \\
      25 & $(\amal{D_4^d}{\bz_{n}}{}{}{})$ & $\infty$ &
        60 & $(\amal{\tD_2^2}{\bz_{2n}}{\bz_{2}}{\bz_2^2}{})$ & $\infty$ &
        95 & $(\amal{\bz_4^2}{\bz_{4n}}{\bz_{4}}{\bz_2^z}{})$ & $\infty$ \\
      26 & $(\amal{D_4^z}{\bz_{n}}{}{}{})$ & $\infty$ &
        61 & $(\amal{\tD_2^2}{\bz_{2n}}{\bz_{2}}{\tD_2}{})$ & $\infty$ &
        96 & $(\amal{\bz_1^2}{D_{n}}{D_{1}}{\bz_1}{})$ & $32$ \\
      27 & $(\amal{D_4^2}{\bz_{n}}{}{}{})$ & $\infty$ &
        62 & $(\amal{\tD_2^2}{\bz_{2n}}{\bz_{2}}{\tD_2^z}{})$ & $\infty$ &
        97 & $(\amal{\bz_2}{D_{n}}{D_{1}}{\bz_1}{})$ & $32$ \\
      28 & $(\amal{\bz_1^2}{\bz_{2n}}{\bz_{2}}{\bz_1}{})$ & $\infty$ &
        63 & $(\amal{\tD_2^2}{\bz_{2n}}{\bz_{2}}{\tD_2^\td}{})$ & $\infty$ &
        98 & $(\amal{\bz_2^z}{D_{n}}{D_{1}}{\bz_1}{})$ & $32$ \\
      29 & $(\amal{\bz_2}{\bz_{2n}}{\bz_{2}}{\bz_1}{})$ & $\infty$ &
        64 & $(\amal{\tD_2^2}{\bz_{2n}}{\bz_{2}}{\tD_1^2}{})$ & $\infty$ &
        99 & $(\amal{\tD_1}{D_{n}}{D_{1}}{\bz_1}{})$ & $16$ \\
      30 & $(\amal{\bz_2^z}{\bz_{2n}}{\bz_{2}}{\bz_1}{})$ & $\infty$ &
        65 & $(\amal{\bz_4^2}{\bz_{2n}}{\bz_{2}}{\bz_2^2}{})$ & $\infty$ &
        100 & $(\amal{\tD_1^\td}{D_{n}}{D_{1}}{\bz_1}{})$ & $16$ \\
      31 & $(\amal{\tD_1}{\bz_{2n}}{\bz_{2}}{\bz_1}{})$ & $\infty$ &
        66 & $(\amal{\bz_4^2}{\bz_{2n}}{\bz_{2}}{\bz_4}{})$ & $\infty$ &
        101 & $(\amal{D_1}{D_{n}}{D_{1}}{\bz_1}{})$ & $16$ \\
      32 & $(\amal{\tD_1^\td}{\bz_{2n}}{\bz_{2}}{\bz_1}{})$ & $\infty$ &
        67 & $(\amal{\bz_4^2}{\bz_{2n}}{\bz_{2}}{\bz_4^z}{})$ & $\infty$ &
        102 & $(\amal{D_1^d}{D_{n}}{D_{1}}{\bz_1}{})$ & $16$ \\
      33 & $(\amal{D_1}{\bz_{2n}}{\bz_{2}}{\bz_1}{})$ & $\infty$ &
        68 & $(\amal{D_2^2}{\bz_{2n}}{\bz_{2}}{\bz_2^2}{})$ & $\infty$ &
        103 & $(\amal{\bz_2^2}{D_{n}}{D_{1}}{\bz_1^2}{})$ & $16$ \\
      34 & $(\amal{D_1^d}{\bz_{2n}}{\bz_{2}}{\bz_1}{})$ & $\infty$ &
        69 & $(\amal{D_2^2}{\bz_{2n}}{\bz_{2}}{D_2}{})$ & $\infty$ &
        104 & $(\amal{\bz_2^2}{D_{n}}{D_{1}}{\bz_2}{})$ & $16$ \\
      35 & $(\amal{\bz_2^2}{\bz_{2n}}{\bz_{2}}{\bz_1^2}{})$ & $\infty$ &
        70 & $(\amal{D_2^2}{\bz_{2n}}{\bz_{2}}{D_2^z}{})$ & $\infty$ &
        105 & $(\amal{\bz_2^2}{D_{n}}{D_{1}}{\bz_2^z}{})$ & $16$ \\
    \bottomrule
  \end{tabular}
\end{table}

  \begin{table}[H]
    \centering
    \caption{Conjugacy Classes of Subgroups in $D_4\times\bz_2\times O(2)$ (part 2)}\label{tab2}
    \vskip 1em
    \begin{tabular}{|rcc|rcc|rcc|}
      \toprule
      ID & $(S)$ & $\abs{W(S)}$ &
        ID & $(S)$ & $\abs{W(S)}$ &
        ID & $(S)$ & $\abs{W(S)}$ \\
      \midrule
      106 & $(\amal{D_2}{D_{n}}{D_{1}}{\bz_2}{})$ & $16$ &
        141 & $(\amal{D_4^\td}{D_{n}}{D_{1}}{D_2^z}{})$ & $8$ &
        176 & $(\amal{\tD_2^z}{D_{2n}}{D_{2}}{\bz_1}{\tD_1^\td})$ & $8$ \\
      107 & $(\amal{D_2}{D_{n}}{D_{1}}{D_1}{})$ & $8$ &
        142 & $(\amal{D_4^\td}{D_{n}}{D_{1}}{\bz_4^z}{})$ & $8$ &
        177 & $(\amal{\tD_2^\td}{D_{2n}}{D_{2}}{\bz_1}{\bz_2^z})$ & $8$ \\
      108 & $(\amal{D_2^z}{D_{n}}{D_{1}}{\bz_2}{})$ & $16$ &
        143 & $(\amal{D_4^\td}{D_{n}}{D_{1}}{\tD_2}{})$ & $8$ &
        178 & $(\amal{\tD_2^\td}{D_{2n}}{D_{2}}{\bz_1}{\tD_1})$ & $8$ \\
      109 & $(\amal{D_2^z}{D_{n}}{D_{1}}{D_1^d}{})$ & $8$ &
        144 & $(\amal{D_4}{D_{n}}{D_{1}}{D_2}{})$ & $8$ &
        179 & $(\amal{\tD_2^\td}{D_{2n}}{D_{2}}{\bz_1}{\tD_1^\td})$ & $8$ \\
      110 & $(\amal{D_2^d}{D_{n}}{D_{1}}{\bz_2^z}{})$ & $8$ &
        145 & $(\amal{D_4}{D_{n}}{D_{1}}{\bz_4}{})$ & $8$ &
        180 & $(\amal{\tD_1^2}{D_{2n}}{D_{2}}{\bz_1}{\bz_1^2})$ & $8$ \\
      111 & $(\amal{D_2^d}{D_{n}}{D_{1}}{D_1}{})$ & $8$ &
        146 & $(\amal{D_4}{D_{n}}{D_{1}}{\tD_2}{})$ & $8$ &
        181 & $(\amal{\tD_1^2}{D_{2n}}{D_{2}}{\bz_1}{\tD_1})$ & $8$ \\
      112 & $(\amal{D_2^d}{D_{n}}{D_{1}}{D_1^d}{})$ & $8$ &
        147 & $(\amal{D_4^d}{D_{n}}{D_{1}}{D_2}{})$ & $8$ &
        182 & $(\amal{\tD_1^2}{D_{2n}}{D_{2}}{\bz_1}{\tD_1^\td})$ & $8$ \\
      113 & $(\amal{D_1^2}{D_{n}}{D_{1}}{\bz_1^2}{})$ & $8$ &
        148 & $(\amal{D_4^d}{D_{n}}{D_{1}}{\bz_4^z}{})$ & $8$ &
        183 & $(\amal{\tD_2^2}{D_{2n}}{D_{2}}{\bz_1^2}{\bz_2^2})$ & $8$ \\
      114 & $(\amal{D_1^2}{D_{n}}{D_{1}}{D_1}{})$ & $8$ &
        149 & $(\amal{D_4^d}{D_{n}}{D_{1}}{\tD_2^z}{})$ & $8$ &
        184 & $(\amal{\tD_2^2}{D_{2n}}{D_{2}}{\bz_1^2}{\tD_1^2})$ & $4$ \\
      115 & $(\amal{D_1^2}{D_{n}}{D_{1}}{D_1^d}{})$ & $8$ &
        150 & $(\amal{D_4^z}{D_{n}}{D_{1}}{D_2^z}{})$ & $8$ &
        185 & $(\amal{\tD_2^2}{D_{2n}}{D_{2}}{\bz_2}{\bz_2^2})$ & $8$ \\
      116 & $(\amal{\bz_4}{D_{n}}{D_{1}}{\bz_2}{})$ & $16$ &
        151 & $(\amal{D_4^z}{D_{n}}{D_{1}}{\bz_4}{})$ & $8$ &
        186 & $(\amal{\tD_2^2}{D_{2n}}{D_{2}}{\bz_2}{\tD_2})$ & $8$ \\
      117 & $(\amal{\bz_4^z}{D_{n}}{D_{1}}{\bz_2}{})$ & $16$ &
        152 & $(\amal{D_4^z}{D_{n}}{D_{1}}{\tD_2^z}{})$ & $8$ &
        187 & $(\amal{\tD_2^2}{D_{2n}}{D_{2}}{\bz_2}{\tD_2^z})$ & $8$ \\
      118 & $(\amal{\tD_2}{D_{n}}{D_{1}}{\bz_2}{})$ & $16$ &
        153 & $(\amal{D_4^2}{D_{n}}{D_{1}}{\tD_2^2}{})$ & $4$ &
        188 & $(\amal{\tD_2^2}{D_{2n}}{D_{2}}{\bz_2^z}{\bz_2^2})$ & $8$ \\
      119 & $(\amal{\tD_2}{D_{n}}{D_{1}}{\tD_1}{})$ & $8$ &
        154 & $(\amal{D_4^2}{D_{n}}{D_{1}}{\bz_4^2}{})$ & $4$ &
        189 & $(\amal{\tD_2^2}{D_{2n}}{D_{2}}{\bz_2^z}{\tD_2^\td})$ & $4$ \\
      120 & $(\amal{\tD_2^z}{D_{n}}{D_{1}}{\bz_2}{})$ & $16$ &
        155 & $(\amal{D_4^2}{D_{n}}{D_{1}}{D_2^2}{})$ & $4$ &
        190 & $(\amal{\tD_2^2}{D_{2n}}{D_{2}}{\tD_1}{\tD_2})$ & $4$ \\
      121 & $(\amal{\tD_2^z}{D_{n}}{D_{1}}{\tD_1^\td}{})$ & $8$ &
        156 & $(\amal{D_4^2}{D_{n}}{D_{1}}{D_4^\td}{})$ & $4$ &
        191 & $(\amal{\tD_2^2}{D_{2n}}{D_{2}}{\tD_1}{\tD_2^\td})$ & $4$ \\
      122 & $(\amal{\tD_2^\td}{D_{n}}{D_{1}}{\bz_2^z}{})$ & $8$ &
        157 & $(\amal{D_4^2}{D_{n}}{D_{1}}{D_4}{})$ & $4$ &
        192 & $(\amal{\tD_2^2}{D_{2n}}{D_{2}}{\tD_1}{\tD_1^2})$ & $4$ \\
      123 & $(\amal{\tD_2^\td}{D_{n}}{D_{1}}{\tD_1}{})$ & $8$ &
        158 & $(\amal{D_4^2}{D_{n}}{D_{1}}{D_4^d}{})$ & $4$ &
        193 & $(\amal{\tD_2^2}{D_{2n}}{D_{2}}{\tD_1^\td}{\tD_2^z})$ & $4$ \\
      124 & $(\amal{\tD_2^\td}{D_{n}}{D_{1}}{\tD_1^\td}{})$ & $8$ &
        159 & $(\amal{D_4^2}{D_{n}}{D_{1}}{D_4^z}{})$ & $4$ &
        194 & $(\amal{\tD_2^2}{D_{2n}}{D_{2}}{\tD_1^\td}{\tD_2^\td})$ & $4$ \\
      125 & $(\amal{\tD_1^2}{D_{n}}{D_{1}}{\bz_1^2}{})$ & $8$ &
        160 & $(\amal{\bz_2^2}{D_{2n}}{D_{2}}{\bz_1}{\bz_1^2})$ & $16$ &
        195 & $(\amal{\tD_2^2}{D_{2n}}{D_{2}}{\tD_1^\td}{\tD_1^2})$ & $4$ \\
      126 & $(\amal{\tD_1^2}{D_{n}}{D_{1}}{\tD_1}{})$ & $8$ &
        161 & $(\amal{\bz_2^2}{D_{2n}}{D_{2}}{\bz_1}{\bz_2})$ & $16$ &
        196 & $(\amal{\bz_4^2}{D_{2n}}{D_{2}}{\bz_2}{\bz_2^2})$ & $8$ \\
      127 & $(\amal{\tD_1^2}{D_{n}}{D_{1}}{\tD_1^\td}{})$ & $8$ &
        162 & $(\amal{\bz_2^2}{D_{2n}}{D_{2}}{\bz_1}{\bz_2^z})$ & $16$ &
        197 & $(\amal{\bz_4^2}{D_{2n}}{D_{2}}{\bz_2}{\bz_4})$ & $8$ \\
      128 & $(\amal{\tD_2^2}{D_{n}}{D_{1}}{\bz_2^2}{})$ & $8$ &
        163 & $(\amal{D_2}{D_{2n}}{D_{2}}{\bz_1}{\bz_2})$ & $16$ &
        198 & $(\amal{\bz_4^2}{D_{2n}}{D_{2}}{\bz_2}{\bz_4^z})$ & $8$ \\
      129 & $(\amal{\tD_2^2}{D_{n}}{D_{1}}{\tD_2}{})$ & $8$ &
        164 & $(\amal{D_2}{D_{2n}}{D_{2}}{\bz_1}{D_1})$ & $8$ &
        199 & $(\amal{D_2^2}{D_{2n}}{D_{2}}{\bz_1^2}{\bz_2^2})$ & $8$ \\
      130 & $(\amal{\tD_2^2}{D_{n}}{D_{1}}{\tD_2^z}{})$ & $8$ &
        165 & $(\amal{D_2^z}{D_{2n}}{D_{2}}{\bz_1}{\bz_2})$ & $16$ &
        200 & $(\amal{D_2^2}{D_{2n}}{D_{2}}{\bz_1^2}{D_1^2})$ & $4$ \\
      131 & $(\amal{\tD_2^2}{D_{n}}{D_{1}}{\tD_2^\td}{})$ & $4$ &
        166 & $(\amal{D_2^z}{D_{2n}}{D_{2}}{\bz_1}{D_1^d})$ & $8$ &
        201 & $(\amal{D_2^2}{D_{2n}}{D_{2}}{\bz_2}{\bz_2^2})$ & $8$ \\
      132 & $(\amal{\tD_2^2}{D_{n}}{D_{1}}{\tD_1^2}{})$ & $4$ &
        167 & $(\amal{D_2^d}{D_{2n}}{D_{2}}{\bz_1}{\bz_2^z})$ & $8$ &
        202 & $(\amal{D_2^2}{D_{2n}}{D_{2}}{\bz_2}{D_2})$ & $8$ \\
      133 & $(\amal{\bz_4^2}{D_{n}}{D_{1}}{\bz_2^2}{})$ & $8$ &
        168 & $(\amal{D_2^d}{D_{2n}}{D_{2}}{\bz_1}{D_1})$ & $8$ &
        203 & $(\amal{D_2^2}{D_{2n}}{D_{2}}{\bz_2}{D_2^z})$ & $8$ \\
      134 & $(\amal{\bz_4^2}{D_{n}}{D_{1}}{\bz_4}{})$ & $8$ &
        169 & $(\amal{D_2^d}{D_{2n}}{D_{2}}{\bz_1}{D_1^d})$ & $8$ &
        204 & $(\amal{D_2^2}{D_{2n}}{D_{2}}{\bz_2^z}{\bz_2^2})$ & $8$ \\
      135 & $(\amal{\bz_4^2}{D_{n}}{D_{1}}{\bz_4^z}{})$ & $8$ &
        170 & $(\amal{D_1^2}{D_{2n}}{D_{2}}{\bz_1}{\bz_1^2})$ & $8$ &
        205 & $(\amal{D_2^2}{D_{2n}}{D_{2}}{\bz_2^z}{D_2^d})$ & $4$ \\
      136 & $(\amal{D_2^2}{D_{n}}{D_{1}}{\bz_2^2}{})$ & $8$ &
        171 & $(\amal{D_1^2}{D_{2n}}{D_{2}}{\bz_1}{D_1})$ & $8$ &
        206 & $(\amal{D_2^2}{D_{2n}}{D_{2}}{D_1}{D_2})$ & $4$ \\
      137 & $(\amal{D_2^2}{D_{n}}{D_{1}}{D_2}{})$ & $8$ &
        172 & $(\amal{D_1^2}{D_{2n}}{D_{2}}{\bz_1}{D_1^d})$ & $8$ &
        207 & $(\amal{D_2^2}{D_{2n}}{D_{2}}{D_1}{D_2^d})$ & $4$ \\
      138 & $(\amal{D_2^2}{D_{n}}{D_{1}}{D_2^z}{})$ & $8$ &
        173 & $(\amal{\tD_2}{D_{2n}}{D_{2}}{\bz_1}{\bz_2})$ & $16$ &
        208 & $(\amal{D_2^2}{D_{2n}}{D_{2}}{D_1}{D_1^2})$ & $4$ \\
      139 & $(\amal{D_2^2}{D_{n}}{D_{1}}{D_2^d}{})$ & $4$ &
        174 & $(\amal{\tD_2}{D_{2n}}{D_{2}}{\bz_1}{\tD_1})$ & $8$ &
        209 & $(\amal{D_2^2}{D_{2n}}{D_{2}}{D_1^d}{D_2^z})$ & $4$ \\
      140 & $(\amal{D_2^2}{D_{n}}{D_{1}}{D_1^2}{})$ & $4$ &
        175 & $(\amal{\tD_2^z}{D_{2n}}{D_{2}}{\bz_1}{\bz_2})$ & $16$ &
        210 & $(\amal{D_2^2}{D_{2n}}{D_{2}}{D_1^d}{D_2^d})$ & $4$ \\
    \bottomrule
  \end{tabular}
\end{table}

  \begin{table}[H]
    \centering
    \caption{Conjugacy Classes of Subgroups in $D_4\times\bz_2\times O(2)$ (part 3)}\label{tab3}
    \vskip 1em
    \begin{tabular}{|rcc|rcc|rcc|}
      \toprule
      ID & $(S)$ & $\abs{W(S)}$ &
        ID & $(S)$ & $\abs{W(S)}$ &
        ID & $(S)$ & $\abs{W(S)}$ \\
      \midrule
      211 & $(\amal{D_2^2}{D_{2n}}{D_{2}}{D_1^d}{D_1^2})$ & $4$ &
        246 & $(\amal{D_4}{D_{4n}}{D_{4}}{\bz_1}{})$ & $4$ &
        281 & $(\amal{\tD_1}{D_{2n}}{\bz_{2}}{\bz_1}{})$ & $8$ \\
      212 & $(\amal{D_4^\td}{D_{2n}}{D_{2}}{\bz_2}{D_2^z})$ & $8$ &
        247 & $(\amal{D_4^d}{D_{4n}}{D_{4}}{\bz_1}{})$ & $4$ &
        282 & $(\amal{\tD_1^\td}{D_{2n}}{\bz_{2}}{\bz_1}{})$ & $8$ \\
      213 & $(\amal{D_4^\td}{D_{2n}}{D_{2}}{\bz_2}{\bz_4^z})$ & $8$ &
        248 & $(\amal{D_4^z}{D_{4n}}{D_{4}}{\bz_1}{})$ & $4$ &
        283 & $(\amal{D_1}{D_{2n}}{\bz_{2}}{\bz_1}{})$ & $8$ \\
      214 & $(\amal{D_4^\td}{D_{2n}}{D_{2}}{\bz_2}{\tD_2})$ & $8$ &
        249 & $(\amal{D_4^2}{D_{4n}}{D_{4}}{\bz_1^2}{})$ & $2$ &
        284 & $(\amal{D_1^d}{D_{2n}}{\bz_{2}}{\bz_1}{})$ & $8$ \\
      215 & $(\amal{D_4}{D_{2n}}{D_{2}}{\bz_2}{D_2})$ & $8$ &
        250 & $(\amal{D_4^2}{D_{4n}}{D_{4}}{\bz_2^z}{})$ & $2$ &
        285 & $(\amal{\bz_2^2}{D_{2n}}{\bz_{2}}{\bz_1^2}{})$ & $8$ \\
      216 & $(\amal{D_4}{D_{2n}}{D_{2}}{\bz_2}{\bz_4})$ & $8$ &
        251 & $(\amal{\bz_1}{D_{n}}{}{}{})$ & $32$ &
        286 & $(\amal{\bz_2^2}{D_{2n}}{\bz_{2}}{\bz_2}{})$ & $8$ \\
      217 & $(\amal{D_4}{D_{2n}}{D_{2}}{\bz_2}{\tD_2})$ & $8$ &
        252 & $(\amal{\bz_1^2}{D_{n}}{}{}{})$ & $16$ &
        287 & $(\amal{\bz_2^2}{D_{2n}}{\bz_{2}}{\bz_2^z}{})$ & $8$ \\
      218 & $(\amal{D_4^d}{D_{2n}}{D_{2}}{\bz_2}{D_2})$ & $8$ &
        253 & $(\amal{\bz_2}{D_{n}}{}{}{})$ & $16$ &
        288 & $(\amal{D_2}{D_{2n}}{\bz_{2}}{\bz_2}{})$ & $8$ \\
      219 & $(\amal{D_4^d}{D_{2n}}{D_{2}}{\bz_2}{\bz_4^z})$ & $8$ &
        254 & $(\amal{\bz_2^z}{D_{n}}{}{}{})$ & $16$ &
        289 & $(\amal{D_2}{D_{2n}}{\bz_{2}}{D_1}{})$ & $4$ \\
      220 & $(\amal{D_4^d}{D_{2n}}{D_{2}}{\bz_2}{\tD_2^z})$ & $8$ &
        255 & $(\amal{\tD_1}{D_{n}}{}{}{})$ & $8$ &
        290 & $(\amal{D_2^z}{D_{2n}}{\bz_{2}}{\bz_2}{})$ & $8$ \\
      221 & $(\amal{D_4^z}{D_{2n}}{D_{2}}{\bz_2}{D_2^z})$ & $8$ &
        256 & $(\amal{\tD_1^\td}{D_{n}}{}{}{})$ & $8$ &
        291 & $(\amal{D_2^z}{D_{2n}}{\bz_{2}}{D_1^d}{})$ & $4$ \\
      222 & $(\amal{D_4^z}{D_{2n}}{D_{2}}{\bz_2}{\bz_4})$ & $8$ &
        257 & $(\amal{D_1}{D_{n}}{}{}{})$ & $8$ &
        292 & $(\amal{D_2^d}{D_{2n}}{\bz_{2}}{\bz_2^z}{})$ & $4$ \\
      223 & $(\amal{D_4^z}{D_{2n}}{D_{2}}{\bz_2}{\tD_2^z})$ & $8$ &
        258 & $(\amal{D_1^d}{D_{n}}{}{}{})$ & $8$ &
        293 & $(\amal{D_2^d}{D_{2n}}{\bz_{2}}{D_1}{})$ & $4$ \\
      224 & $(\amal{D_4^2}{D_{2n}}{D_{2}}{\bz_2^2}{\tD_2^2})$ & $4$ &
        259 & $(\amal{\bz_2^2}{D_{n}}{}{}{})$ & $8$ &
        294 & $(\amal{D_2^d}{D_{2n}}{\bz_{2}}{D_1^d}{})$ & $4$ \\
      225 & $(\amal{D_4^2}{D_{2n}}{D_{2}}{\bz_2^2}{\bz_4^2})$ & $4$ &
        260 & $(\amal{D_2}{D_{n}}{}{}{})$ & $8$ &
        295 & $(\amal{D_1^2}{D_{2n}}{\bz_{2}}{\bz_1^2}{})$ & $4$ \\
      226 & $(\amal{D_4^2}{D_{2n}}{D_{2}}{\bz_2^2}{D_2^2})$ & $4$ &
        261 & $(\amal{D_2^z}{D_{n}}{}{}{})$ & $8$ &
        296 & $(\amal{D_1^2}{D_{2n}}{\bz_{2}}{D_1}{})$ & $4$ \\
      227 & $(\amal{D_4^2}{D_{2n}}{D_{2}}{D_2}{D_2^2})$ & $4$ &
        262 & $(\amal{D_2^d}{D_{n}}{}{}{})$ & $4$ &
        297 & $(\amal{D_1^2}{D_{2n}}{\bz_{2}}{D_1^d}{})$ & $4$ \\
      228 & $(\amal{D_4^2}{D_{2n}}{D_{2}}{D_2}{D_4})$ & $4$ &
        263 & $(\amal{D_1^2}{D_{n}}{}{}{})$ & $4$ &
        298 & $(\amal{\bz_4}{D_{2n}}{\bz_{2}}{\bz_2}{})$ & $8$ \\
      229 & $(\amal{D_4^2}{D_{2n}}{D_{2}}{D_2}{D_4^d})$ & $4$ &
        264 & $(\amal{\bz_4}{D_{n}}{}{}{})$ & $8$ &
        299 & $(\amal{\bz_4^z}{D_{2n}}{\bz_{2}}{\bz_2}{})$ & $8$ \\
      230 & $(\amal{D_4^2}{D_{2n}}{D_{2}}{D_2^z}{D_2^2})$ & $4$ &
        265 & $(\amal{\bz_4^z}{D_{n}}{}{}{})$ & $8$ &
        300 & $(\amal{\tD_2}{D_{2n}}{\bz_{2}}{\bz_2}{})$ & $8$ \\
      231 & $(\amal{D_4^2}{D_{2n}}{D_{2}}{D_2^z}{D_4^\td})$ & $4$ &
        266 & $(\amal{\tD_2}{D_{n}}{}{}{})$ & $8$ &
        301 & $(\amal{\tD_2}{D_{2n}}{\bz_{2}}{\tD_1}{})$ & $4$ \\
      232 & $(\amal{D_4^2}{D_{2n}}{D_{2}}{D_2^z}{D_4^z})$ & $4$ &
        267 & $(\amal{\tD_2^z}{D_{n}}{}{}{})$ & $8$ &
        302 & $(\amal{\tD_2^z}{D_{2n}}{\bz_{2}}{\bz_2}{})$ & $8$ \\
      233 & $(\amal{D_4^2}{D_{2n}}{D_{2}}{\bz_4}{\bz_4^2})$ & $4$ &
        268 & $(\amal{\tD_2^\td}{D_{n}}{}{}{})$ & $4$ &
        303 & $(\amal{\tD_2^z}{D_{2n}}{\bz_{2}}{\tD_1^\td}{})$ & $4$ \\
      234 & $(\amal{D_4^2}{D_{2n}}{D_{2}}{\bz_4}{D_4})$ & $4$ &
        269 & $(\amal{\tD_1^2}{D_{n}}{}{}{})$ & $4$ &
        304 & $(\amal{\tD_2^\td}{D_{2n}}{\bz_{2}}{\bz_2^z}{})$ & $4$ \\
      235 & $(\amal{D_4^2}{D_{2n}}{D_{2}}{\bz_4}{D_4^z})$ & $4$ &
        270 & $(\amal{\tD_2^2}{D_{n}}{}{}{})$ & $4$ &
        305 & $(\amal{\tD_2^\td}{D_{2n}}{\bz_{2}}{\tD_1}{})$ & $4$ \\
      236 & $(\amal{D_4^2}{D_{2n}}{D_{2}}{\bz_4^z}{\bz_4^2})$ & $4$ &
        271 & $(\amal{\bz_4^2}{D_{n}}{}{}{})$ & $4$ &
        306 & $(\amal{\tD_2^\td}{D_{2n}}{\bz_{2}}{\tD_1^\td}{})$ & $4$ \\
      237 & $(\amal{D_4^2}{D_{2n}}{D_{2}}{\bz_4^z}{D_4^\td})$ & $4$ &
        272 & $(\amal{D_2^2}{D_{n}}{}{}{})$ & $4$ &
        307 & $(\amal{\tD_1^2}{D_{2n}}{\bz_{2}}{\bz_1^2}{})$ & $4$ \\
      238 & $(\amal{D_4^2}{D_{2n}}{D_{2}}{\bz_4^z}{D_4^d})$ & $4$ &
        273 & $(\amal{D_4^\td}{D_{n}}{}{}{})$ & $4$ &
        308 & $(\amal{\tD_1^2}{D_{2n}}{\bz_{2}}{\tD_1}{})$ & $4$ \\
      239 & $(\amal{D_4^2}{D_{2n}}{D_{2}}{\tD_2}{\tD_2^2})$ & $4$ &
        274 & $(\amal{D_4}{D_{n}}{}{}{})$ & $4$ &
        309 & $(\amal{\tD_1^2}{D_{2n}}{\bz_{2}}{\tD_1^\td}{})$ & $4$ \\
      240 & $(\amal{D_4^2}{D_{2n}}{D_{2}}{\tD_2}{D_4^\td})$ & $4$ &
        275 & $(\amal{D_4^d}{D_{n}}{}{}{})$ & $4$ &
        310 & $(\amal{\tD_2^2}{D_{2n}}{\bz_{2}}{\bz_2^2}{})$ & $4$ \\
      241 & $(\amal{D_4^2}{D_{2n}}{D_{2}}{\tD_2}{D_4})$ & $4$ &
        276 & $(\amal{D_4^z}{D_{n}}{}{}{})$ & $4$ &
        311 & $(\amal{\tD_2^2}{D_{2n}}{\bz_{2}}{\tD_2}{})$ & $4$ \\
      242 & $(\amal{D_4^2}{D_{2n}}{D_{2}}{\tD_2^z}{\tD_2^2})$ & $4$ &
        277 & $(\amal{D_4^2}{D_{n}}{}{}{})$ & $2$ &
        312 & $(\amal{\tD_2^2}{D_{2n}}{\bz_{2}}{\tD_2^z}{})$ & $4$ \\
      243 & $(\amal{D_4^2}{D_{2n}}{D_{2}}{\tD_2^z}{D_4^d})$ & $4$ &
        278 & $(\amal{\bz_1^2}{D_{2n}}{\bz_{2}}{\bz_1}{})$ & $16$ &
        313 & $(\amal{\tD_2^2}{D_{2n}}{\bz_{2}}{\tD_2^\td}{})$ & $2$ \\
      244 & $(\amal{D_4^2}{D_{2n}}{D_{2}}{\tD_2^z}{D_4^z})$ & $4$ &
        279 & $(\amal{\bz_2}{D_{2n}}{\bz_{2}}{\bz_1}{})$ & $16$ &
        314 & $(\amal{\tD_2^2}{D_{2n}}{\bz_{2}}{\tD_1^2}{})$ & $2$ \\
      245 & $(\amal{D_4^\td}{D_{4n}}{D_{4}}{\bz_1}{})$ & $4$ &
        280 & $(\amal{\bz_2^z}{D_{2n}}{\bz_{2}}{\bz_1}{})$ & $16$ &
        315 & $(\amal{\bz_4^2}{D_{2n}}{\bz_{2}}{\bz_2^2}{})$ & $4$ \\
    \bottomrule
  \end{tabular}
\end{table}

  \begin{table}[H]
    \centering
    \caption{Conjugacy Classes of Subgroups in $D_4\times\bz_2\times O(2)$ (part 4)}\label{tab4}
    \vskip 1em
    \begin{tabular}{|rcc|rcc|rcc|}
      \toprule
      ID & $(S)$ & $\abs{W(S)}$ &
        ID & $(S)$ & $\abs{W(S)}$ &
        ID & $(S)$ & $\abs{W(S)}$ \\
      \midrule
      316 & $(\amal{\bz_4^2}{D_{2n}}{\bz_{2}}{\bz_4}{})$ & $4$ &
        351 & $(\amal{D_2}{SO(2)}{}{}{})$ & $8$ &
        386 & $(\amal{D_1^2}{O(2)}{D_{1}}{\bz_1^2}{})$ & $4$ \\
      317 & $(\amal{\bz_4^2}{D_{2n}}{\bz_{2}}{\bz_4^z}{})$ & $4$ &
        352 & $(\amal{D_2^z}{SO(2)}{}{}{})$ & $8$ &
        387 & $(\amal{D_1^2}{O(2)}{D_{1}}{D_1}{})$ & $4$ \\
      318 & $(\amal{D_2^2}{D_{2n}}{\bz_{2}}{\bz_2^2}{})$ & $4$ &
        353 & $(\amal{D_2^d}{SO(2)}{}{}{})$ & $4$ &
        388 & $(\amal{D_1^2}{O(2)}{D_{1}}{D_1^d}{})$ & $4$ \\
      319 & $(\amal{D_2^2}{D_{2n}}{\bz_{2}}{D_2}{})$ & $4$ &
        354 & $(\amal{D_1^2}{SO(2)}{}{}{})$ & $4$ &
        389 & $(\amal{\bz_4}{O(2)}{D_{1}}{\bz_2}{})$ & $8$ \\
      320 & $(\amal{D_2^2}{D_{2n}}{\bz_{2}}{D_2^z}{})$ & $4$ &
        355 & $(\amal{\bz_4}{SO(2)}{}{}{})$ & $8$ &
        390 & $(\amal{\bz_4^z}{O(2)}{D_{1}}{\bz_2}{})$ & $8$ \\
      321 & $(\amal{D_2^2}{D_{2n}}{\bz_{2}}{D_2^d}{})$ & $2$ &
        356 & $(\amal{\bz_4^z}{SO(2)}{}{}{})$ & $8$ &
        391 & $(\amal{\tD_2}{O(2)}{D_{1}}{\bz_2}{})$ & $8$ \\
      322 & $(\amal{D_2^2}{D_{2n}}{\bz_{2}}{D_1^2}{})$ & $2$ &
        357 & $(\amal{\tD_2}{SO(2)}{}{}{})$ & $8$ &
        392 & $(\amal{\tD_2}{O(2)}{D_{1}}{\tD_1}{})$ & $4$ \\
      323 & $(\amal{D_4^\td}{D_{2n}}{\bz_{2}}{D_2^z}{})$ & $4$ &
        358 & $(\amal{\tD_2^z}{SO(2)}{}{}{})$ & $8$ &
        393 & $(\amal{\tD_2^z}{O(2)}{D_{1}}{\bz_2}{})$ & $8$ \\
      324 & $(\amal{D_4^\td}{D_{2n}}{\bz_{2}}{\bz_4^z}{})$ & $4$ &
        359 & $(\amal{\tD_2^\td}{SO(2)}{}{}{})$ & $4$ &
        394 & $(\amal{\tD_2^z}{O(2)}{D_{1}}{\tD_1^\td}{})$ & $4$ \\
      325 & $(\amal{D_4^\td}{D_{2n}}{\bz_{2}}{\tD_2}{})$ & $4$ &
        360 & $(\amal{\tD_1^2}{SO(2)}{}{}{})$ & $4$ &
        395 & $(\amal{\tD_2^\td}{O(2)}{D_{1}}{\bz_2^z}{})$ & $4$ \\
      326 & $(\amal{D_4}{D_{2n}}{\bz_{2}}{D_2}{})$ & $4$ &
        361 & $(\amal{\tD_2^2}{SO(2)}{}{}{})$ & $4$ &
        396 & $(\amal{\tD_2^\td}{O(2)}{D_{1}}{\tD_1}{})$ & $4$ \\
      327 & $(\amal{D_4}{D_{2n}}{\bz_{2}}{\bz_4}{})$ & $4$ &
        362 & $(\amal{\bz_4^2}{SO(2)}{}{}{})$ & $4$ &
        397 & $(\amal{\tD_2^\td}{O(2)}{D_{1}}{\tD_1^\td}{})$ & $4$ \\
      328 & $(\amal{D_4}{D_{2n}}{\bz_{2}}{\tD_2}{})$ & $4$ &
        363 & $(\amal{D_2^2}{SO(2)}{}{}{})$ & $4$ &
        398 & $(\amal{\tD_1^2}{O(2)}{D_{1}}{\bz_1^2}{})$ & $4$ \\
      329 & $(\amal{D_4^d}{D_{2n}}{\bz_{2}}{D_2}{})$ & $4$ &
        364 & $(\amal{D_4^\td}{SO(2)}{}{}{})$ & $4$ &
        399 & $(\amal{\tD_1^2}{O(2)}{D_{1}}{\tD_1}{})$ & $4$ \\
      330 & $(\amal{D_4^d}{D_{2n}}{\bz_{2}}{\bz_4^z}{})$ & $4$ &
        365 & $(\amal{D_4}{SO(2)}{}{}{})$ & $4$ &
        400 & $(\amal{\tD_1^2}{O(2)}{D_{1}}{\tD_1^\td}{})$ & $4$ \\
      331 & $(\amal{D_4^d}{D_{2n}}{\bz_{2}}{\tD_2^z}{})$ & $4$ &
        366 & $(\amal{D_4^d}{SO(2)}{}{}{})$ & $4$ &
        401 & $(\amal{\tD_2^2}{O(2)}{D_{1}}{\bz_2^2}{})$ & $4$ \\
      332 & $(\amal{D_4^z}{D_{2n}}{\bz_{2}}{D_2^z}{})$ & $4$ &
        367 & $(\amal{D_4^z}{SO(2)}{}{}{})$ & $4$ &
        402 & $(\amal{\tD_2^2}{O(2)}{D_{1}}{\tD_2}{})$ & $4$ \\
      333 & $(\amal{D_4^z}{D_{2n}}{\bz_{2}}{\bz_4}{})$ & $4$ &
        368 & $(\amal{D_4^2}{SO(2)}{}{}{})$ & $2$ &
        403 & $(\amal{\tD_2^2}{O(2)}{D_{1}}{\tD_2^z}{})$ & $4$ \\
      334 & $(\amal{D_4^z}{D_{2n}}{\bz_{2}}{\tD_2^z}{})$ & $4$ &
        369 & $(\amal{\bz_1^2}{O(2)}{D_{1}}{\bz_1}{})$ & $16$ &
        404 & $(\amal{\tD_2^2}{O(2)}{D_{1}}{\tD_2^\td}{})$ & $2$ \\
      335 & $(\amal{D_4^2}{D_{2n}}{\bz_{2}}{\tD_2^2}{})$ & $2$ &
        370 & $(\amal{\bz_2}{O(2)}{D_{1}}{\bz_1}{})$ & $16$ &
        405 & $(\amal{\tD_2^2}{O(2)}{D_{1}}{\tD_1^2}{})$ & $2$ \\
      336 & $(\amal{D_4^2}{D_{2n}}{\bz_{2}}{\bz_4^2}{})$ & $2$ &
        371 & $(\amal{\bz_2^z}{O(2)}{D_{1}}{\bz_1}{})$ & $16$ &
        406 & $(\amal{\bz_4^2}{O(2)}{D_{1}}{\bz_2^2}{})$ & $4$ \\
      337 & $(\amal{D_4^2}{D_{2n}}{\bz_{2}}{D_2^2}{})$ & $2$ &
        372 & $(\amal{\tD_1}{O(2)}{D_{1}}{\bz_1}{})$ & $8$ &
        407 & $(\amal{\bz_4^2}{O(2)}{D_{1}}{\bz_4}{})$ & $4$ \\
      338 & $(\amal{D_4^2}{D_{2n}}{\bz_{2}}{D_4^\td}{})$ & $2$ &
        373 & $(\amal{\tD_1^\td}{O(2)}{D_{1}}{\bz_1}{})$ & $8$ &
        408 & $(\amal{\bz_4^2}{O(2)}{D_{1}}{\bz_4^z}{})$ & $4$ \\
      339 & $(\amal{D_4^2}{D_{2n}}{\bz_{2}}{D_4}{})$ & $2$ &
        374 & $(\amal{D_1}{O(2)}{D_{1}}{\bz_1}{})$ & $8$ &
        409 & $(\amal{D_2^2}{O(2)}{D_{1}}{\bz_2^2}{})$ & $4$ \\
      340 & $(\amal{D_4^2}{D_{2n}}{\bz_{2}}{D_4^d}{})$ & $2$ &
        375 & $(\amal{D_1^d}{O(2)}{D_{1}}{\bz_1}{})$ & $8$ &
        410 & $(\amal{D_2^2}{O(2)}{D_{1}}{D_2}{})$ & $4$ \\
      341 & $(\amal{D_4^2}{D_{2n}}{\bz_{2}}{D_4^z}{})$ & $2$ &
        376 & $(\amal{\bz_2^2}{O(2)}{D_{1}}{\bz_1^2}{})$ & $8$ &
        411 & $(\amal{D_2^2}{O(2)}{D_{1}}{D_2^z}{})$ & $4$ \\
      342 & $(\amal{\bz_1}{SO(2)}{}{}{})$ & $32$ &
        377 & $(\amal{\bz_2^2}{O(2)}{D_{1}}{\bz_2}{})$ & $8$ &
        412 & $(\amal{D_2^2}{O(2)}{D_{1}}{D_2^d}{})$ & $2$ \\
      343 & $(\amal{\bz_1^2}{SO(2)}{}{}{})$ & $16$ &
        378 & $(\amal{\bz_2^2}{O(2)}{D_{1}}{\bz_2^z}{})$ & $8$ &
        413 & $(\amal{D_2^2}{O(2)}{D_{1}}{D_1^2}{})$ & $2$ \\
      344 & $(\amal{\bz_2}{SO(2)}{}{}{})$ & $16$ &
        379 & $(\amal{D_2}{O(2)}{D_{1}}{\bz_2}{})$ & $8$ &
        414 & $(\amal{D_4^\td}{O(2)}{D_{1}}{D_2^z}{})$ & $4$ \\
      345 & $(\amal{\bz_2^z}{SO(2)}{}{}{})$ & $16$ &
        380 & $(\amal{D_2}{O(2)}{D_{1}}{D_1}{})$ & $4$ &
        415 & $(\amal{D_4^\td}{O(2)}{D_{1}}{\bz_4^z}{})$ & $4$ \\
      346 & $(\amal{\tD_1}{SO(2)}{}{}{})$ & $8$ &
        381 & $(\amal{D_2^z}{O(2)}{D_{1}}{\bz_2}{})$ & $8$ &
        416 & $(\amal{D_4^\td}{O(2)}{D_{1}}{\tD_2}{})$ & $4$ \\
      347 & $(\amal{\tD_1^\td}{SO(2)}{}{}{})$ & $8$ &
        382 & $(\amal{D_2^z}{O(2)}{D_{1}}{D_1^d}{})$ & $4$ &
        417 & $(\amal{D_4}{O(2)}{D_{1}}{D_2}{})$ & $4$ \\
      348 & $(\amal{D_1}{SO(2)}{}{}{})$ & $8$ &
        383 & $(\amal{D_2^d}{O(2)}{D_{1}}{\bz_2^z}{})$ & $4$ &
        418 & $(\amal{D_4}{O(2)}{D_{1}}{\bz_4}{})$ & $4$ \\
      349 & $(\amal{D_1^d}{SO(2)}{}{}{})$ & $8$ &
        384 & $(\amal{D_2^d}{O(2)}{D_{1}}{D_1}{})$ & $4$ &
        419 & $(\amal{D_4}{O(2)}{D_{1}}{\tD_2}{})$ & $4$ \\
      350 & $(\amal{\bz_2^2}{SO(2)}{}{}{})$ & $8$ &
        385 & $(\amal{D_2^d}{O(2)}{D_{1}}{D_1^d}{})$ & $4$ &
        420 & $(\amal{D_4^d}{O(2)}{D_{1}}{D_2}{})$ & $4$ \\
    \bottomrule
  \end{tabular}
\end{table}

  \begin{table}[H]
    \centering
    \caption{Conjugacy Classes of Subgroups in $D_4\times\bz_2\times O(2)$ (part 5)}\label{tab5}
    \vskip 1em
    \begin{tabular}{|rcc|rcc|rcc|}
      \toprule
      ID & $(S)$ & $\abs{W(S)}$ &
        ID & $(S)$ & $\abs{W(S)}$ &
        ID & $(S)$ & $\abs{W(S)}$ \\
      \midrule
      421 & $(\amal{D_4^d}{O(2)}{D_{1}}{\bz_4^z}{})$ & $4$ &
        434 & $(\amal{\bz_1^2}{O(2)}{}{}{})$ & $8$ &
        447 & $(\amal{\bz_4^z}{O(2)}{}{}{})$ & $4$ \\
      422 & $(\amal{D_4^d}{O(2)}{D_{1}}{\tD_2^z}{})$ & $4$ &
        435 & $(\amal{\bz_2}{O(2)}{}{}{})$ & $8$ &
        448 & $(\amal{\tD_2}{O(2)}{}{}{})$ & $4$ \\
      423 & $(\amal{D_4^z}{O(2)}{D_{1}}{D_2^z}{})$ & $4$ &
        436 & $(\amal{\bz_2^z}{O(2)}{}{}{})$ & $8$ &
        449 & $(\amal{\tD_2^z}{O(2)}{}{}{})$ & $4$ \\
      424 & $(\amal{D_4^z}{O(2)}{D_{1}}{\bz_4}{})$ & $4$ &
        437 & $(\amal{\tD_1}{O(2)}{}{}{})$ & $4$ &
        450 & $(\amal{\tD_2^\td}{O(2)}{}{}{})$ & $2$ \\
      425 & $(\amal{D_4^z}{O(2)}{D_{1}}{\tD_2^z}{})$ & $4$ &
        438 & $(\amal{\tD_1^\td}{O(2)}{}{}{})$ & $4$ &
        451 & $(\amal{\tD_1^2}{O(2)}{}{}{})$ & $2$ \\
      426 & $(\amal{D_4^2}{O(2)}{D_{1}}{\tD_2^2}{})$ & $2$ &
        439 & $(\amal{D_1}{O(2)}{}{}{})$ & $4$ &
        452 & $(\amal{\tD_2^2}{O(2)}{}{}{})$ & $2$ \\
      427 & $(\amal{D_4^2}{O(2)}{D_{1}}{\bz_4^2}{})$ & $2$ &
        440 & $(\amal{D_1^d}{O(2)}{}{}{})$ & $4$ &
        453 & $(\amal{\bz_4^2}{O(2)}{}{}{})$ & $2$ \\
      428 & $(\amal{D_4^2}{O(2)}{D_{1}}{D_2^2}{})$ & $2$ &
        441 & $(\amal{\bz_2^2}{O(2)}{}{}{})$ & $4$ &
        454 & $(\amal{D_2^2}{O(2)}{}{}{})$ & $2$ \\
      429 & $(\amal{D_4^2}{O(2)}{D_{1}}{D_4^\td}{})$ & $2$ &
        442 & $(\amal{D_2}{O(2)}{}{}{})$ & $4$ &
        455 & $(\amal{D_4^\td}{O(2)}{}{}{})$ & $2$ \\
      430 & $(\amal{D_4^2}{O(2)}{D_{1}}{D_4}{})$ & $2$ &
        443 & $(\amal{D_2^z}{O(2)}{}{}{})$ & $4$ &
        456 & $(\amal{D_4}{O(2)}{}{}{})$ & $2$ \\
      431 & $(\amal{D_4^2}{O(2)}{D_{1}}{D_4^d}{})$ & $2$ &
        444 & $(\amal{D_2^d}{O(2)}{}{}{})$ & $2$ &
        457 & $(\amal{D_4^d}{O(2)}{}{}{})$ & $2$ \\
      432 & $(\amal{D_4^2}{O(2)}{D_{1}}{D_4^z}{})$ & $2$ &
        445 & $(\amal{D_1^2}{O(2)}{}{}{})$ & $2$ &
        458 & $(\amal{D_4^z}{O(2)}{}{}{})$ & $2$ \\
      433 & $(\amal{\bz_1}{O(2)}{}{}{})$ & $16$ &
        446 & $(\amal{\bz_4}{O(2)}{}{}{})$ & $4$ &
        459 & $(\amal{D_4^2}{O(2)}{}{}{})$ & $1$ \\
    \bottomrule
  \end{tabular}
\end{table}


\begin{example}\rm 
\label{ex:S4xO2} Recall that for given two groups $\mathscr G_1$ and $\mathscr G_2$, we have introduced the notation $\mathcal{H}=H^\varphi\times_L^\psi K$ for a subgroup $\mathcal{H}\subset \mathscr G_1\times\mathscr G_2$, where $H\subset \mathscr G_1$, $K\subset \mathscr G_2 $ are two subgroups, $\varphi:H\rightarrow L$ and $\psi:H\rightarrow L$ are epimorphismns. Consider as our main example the group $S_4\times O(2)$
and assume that $\mathcal{H}$ is a closed subgroup. The method explained in this section was applied by Hao-Pin Wu (cf \cite{Pin}) to the group $S_4\times O(2)$ using GAP programming who obtain the classification of the  all conjugacy classes $(\mathcal{H})$ of  closed subgroups in $S_4\times O(2)$.  To be more precise, in order to 
identify $L$ with $K/\text{Ker\,}(\psi)\subset O(2)$ and denote by $r$ the
rotation generator in $L$. Next we put 
\begin{align*}
Z=\text{Ker\,}(\varphi)\quad \text{ and }\quad R=\varphi^{-1}(\left<r\right>)
\end{align*}
and define 
\begin{equation}  \label{eq:amalg}
\mathcal{H}=:H\prescript{Z}{R}\times_{L}K,.
\end{equation}
Of course, in the case when all the epimorphisms $\varphi$ with the kernel $Z $ are conjugate, there is no need to use the symbol $Z$ in \eqref{eq:amalg}
and we will simply write $\mathcal{H}=H\prescript{Z}{}\times_{L}K$.
Moreover, in the case all epimorphisms $\varphi$ from $H$ to $L$ are
conjugate, we can also omit the symbol $L$, i.e. we will write $\mathcal{H}=H\prescript{}{}\times_{L}K$. The conjugacy classes of subgroups in $S_4\times
O(2)$ are listed Table \ref{tab: Tab1-s4}, which were obtained in \cite{Pin}.
Notice in Table Table \ref{tab: Tab1-s4} the classes 38, 39 and 40, 41, 42.
Here $<r>={\mathbb{Z}}_2$, so for the class 38, the isomorphism $\varphi:D_2\to D_2$ is the identity, while for the class 38, we have $\varphi(\kappa)=-1$. Similarly, for the class 40, the epimorphism maps $V_4$ onto ${\mathbb{Z}}_2$, for the class 41, $\varphi(\kappa)=-1$, and for class 42, $\varphi({\mathbb{Z}}_4)={\mathbb{Z}}_2$. 
\end{example}
\vs
\begin{example} \label{ex:D4xZ2xO2} \rm As another example that will be used later in this paper, we consider the group $\Gamma:=D_4\times \bz_2$. Since $D_4\times \bz_2\le D_4\times S^1$, we can use the same convention as in \cite{AED} in order to denote the conjugacy classes of subgroups in $D_4\times \bz_2$, namely
\begin{gather*}
 (\bz_1),\, (\bz_1^2),\,  (\bz_2),\, (\bz_2^z,\, (\tD_1),\, (\tD_1^\td),\, (D_1),\, (D_1^d),\, (\bz_2^2),\, D_2),
 (D_2^z),\, (D_2^d),\, (D_1^2),\, (\bz_4),\\
  (\bz_4^z),\, \tD_2),\, (\tD_2^z),\, (\tD_2^\td),\, (\tD_1^2),\, (\tD_2^2),(\bz_4^2),\, (D_2^2),\, (D_4^\td),\, (D_4),\, (D_4^d),\, (D_4^z),\,(D_4^2). 
\end{gather*}
By applying GAP programs developed by Hao-Pin Wu (cf \cite{Pin}) to the group $D_4\times \bz_2\times O(2)$ we get all the conjugacy classes of subgroups in $D_4\times \bz_2\times O(2)$ which are listed in Tables \ref{tab1}--\ref{tab5}.

\end{example}


\vs
\section{Computations of the Euler Ring $U(\Gamma\times O(2))$}\label{sec:UGO2}
Before we begin our discussion about the Euler ring $U(\Gamma\times O(2))$ we collect the related to its computation information. 
\vs 
\subsection{Twisted Subgroups and Related Modules}
\label{sec:A-mod} Let $\Gamma$ is a finite group and $G = \Gamma \times S^1$. In this case, there are exactly two types of subgroups $H \le G$, namely,
\begin{itemize}
\item[(a)] the product groups  $H = K \times S^1$ with $K$ being a subgroup of $\Gamma$;

\item[(b)] the so-called \textit{$\varphi$-twisted $l$-folded} subgroups $K^{\varphi,l}$ (in short, twisted subgroups) defined as follows: if $K$ is a
subgroup of $\Gamma$, $\varphi : K \to S^1$ a homomorphism and $l = 1,...$,
then 
\begin{equation*}
K^{\varphi,l} :=\{(\gamma,z) \in K \times S^1 : \; \varphi(\gamma) = z^l\}.
\end{equation*}
For a trivial homomorphism $\varphi: K\to S^1$, $\varphi(k)=1$ for all $k\in
K$, the instead of $K^{\varphi,l}$ we will write $K^l$, and if in addition $%
l=1$, then we will simply write $K$. Moreover, if $l=1$, then instead of $%
K^{\varphi,1}$ we will write $K^\varphi$. \vskip.3cm
\end{itemize}
Clearly $S^1$ can be identified with $SO(2)\subset O(2)$, then the twisted
subgroups $K^{\varphi,l} $ are also amalgamated subgroups and we have 
\begin{equation*}
K^{\varphi,l}=K\times ^\varphi_{\mathbb{Z}_k}\mathbb{Z}_n,
\end{equation*}
where $\varphi(K)=\mathbb{Z}_k$ and $n=k\cdot l$. 
\vs

\begin{table}
\begin{center}
\begin{tabular}{|c|c|c|c|c|}
\noalign{\vskip2pt\hrule\vskip2pt}
$H^{\varphi}$& $\Ker \varphi$ & $\Im \varphi$& $N(H^{\varphi})$& $W(H^{\varphi})$\\
\noalign{\vskip2pt\hrule\vskip2pt}
$S_4^l$ & $S_4$ & $\bz_1$ & $S_4\times S^1$ & $S^1$ \\
$A_4^l$ & $A_4$ & $\bz_1$ & $S_4\times S^1$ & $\bz_2\times S^1$ \\
$D_4^l$ & $D_4$ & $\bz_1$ & $D_4\times S^1$ & $S^1$ \\
$D_3^l$ & $D_3$ & $\bz_1$ & $D_3\times S^1$ & $S^1$ \\
$D_2^l$ & $D_2$ & $\bz_1$ & $D_4\times S^1$ & $\bz_2\times S^1$\\
$\bz_4^l$ & $\bz_4$ & $\bz_1$ & $D_4\times S^1$ & $\bz_2\times S^1$ \\
$V_4^l$ & $V_4$ & $\bz_1$ & $S_4\times S^1$ & $S_3\times S^1$  \\
$\bz_3$ & $\bz_3$ & $\bz_1$ & $D_3\times S^1$ & $\bz_2\times S^1$\\
$D_1^l$ & $D_1$ & $\bz_1$ & $D_2\times S^1$ & $\bz_2\times S^1$ \\
$\bz_2^l$ & $\bz_2$ & $\bz_1$ & $D_4\times S^1$ & $V_4\times S^1$ \\
$\bz_1^l$ & $\bz_1$ & $\bz_1$ & $S_4\times S^1$ & $S_4\times S^1$  \\
\noalign{\vskip2pt\hrule\vskip2pt}
$S_4^{-,l}$ &$A_4$& $\bz_2$ & $S_4\times S^1$ & $S^1$ \\
$D_4^{z,l}$ & $\bz_4$ &$\bz_2$ & $D_4\times S^1$ &$S^1$\\
$D_4^{d,l}$ & $D_2$ &$\bz_2$ & $D_4\times S^1$ &$S^1$\\
$D_4^{\hat d,l}$ & $V_4$ &$\bz_2$ & $D_4\times S^1$ &$S^1$\\
$A_4^{t,l}$  & $V_4$ & $\bz_3$ &$A_4\times S^1$& $S^1$\\
$D_3^{z,l}$ & $\bz_3$ & $\bz_2$ & $D_3\times S^1$& $S^1$\\
$D_2^{z,l}$ & $\bz_2$ &$ \bz_2$ &$D_4\times S^1$ &$\bz_2\times S^1$\\
$D_2^{d,l}$ & $D_1$ &$ \bz_2$ &$D_2\times S^1$ &$ S^1$\\
$V_4^{-,l}$ & $\bz_2$  & $\bz_2$ & $D_4\times S^1$ & $\bz_2\times S^1$\\
$D_1^{z,l}$ &$\bz_1$& $\bz_2$ &$D_2\times S^1$ & $\bz_2\times S^1$\\
$\bz^{-,l}_4$ &$\bz_2$ & $\bz_2$& $D_4\times S^1$ &$\bz_2\times S^1$ \\
$\bz^{c,l}_4$ &$\bz_1$ &$\bz_{4}$ & $\bz_4\times S^1$& $S^1$ \\
$\bz^{t,l}_3$ &$\bz_1$ &$\bz_3$ & $\bz_3\times S^1$& $S^1$ \\
$\bz_2^{-,l}$ &$\bz_1$ & $\bz_2$ & $D_4\times S^1$ &$V_4\times S^1$\\
\noalign{\vskip2pt\hrule\vskip2pt}
\end{tabular}\end{center}
\caption{Twisted subgroups in $S_4\times S^1$, where $\varphi : H \to S^1$ is a homomorphism $l\in \bn$.
}\label{tab:toct}
\end{table}


\begin{example}\rm
\label{ex:S1twist} Let us consider the group $S_4\times S^1$. For
the representatives of the conjugacy classes $K$ in $S_4$, we introduce
the following  homomorphisms: $\varphi: K\to S^1$: 
\begin{alignat*}{2}
z:&D_n\to {\mathbb{Z}}_2\subset S^1, \;\; \text{Ker\,} z={\mathbb{Z}}_n,
\;\; n=2,3,4,\quad & d:& D_{2m}\to {\mathbb{Z}}_2\subset S^1,\;\; \text{Ker\,%
} d=D_m,\;\; m=1,2, \\
\widetilde d:&D_4\to {\mathbb{Z}}_2\subset S^1,\;\; \text{Ker\,} \widetilde
d=V_4,\quad & c:&{\mathbb{Z}}_4\to {\mathbb{Z}}_4\subset S^1,\;\; \text{Ker\,%
} c={\mathbb{Z}}_1, \\
\nu:&{\mathbb{Z}}_{2m}\to {\mathbb{Z}}_2\subset S^1, \; \; \text{Ker\,} \nu={%
\mathbb{Z}}_m, m=1,2, \;\; & \nu:&V_4\to {\mathbb{Z}}_2\subset S^1, \;\; 
\text{Ker\,} \nu={\mathbb{Z}}_2, \\
\nu:&S_4\to {\mathbb{Z}}_2\subset S^1,\;\; \text{Ker\,} \nu=A_4,\quad & t:&{%
\mathbb{Z}}_3\to {\mathbb{Z}}_3\subset S^1,\;\;\text{Ker\,} t={\mathbb{Z}}_1,
\\
t&:A_4\to {\mathbb{Z}}_3\subset S^1,\;\; \text{Ker\,} t=V_4.\;\; & ~&~
\end{alignat*}
In what follows, instead of $K^{\nu,l}$ we will write $K^{-.,l}$, i.e. we
have the twisted subgroups ${\mathbb{Z}}_2^{-,l}$, ${\mathbb{Z}}_4^{-,l}$, $%
V_4^{-,l}$, and $S_4^{-,l}$. All the twisted subgroups in $S_4\times S^1$
are listed in Table \ref{tab:toct}. 
\end{example}


\vs
Put 
\begin{equation*}
\Phi_1^t(G):=\{(H)\in \Phi(G): \; H=K^{\varphi,l} \text{ for some $K\le
\Gamma$}\},
\end{equation*}
where $G:=\Gamma\times S^1$ and $\Gamma$ is a finite group.
One can easily observe that $\Phi_1(G)=\Phi_1^t(G)$ and $\Phi_0(G)$ can be
identified with $\Phi(\Gamma)$ (with $(K\times S^1)$ identified with $(K)$).
Therefore, the Euler ring $U(G)$, as a $\mathbb{Z}$-module is $A(\Gamma)\oplus A_1(G)$, where (by following the same  notations as in \cite{AED}) $A_1(G)$ stands for $\mathbb{Z}[\Phi_1(G)]$. It is well-known (see \cite{AED,survey}) that $A_1(G)$ admits a structure of $A(\Gamma)$-module. The
Burnside ring structure $A(\Gamma)$ as well as the $A(\Gamma)$-module structure for$A_1(G)$ were explicitly described for many specific groups $\Gamma$ (see \cite{AED}) and several computational routines were created for
these algebraic structures.

\vs The following result (cf. \cite{RR}) provides a more practical description of the multiplication in the Euler ring $U(G)$: 
\vs

\begin{proposition}
\label{pro:twisted-euler} Let $G=\Gamma\times S^1$ with $\Gamma$ being a
finite group. Then the multiplication `$*$' in the Euler ring $U(G)=\bz[\Phi(G)]=\bz[\Phi_0(G)]\oplus \bz[\Phi_1(G)]\simeq A(\Gamma)\oplus A_1(G)$ can be described by the following table 
\vs

\begin{tabular}{|c|c|c|}
\hline
$*$ & $A(\Gamma)$ & $A_1(G)$ \\ \hline
$A(\Gamma)$ & Burnside ring $A(\Gamma)$ multiplication & $A(\Gamma)$-module
multiplication \\ 
$A_1(G)$ & $A(\Gamma)$-module multiplication & $0$ \\ \hline
\end{tabular}
\end{proposition}

\vskip.3cm

\begin{example}\rm 
Let us discuss as an example, the group $G=S_4\times S^1$. The representatives of the twisted conjugacy classes of subgroups in $G:=S_4\times S^1$  are listed in Table \ref{tab:toct}, where the upper part of the table also provides the list of the representatives of the conjugacy classes of subgroup in $S_4$. Thus,  we have 
\begin{equation*}
\Phi(S_4):=\{ (S_4), (A_4),(D_4),(D_3),(V_4), (D_2), ({\mathbb{Z}}_3),(D_1),({\mathbb{Z}}_2),({\mathbb{Z}}_1)\},
\end{equation*}
and 
\begin{align*}
\Phi_1^t(S_4\times S^1):=&\{(S_4^l), (A_4^l),(D_4^l),(D_3^l),(V_4^l),
(D_2^l), ({\mathbb{Z}}_3^l), (D_1^l),({\mathbb{Z}}_2^l),({\mathbb{Z}}_1^l) \\
&(S_4^{-,l}),(D_4^{z,l}), (D_4^{\hat d,l}), (D_4^{d,l}), (A_4^{t,l}),
(D_3^{z,l}), (D_2^{z,l}),(D_2^{d,l}),(V_4^{-,l}),(D_1^{z,l}), \\
&({\mathbb{Z}}_4^{-,l}),({\mathbb{Z}}_4^{c,l}),({\mathbb{Z}}_3^{t,l}),({\mathbb{Z}}_2^{-,l}), l\in \mathbb{N}\},
\end{align*}
where $l$ is used to indicate that the twisted subgroup is $l$-folded. Therefore, the
multiplication table for the Euler ring $U(S_4\times S^1)$ can be easily established from Proposition \ref{pro:twisted-euler} using  the $A(S_4)$-module multiplication table  for $A_1^t(S_4\times S^1)$ (see \cite{AED}, Table 6.14, which also contains the Burnside ring $A(S_4)$).
\end{example}

\subsection{Multiplication in  the Euler Ring $U(\Gamma\times O(2))$}\label{sec:mUGO2}

Assume that $\Gamma$ is a finite group and consider the group $G:=\Gamma\times O(2)$. Our goal is to describe the multiplicative structure
of the Euler ring  $U(\Gamma\times O(2))$. More precisely, we
are interested in finding algorithms allowing the computation of the
products $(H)*(K)$ for $(H)$, $(K)\in \Phi(\Gamma\times O(2)$. We will be
using following information that is already available to us:

\begin{itemize}
\item[(a)] The  Burnside Ring $A(\Gamma\times O(2))$ is a
part of the Euler ring $U(\Gamma\times O(2))$. In fact, the computer
algorithms implemented in the software created by  Hao-Pin Wu
(cf. \cite{Pin}) allow an effective computations of the full multiplication
table for groups of type $\Gamma\times O(2))$. The size of such a table can
be significant. For example, in the case of $G:=S_4\times O(2)$, this table
has 5625 entries.

\item[(b)] The multiplicative structure of the Euler ring $U(\Gamma\times
S^1)$, which was described in Proposition \ref{pro:twisted-euler}.

\end{itemize}

Denote by $\Psi :U(\Gamma\times O(2))\to U(\Gamma\times S^1)$the Euler ring homomorphism  induced by  the natural embedding $\Gamma\times SO(2)\hookrightarrow \Gamma\times O(2)$ (here  $SO(2)\simeq S^{1}$). \vs

Therefore, for the group $G:=\Gamma\times O(2)$ we have $\Phi_0(G)=\Phi_0^{I}(G)\cup \Phi_0^{II}(G)\cup\Phi_0^{III}$, where

\begin{align*}
\Phi _{0}^{\text{I}}(G)& :=\Phi _{0}(\Gamma \times SO(2)); \\
\Phi _{0}^{\text{II}}(G)& :=\{(K{^{\varphi }\times _{L}}D_{n}):n\in \mathbb{N%
}\}; \\
\Phi _{0}^{\text{III}}(G)& :=\{(K\times _{L}O(2)):L:=\mathbb{Z}_{1},\,%
\mathbb{Z}_{2}\};
\end{align*}
On the other hand, we have $\Phi_1(G)=\Phi _{1}(\Gamma \times SO(2))$.
\vs

In order to describe the Euler ring structure in $U(\Gamma \times O(2))$ we
will apply the Euler ring homomorphism 
\begin{equation*}
\Psi :U(\Gamma \times O(2))\rightarrow U(\Gamma \times S^{1})
\end{equation*}%
induced by the natural inclusion $\psi :\Gamma \times S^{1}\rightarrow
\Gamma \times O(2)$ and use the known Euler ring structure of $U(\Gamma
\times S^{1})$ (see Proposition \ref{pro:twisted-euler}) to describe the
multiplicative structure of $U(\Gamma \times O(2))$. By direct application
of the definition of the homomorphism $\Psi $, we obtain that 
\begin{equation}
\Psi (\mathcal{H})=\begin{cases}
2(H_o) &\text{ if } H\le \Gamma\times SO(2)\\
(H_o) &\text{ if } H\not\le \Gamma\times SO(2),
\end{cases}
\end{equation}
where $H_o:=H\cap (\Gamma\times SO(2))$ is identified with a subgroup of $\Gamma\times S^1$.
\vs
We have the following straight forward but still  important result:
\vs
\begin{theorem}\label{th:UGO2}
Assume that $\Gamma$ is a finite group and $G:=\Gamma\times O(2)$. Then the multiplication table for the Euler ring $U(\Gamma\times O(2))$ can be completely determined from the multiplication table for the Burnside ring $A(G)$ and the  $A(\Gamma)$-module structure $A_1^t(\Gamma\times S^1)$ via the Euler ring homomorphism $\Psi : U(\Gamma\times O(2))\to U(\Gamma\times S^1)$. More precisely, for $(\mathcal H)$, $(\mathcal K)\in \Phi(G)$, we can write the product of $(\mathcal H)$ and $(\mathcal K)$ as the following sum 
\[
(\mathcal H)*(\mathcal K)=\sum_{(\mathcal L)\in \Phi_0(G)} n_{\mathcal L} \, (\mathcal L)+ \sum_{(\mathcal L')\in \Phi_1(G)} x_{\mathcal L'}\, (\mathcal L'),
\]
where $n_{\mathcal L}$ are known coefficients, while $x_{\mathcal L'}$ are unknown coefficients that can be determined from the following relation
\begin{equation}\label{eq:relUG}
 \sum_{(\mathcal L')\in \Phi_1(G)} x_{\mathcal L'}\, (\mathcal L')=\frac 12\left(\Psi(\mathcal H)*\Psi(\mathcal K)-\sum_{(\mathcal L)\in \Phi_0(G)} n_{\mathcal L} \, \Psi(\mathcal L)\right).
\end{equation} 
and $(\mathcal L')\in \Phi_1(\Gamma \times SO(2))$.
\end{theorem}
\begin{proof}
Denote by $\mathscr M$ the $\bz$-submodule of $U(G)$ generated by $\Phi_1(G)$, i.e. $\mathscr M=\bz[\Phi_1(\Gamma\times SO(2))]$. Clearly, the restriction of the homomorphism $\Psi$ to $\mathscr M$ is given by $\Psi(H)=2(H)$ for all $(H)\in \Phi_1(\Gamma\times SO(2))$. Therefore, $\Psi:\mathscr M\to U(\Gamma\times S^1)$ is a monomorphism of $\bz$-modules. Indeed,  by Theorem  \ref{th:conj} two subgroups $H$, $K\le \Gamma\times SO(2)$ are conjugate in $\Gamma\times SO(2)$ if and only if they are conjugate in $\Gamma\times O(2)$. 
Next,
by the definition of the Euler ring
multiplication, for any $(\mathcal{H})$, $(\mathcal{K})\in \Phi(\Gamma\times
O(2))$ we have 
\begin{equation*}
(\mathcal{H})*(\mathcal{K})=\sum _{(\mathcal{L})\in \Phi(G)} n_{\mathcal{L}} (\mathcal{L})
\end{equation*}
where $\mathcal{L}=g\mathcal{H }g^{-1}\cap \mathcal{K}$ for
some $g\in \Gamma\times O(2)$. 
Since $\mathcal L\le \mathcal H$ implies dim$\,W(\mathcal L)\ge \text{dim\,}W(\mathcal H)$, if $(\mathcal H)\in \Phi_1(G)$ and  $\mathcal{L}=g\mathcal{H }g^{-1}\cap \mathcal{K}$, then  
dim$\,W(\mathcal L)=1$, i.e. $(\mathcal L)\in \Phi_1(G)$.  
\vs
On the other hand, we have for  $(\mathcal{H})$, $(\mathcal{K})\in \Phi_0(\Gamma\times
O(2))$ the following Burnside ring multiplication
\begin{equation*}  \label{eq:Burn-prod}
(\mathcal{H})\cdot(\mathcal{K})=\sum _{(\mathcal{L})\in \Phi_0(G)} n_{\mathcal{L}} (\mathcal{L})
\end{equation*}
where all the coefficients $n_{\mathcal L}$ can be evaluated using  the algorithmic formulae given in  \cite{AED}, and 
\begin{equation}  \label{eq:Euler-prod}
(\mathcal{H})*(\mathcal{K})=\sum _{(\mathcal{L})\in \Phi_0(G)} n_{\mathcal{L}} (\mathcal{L})+\sum _{(\mathcal{L'})\in \Phi_1(G)} x_{\mathcal{L'}} (\mathcal{'L}),
\end{equation}
where the coefficients $x_{\mathcal L'}$ are to be determined. Then, by applying the Euler ring homomorphism $\Psi$ to \eqref{eq:Euler-prod} we obtain 
\begin{align*}
\Psi(\mathcal{H})*\Psi(\mathcal{K})&=\Psi\Big((\mathcal{H})*(\mathcal{K})   \Big)\\
&=\Psi\left(\sum _{(\mathcal{L})\in \Phi_0(G)} n_{\mathcal{L}} (\mathcal{L})+\sum _{(\mathcal{L'})\in \Phi_1(G)} x_{\mathcal{L'}} (\mathcal{'L})   \right)\\
&=\sum _{(\mathcal{L})\in \Phi_0(G)} n_{\mathcal{L}} \Psi(\mathcal{L})+\sum _{(\mathcal{L'})\in \Phi_1(G)} x_{\mathcal{L'}}\Psi (\mathcal{'L})\\
&=\sum _{(\mathcal{L})\in \Phi_0(G)} n_{\mathcal{L}} \Psi(\mathcal{L})+2\sum _{(\mathcal{L'})\in \Phi_1(G)} x_{\mathcal{L'}}  (\mathcal{'L}).
\end{align*}
Consequently, the formula \eqref{eq:relUG} follows.
\end{proof}
\vs

\begin{example}\rm 
(a)  Suppose $(H\times SO(2))$, $(K\times SO(2))\in \Phi_0(\Gamma\times
O(2))$, then clearly 
\begin{equation*}
(H\times SO(2))*(K\times SO(2))=(H\times SO(2))\cdot (K\times SO(2)).
\end{equation*}
Indeed, since for all $g\in \Gamma\times O(2)$, we have for some $L\subset
\Gamma$ that $gH\times SO(2)g^{-1}\cup K\times SO(2)=L\times SO(2) $, it
follows that all the conjugacy classes $(\mathcal{L})$ in \eqref{eq:Euler-prod} (for $\mathcal{H}:=H\times SO(2)$ and $\mathcal{K}:=K\times SO(2)$) belong to $\Phi_0(\Gamma\times O(2))$. Consequently, for
the elements $(H\times SO(2))$ and $(K\times SO(2))$ from $%
\Phi_0(\Gamma\times O(2))$, their Euler ring product coincides with their
Burnside ring product.
\vs
\noi
(b) Consider now $(H\times SO(2))\in \Phi_0^{\text{I}}(\Gamma\times
O(2))$ and $(K^{\varphi,l})\in \Phi_1(\Gamma\times SO(2)) $. Then
the product of these two elements in $U(\Gamma\times O(2))$ is given by 
\begin{equation*}
(H\times SO(2))*(K^{\varphi,l})=\sum_{(\mathcal{L'})\in \Phi_1(G)} x_{\mathcal{L'}}(\mathcal{L}).
\end{equation*}
Notice that that the conjugacy classes $(\mathcal{L'})$ are such that for
some $\gamma\in \Gamma$ 
\begin{equation*}
\mathcal{L'}=\gamma H\gamma^{-1}\times SO(2)\cap K^{\varphi,l} =
L^{\varphi,n},\quad L:=\gamma H\gamma^{-1}\cap K.
\end{equation*}
That implies $(\mathcal{L})\in \Phi_1^I(\Gamma\times O(2))$ and by applying
the homomorphism $\Psi$ we have 
\begin{align*}
\Psi\Big((H\times SO(2))*(K^{\varphi,l})\Big)&=\Psi\left(\sum_{(L)\in
\Phi(\Gamma)} x_L(L^{\varphi,l}) \right) \\
&=\sum_{(L)\in \Phi(\Gamma)} x_L \Psi(L^{\varphi,l}) \\
&=\sum_{(L)\in \Phi(\Gamma)} 2x_L (L^{\varphi,l}).
\end{align*}
On the other hand, since $\Psi$ is a ring homomorphism, we have 
\begin{align*}
\Psi\Big((H\times SO(2))*(K^{\varphi,l})\Big)&=\Psi(H\times
SO(2))*\Psi(K^{\varphi,l}) \\
&=4(H\times SO(2))*(K^{\varphi,l}) \\
&=\sum_{(L)\in \Phi(\Gamma)} 4m_L (L^{\varphi,l}),
\end{align*}
where 
\begin{equation*}
(H\times SO(2))\circ (K^{\varphi,l})= \sum_{(L)\in \Phi(\Gamma)} m_L
(L^{\varphi,l}),
\end{equation*}
is the multiplication in $A(\Gamma)$-module $A_1(\Gamma\times S^1)$.
Therefore, we obtain that $x_L=2m_L$, and the multiplication of such
conjugacy classes in $U(\Gamma\times O(2))$ coincides with the products in $%
A(\Gamma)$-module $A_1(\Gamma\times S^1)$ multiplied by 2.
\vs

\noi (c) If $(H^{\varphi,l})$, $({H^{\prime }}^{n^{\prime},l^{\prime
}})\in \Phi_1(\Gamma\times SO(2))$, then we claim that $(H^{\varphi,n})*({H^{\prime }}^{\varphi^{\prime },l^{\prime }})=0$. Indeed,
since 
\begin{equation*}
(H^{\varphi,l})*({H^{\prime }}^{\varphi^{\prime },l^{\prime
}})=\sum_{i=1}^kx_i(H_i^{\varphi_i,l_i}),
\end{equation*}
then by applying the homomorphism $\Psi$ we obtain 
\begin{align*}
0&=4(H^{\varphi,l})*({H^{\prime }}^{\varphi^{\prime
},l^{\prime}})=\Psi(H^{\varphi,l})*\Psi ({H^{\prime }}^{\varphi^{\prime
},l^{\prime }}) \\
&=\Psi((H^{\varphi,l})*({H^{\prime }}^{\varphi^{\prime
},l^{\prime}}))=\Psi\left(\sum_{i=1}^k x_i(H_i^{\varphi_i,l_i})\right) \\
&=2\sum_{i=1}^k x_i(H_i^{\varphi_i,l_i}),
\end{align*}
which implies $x_i=0$ for all $i$.

\vs
\noi (d) Let $(\mathcal{H})$, $(\mathcal{K})\in \Phi^{II}
_0(\Gamma\times O(2))$. Then, the space 
\begin{equation*}
\mathcal{Y}:=\frac{\Gamma\times O(2) }{\mathcal{H}}\times \frac{\Gamma\times
O(2)}{\mathcal{K}}
\end{equation*}
has exactly two types of orbits $(\mathcal{L})\in \Phi_0^{\text{II}}$ and $(%
\mathcal{L}^{\prime })\in \Phi_1^{\text{I}}$. One can easily notice that if $%
(\mathcal{L})$ is an orbit type in $\mathcal{Y}$ which is also from $%
\Phi_0^{II}$, then $(\mathcal{L}^{\prime })$, where $\mathcal{L}^{\prime }:=%
\mathcal{L}\cap \Gamma\times SO(2)$, is also an orbit type in $\mathcal{Y}$.
Moreover, $\Psi(\mathcal{L})=(\mathcal{L}^{\prime })$. On the other hand we
have the following product in the Burnside ring $A(\Gamma\times O(2))$ 
\begin{equation*}
(\mathcal{H})\cdot (\mathcal{K})=\sum_{\mathcal{L}} n_{\mathcal{L}}(\mathcal{%
L}),
\end{equation*}
where $(\mathcal{L})\in \Phi_0^{\text{II}}(\Gamma\times O(2))$. Since these
coefficients $n_{\mathcal{L}}$ are the same for the Euler ring
multiplication, we have 
\begin{equation*}
(\mathcal{H})* (\mathcal{K})=\sum_{\mathcal{L}} n_{\mathcal{L}}(L)+\sum_{%
\mathcal{L}^{\prime }} m_{\mathcal{L}^{\prime }}(\mathcal{L}^{\prime }),
\end{equation*}
where $(\mathcal{L}^{\prime \text{I}}_1(\Gamma\times O(2))$ and $m_{\mathcal{%
L}^{\prime }}$ are unknown coefficients. Next, by applying the Euler ring
homomorphism $\Psi$, we obtain 
\begin{align*}
0&= \Psi(\mathcal{H})* \Psi(\mathcal{K})= \Psi\Big((\mathcal{H})* (\mathcal{K%
})\Big) \\
&=\Psi\left(\sum_{\mathcal{L}} n_{\mathcal{L}}(\mathcal{L})+\sum_{\mathcal{L}%
^{\prime }} m_{\mathcal{L}^{\prime }}(\mathcal{L}^{\prime })\right) \\
&=\sum_{\mathcal{L}} n_{\mathcal{L}}\Psi (\mathcal{L})+\sum_{\mathcal{L}%
^{\prime }} m_{\mathcal{L}^{\prime }}\Psi(\mathcal{L}^{\prime }) \\
&=\sum_{\mathcal{L}} n_{\mathcal{L}}\Psi (\mathcal{L})+\sum_{\mathcal{L}%
^{\prime }} 2m_{\mathcal{L}^{\prime }}(\mathcal{L}^{\prime }).
\end{align*}
Since for each orbit type $(\mathcal{L}^{\prime })$ in $\mathcal{Y}$ there
exists an orbit type $(\mathcal{L})$ in $\mathcal{Y}$, such that $\Psi(%
\mathcal{L})=(\mathcal{L}^{\prime })$, the above relations allow us to
evaluate each $m_{\mathcal{L}^{\prime }}$ in a unique way.
\vs 
\noi
(e) Suppose now that $(\mathcal{H})\in \Phi_1(\Gamma\times
O(2))$ and $(\mathcal{K})\in \Phi^{\text{II}}_0(\Gamma\times O(2))$.
Clearly, all the orbit types $(\mathcal{L}^{\prime })$ in $\mathcal{Y}:=\frac{\Gamma\times O(2) }{\mathcal{H}}\times \frac{\Gamma\times O(2)}{\mathcal{K}}$ belong to $\Phi_1(\Gamma\times O(2)$, thus we have 
\begin{equation*}
(\mathcal{H})*(\mathcal{K})=\sum_{\mathcal{L}^{\prime }}x_{\mathcal{L}^{\prime }}(\mathcal{L}^{\prime }),
\end{equation*}
where $n_{\mathcal{L}^{\prime }}$ are unknown coefficients, thus by applying
the Euler ring homomorphism $\Psi$, we obtain 
\begin{align*}
0&=\Psi(\mathcal{H})*\Psi(\mathcal{K})=\Psi\Big((\mathcal{H})*(\mathcal{K})\Big) \\
&=\Psi\left(\sum_{\mathcal{L}^{\prime }}x_{\mathcal{L}^{\prime }}(\mathcal{L}^{\prime })\right) \\
&=\sum_{\mathcal{L}^{\prime }}x_{\mathcal{L}^{\prime }}\Psi(\mathcal{L}^{\prime }) \\
&=\sum_{\mathcal{L}^{\prime }}x_{\mathcal{L}^{\prime }}2(\mathcal{L}^{\prime}),
\end{align*}
which implies that all $x_{\mathcal{L}^{\prime }}=0$.
\vs

\noi(f) Assume now S that $(\mathcal{H})\in \Phi^{\text{I}}_0(\Gamma\times O(2))$ and $(\mathcal{K})\in \Phi^{\text{II}}_0(\Gamma\times O(2))$. Since all the orbit types $(\mathcal{L}^{\prime })$
in $\mathcal{Y}:=\frac{\Gamma\times O(2) }{\mathcal{H}}\times \frac{\Gamma\times O(2)}{\mathcal{K}}$ belong to $\Phi_1(\Gamma\times O(2))$, thus we have 
\begin{equation*}
(\mathcal{H})*(\mathcal{K})=\sum_{\mathcal{L}^{\prime }}x_{\mathcal{L}^{\prime }}(\mathcal{L}^{\prime }),
\end{equation*}
where $x_{\mathcal{L}^{\prime }}$ are unknown coefficients, thus by applying
the Euler ring homomorphism $\Psi$, we obtain 
\begin{align*}
2(\mathcal{H})*\Psi(\mathcal{K})&=\Psi(\mathcal{H})*\Psi(\mathcal{K})=\Psi\Big((\mathcal{H})*(\mathcal{K})\Big) \\
&=\Psi\left(\sum_{\mathcal{L}^{\prime }}x_{\mathcal{L}^{\prime }}(\mathcal{L}^{\prime })\right) \\
&=\sum_{\mathcal{L}^{\prime }}x_{\mathcal{L}^{\prime }}\Psi(\mathcal{L}^{\prime }) \\
&=\sum_{\mathcal{L}^{\prime }}x_{\mathcal{L}^{\prime }}2(\mathcal{L}^{\prime
}).
\end{align*}
Put $\mathcal{H}=H\times SO(2)$ and $\mathcal{K}=K^{\theta,n}$. Then, 
\begin{equation*}
(\mathcal{H})*\Psi(\mathcal{K})=(H\times SO(2))\circ (K^{\theta,n})=\sum_{L}
m_{L}(L^{\theta,n}),
\end{equation*}
where the last product is taken in the $A(\Gamma)$-module $A_1(\Gamma\times
S^1)$. Consequently, we obtain the relations 
\begin{equation*}
\sum_{L} 2m_{L}(L^{\theta,n})=\sum_{\mathcal{L}^{\prime }}x_{\mathcal{L}^{\prime }}2(\mathcal{L}^{\prime }).
\end{equation*}
Notice, that the orbit types $(\mathcal{L}^{\prime })$ are also of the type $(L^{\theta,n})$ for some $L\subset \Gamma$. Therefore, the above relations
can be solved for $x_{\mathcal{L}^{\prime }}$.
\vs
\noi
(g) Since for $(\mathcal{H})$), $(\mathcal{H})\in \Phi^{\text{III}}_0(\Gamma\times O(2)$, all the orbit types $(\mathcal{L})$ in $\mathcal{Y}:=\frac{\Gamma\times O(2) }{\mathcal{H}}\times \frac{\Gamma\times O(2)}{\mathcal{K}}$ ether belong to $\Phi^{\text{II}}_0(\Gamma\times O(2))$, or it
is $(\mathcal{L}^{\prime })$, where $\mathcal{L}^{\prime }:=\mathcal{L}\cap
\Gamma\times SO(2)$, i.e. $(\mathcal{L}^{\prime })\in \Phi_1(\Gamma\times O(2))$. Thus we have 
\begin{equation*}
(\mathcal{H})*(\mathcal{K})=(\mathcal{H})\cdot(\mathcal{K}),
\end{equation*}
i.e. the product $(\mathcal{H})*(\mathcal{K})$ coincide with the product $(\mathcal{H})\cdot(\mathcal{K})$ in the Burnside ring $A(\Gamma\times )(2))$.
\vs

\noi (h) Suppose now that $(\mathcal{H})\in \Phi^{\text{III}}_0(\Gamma\times O(2))$ and $(\mathcal{K})\in \Phi^{\text{II}}_0(\Gamma\times O(2))$. Clearly, all the orbit types in $\mathcal{Y}:=\frac{\Gamma\times O(2) }{\mathcal{H}}\times \frac{\Gamma\times O(2)}{\mathcal{K}}$
belong to either to $\Phi^{\text{II}}_0(\Gamma\times O(2)$ or $\Phi^{\text{I}}_1(\Gamma\times O(2)$, thus we have 
\begin{equation*}
(\mathcal{H})*(\mathcal{K})=\sum_{\mathcal{L}}n_{\mathcal{L}}(\mathcal{L})+
\sum_{\mathcal{L}^{\prime }}x_{\mathcal{L}^{\prime }}(\mathcal{L}^{\prime }),
\end{equation*}
where the coefficients $n_{\mathcal{L}}$ are obtain from the product formula
in the Burnside ring $A(\Gamma\times O(2))$ and the coefficients $x_{\mathcal{L}^{\prime }}$ are unknown. By applying the Euler ring homomorphism 
$\Psi$ we obtain 
\begin{align*}
\Psi(\mathcal{H})*\Psi(\mathcal{K})&=\Psi\Big((\mathcal{H})*(\mathcal{K})\Big) \\
&=\Psi\left(\sum_{\mathcal{L}}n_{\mathcal{L}}(\mathcal{L})+\sum_{\mathcal{L}^{\prime }}x_{\mathcal{L}^{\prime }}(\mathcal{L}^{\prime })\right) \\
&=\sum_{\mathcal{L}}n_{\mathcal{L}}\Psi(\mathcal{L})+\sum_{\mathcal{L}^{\prime }}x_{\mathcal{L}^{\prime }}\Psi(\mathcal{L}^{\prime }) \\
&=\sum_{\mathcal{L}}n_{\mathcal{L}}\Psi(\mathcal{L})+\sum_{\mathcal{L}^{\prime }}x_{\mathcal{L}^{\prime }}2(\mathcal{L}^{\prime }).
\end{align*}
On the other hand if $\mathcal{H}=H\times SO(2)$ and $\Psi(\mathcal{K})=(K^{\theta,n})$, then 
\begin{equation*}
\Psi(\mathcal{H})*\Psi(\mathcal{K})=(H\times SO(2))\circ
(K^{\theta,n})=\sum_{L} k_{L}(L^{\theta,n}),
\end{equation*}
where the last product is taken in the $A(\Gamma)$-module $A_1(\Gamma\times
S^1)$. Consequently, we obtain the relations 
\begin{equation*}
\sum_{L} 2k_{L}(L^{\theta,n})=\sum_{\mathcal{L}}n_{\mathcal{L}}\Psi(\mathcal{L})+\sum_{\mathcal{L}^{\prime }}x_{\mathcal{L}^{\prime }}2(\mathcal{L}^{\prime }),
\end{equation*}
which can be solved with respect to $x_{\mathcal{L}^{\prime }}$.
\vs 
\noi (i) Assume now $(\mathcal{H})\in \Phi^{\text{III}}_0(\Gamma\times
O(2))$ and $(\mathcal{K})\in \Phi_1(\Gamma\times O(2))$. Then,
all the orbit types $(\mathcal{L}^{\prime })$ in $\mathcal{Y}:=\frac{\Gamma\times O(2) }{\mathcal{H}}\times \frac{\Gamma\times O(2)}{\mathcal{K}}$
belong to $\Phi_1(\Gamma\times O(2)$, thus we have 
\begin{equation*}
(\mathcal{H})*(\mathcal{K})=\sum_{\mathcal{L}^{\prime }}x_{\mathcal{L}^{\prime }}(\mathcal{L}^{\prime }),
\end{equation*}
where the coefficients $x_{\mathcal{L}^{\prime }}$ are unknown. By applying
the Euler ring homomorphism $\Psi$ we obtain 
\begin{align*}
\Psi(\mathcal{H})*\Psi(\mathcal{K})&=\Psi\Big((\mathcal{H})*(\mathcal{K})\Big) \\
&=\Psi\left(\sum_{\mathcal{L}^{\prime }}x_{\mathcal{L}^{\prime }}(\mathcal{L}^{\prime })\right) \\
&=\sum_{\mathcal{L}^{\prime }}x_{\mathcal{L}^{\prime }}\Psi(\mathcal{L}^{\prime }) \\
&=\sum_{\mathcal{L}^{\prime }}x_{\mathcal{L}^{\prime }}2(\mathcal{L}^{\prime
}).
\end{align*}
On the other hand if $\mathcal{H}=H\times SO(2)$ and $\Psi(\mathcal{K})=(K^{\theta,n})$, then 
\begin{equation*}
\Psi(\mathcal{H})*\Psi(\mathcal{K})=2(H\times SO(2))\circ (K^{\theta,n})=2\sum_{L} k_{L}(L^{\theta,n}),
\end{equation*}
where the last product is taken in the $A(\Gamma)$-module $A_1(\Gamma\times
S^1)$. Consequently, we obtain the relations 
\begin{equation*}
\sum_{L} 2k_{L}(L^{\theta,n})=\sum_{\mathcal{L}^{\prime }}x_{\mathcal{L}
^{\prime }}2(\mathcal{L}^{\prime }),
\end{equation*}
which can be solved with respect to $x_{\mathcal{L}^{\prime }}$.
\vs
\noi
(j) Finally, assume $(\mathcal{H})\in \Phi^{\text{III}
}_0(\Gamma\times O(2))$ and $(\mathcal{K})\in \Phi^{\text{I}}_0(\Gamma\times
O(2))$. Then, all the orbit types $(\mathcal{L})$ in $\mathcal{Y}:=\frac{
\Gamma\times O(2) }{\mathcal{H}}\times \frac{\Gamma\times O(2)}{\mathcal{K}}$
belong to $\Phi^{\text{I}}_0(\Gamma\times O(2)$, thus we have 
\begin{equation*}
(\mathcal{H})*(\mathcal{K})=\sum_{\mathcal{L}}n_{\mathcal{L}}(\mathcal{L}),
\end{equation*}
where the coefficients $n_{\mathcal{L}}$ are obtain from the product formula
in the Burnside ring $A(\Gamma\times O(2))$ .
\end{example}

\
\vskip.3cm

\section{Properties of $G$-Equivariant Gradient Degree}

In what follows we assume that $G$ is a compact Lie group. In particular we
will be interested in the case when $G=\Gamma\times O(2)$, where $\Gamma$ is
a finite group.

\subsection{$G$-Equivariant Degree Without Free Parameter}

Assume that $V$ is an orthogonal $G$-representation and $\Omega\subset V$ an
open bounded $G$-invariant set. A $G$-equivariant (continuous) map $f:V\to V$
is called \textit{$\Omega$-admissible} if $f(x)\not=0$ for $x\in \partial
\Omega$. Then the pair $(f,\Omega)$ is called \textit{$G$-admissible}. We
denote by $\mathcal{M}^G(V,V)$ the set of all such admissible $G$-pairs, and
put $\mathcal{M}^G:=\bigcup_V \mathcal{M}^{G}(V,V)$, where $V$ is an
orthogonal $G$-representation. We have the following result (see \cite{AED}):

\begin{theorem}
\label{thm:GpropDeg} There exists a unique map $G$-$\deg:\mathcal{M}^G\to
A(G)$, which assigns to every admissible $G$-pair $(f,\Omega )$ an element $%
G $-$\deg(f,\Omega )\in A(G)$, called the \textit{$G$-equivariant degree (or
simply $G$-degree)} of $f$ on $\Omega$, 
\begin{equation}  \label{eq:G-deg0}
G\text{{\rm -deg}}(f,\Omega)=\sum_{(H_i)\in \Phi_0(G)}
n_{H_i}(H_i)=n_{H_1}(H_1)+\dots + n_{H_m}(H_m),
\end{equation}
satisfying the following properties:

\begin{itemize}
\item[($G$1)] \textbf{(Existence)} If $G$-$\deg\,(f,\Omega )\ne 0$, i.e.
there is in \eqref{eq:G-deg0} a non-zero coefficient $n_{H_i}$, then $%
\exists_{x\in\Omega }$ such that $f(x)=0$ and $(G_x)\geq (H_i)$.

\item[($G$2)] \textbf{(Additivity)} Let $\Omega _1$ and $\Omega _2$ be two
disjoint open $G$-invariant subsets of $\Omega $ such that $f^{-1}(0)\cap
\Omega \subset \Omega _1\cup \Omega _2.$ Then, 
\begin{equation*}
G\text{{\rm -deg}}(f,\Omega )=G\text{{\rm -deg}}(f,\Omega _1) + G\text{%
{\rm -deg}}(f,\Omega _2).
\end{equation*}

\item[($G$3)] \textbf{(Homotopy)} If $h:[0,1]\times V\to V$ is an $\Omega $%
-admissible $G$-homotopy, then 
\begin{equation*}
G\text{{\rm -deg}}(h_t,\Omega) =\;\text{\textit{constant}}.
\end{equation*}

\item[($G$4)] \textbf{(Normalization)} Let $\Omega $ be a $G$-invariant open
bounded neighborhood of $0$ in $V$. Then, 
\begin{equation*}
G\text{{\rm -deg}}(\text{{\rm Id\,}},\Omega ) =(G).
\end{equation*}

\item[($G$5)] \textbf{(Multiplicativity)} For any $(f_1,\Omega_1),
(f_2,\Omega_2) \in \mathcal{M}^{G}$, 
\begin{equation*}
G\text{{\rm-deg}}(f_1\times f_2,\Omega_1\times \Omega_2) =G\text{{\rm
-deg}}(f_1,\Omega_1)\cdot G\text{{\rm -deg}}(f_2,\Omega_2),
\end{equation*}
where the multiplication `$\cdot$' is taken in the Burnside ring $A(G)$.

\item[($G$6)] \textbf{(Suspension)} \textit{If $\,W$ is an orthogonal $G$%
-representation and $\mathcal{B}$ is an open bounded invariant neighborhood
of $0\in W$, then} 
\begin{equation*}
G\text{{\rm -deg}}(f\times \text{{\rm Id}}_W,\Omega\times \mathcal{B}%
)=G\text{{\rm -deg}}(f,\Omega).
\end{equation*}

\item[($G$7)] \textbf{(Recurrence Formula)} For an admissible $G$-pair $%
(f,\Omega)$, the $G$-degree \eqref{eq:G-deg0} can be computed using the
following recurrence formula 
\begin{equation}  \label{eq:RF-0}
n_H = \frac{\deg(f^H, \Omega^H)-\sum_{(K)>(H)} n_K\, n(H,K)\, |W(K)|} {|W(H)|%
},
\end{equation}
where $|X|$ stands for the number of elements in the set $X$ and $\deg(f^H,\Omega^H)$ is the Brouwer degree of the map $f^H:=f|_{V^H}$ on the
set $\Omega^H\subset V^H$.

\item[($G$8)] \textbf{(Hopf Property)} Assume $B(V)$ is the unit ball of an
orthogonal $G$-representa\-tion $V$ and for $(f_1,B(V)),(f_2,B(V)) \in 
\mathcal{M}^{G}$, one has $\Gamma\text{\textrm{-deg\,}}(f_1,B(V))$ = $G\text{\textrm{-deg\,}}(f_2,B(V))$. Then, $f_1$ and $f_2$ are $B(V)$-admissible $G$%
-homotopic.
\end{itemize}
\end{theorem}

Suppose that $\psi:G^{\prime }\to G$ is a homomorphism between two (for
simplicity) \textit{finite} groups. Then, $\psi$ induces a homomorphism $\Psi: A(G)\to A(G^{\prime })$ of Burnside rings. More precisely, the formula 
\begin{equation}  \label{eq:induced-action}
{g^{\prime }}x := \psi({g^{\prime }})x \quad (g^{\prime }\in G^{\prime })
\end{equation}
defines a ${G^{\prime }}$-action on $G$. In particular, for any subgroup $H
\subset G$, the map $\psi$ induces the $G^{\prime}$-action on $G/H$. Then, 
\begin{equation}  \label{eq:Burnside-homo}
\Psi(H)=\sum_{(K)\in \Phi_0(\Gamma^{\prime })} n_K (K),
\end{equation}
where 
\begin{equation}  \label{eq:Burnside-homoCoef}
n_K:=\left| (\Gamma/H)_{(K)}/ \Gamma^{\prime }\right|= \left|
(\Gamma/H)_{K}/ N(K)\right|.
\end{equation}

\begin{itemize}
\item[($G$9)] \textbf{(Functoriality Property)} Let $(f,\Omega) \in 
\mathcal{M}^{G}$ and $\psi :G^{\prime }\to G$ a homomorphism of finite
groups. Then, $(f,\Omega) \ \in \mathcal{M}^{G^{\prime }}$ and 
\begin{equation}  \label{eq:deg-hom-gr}
G^{\prime }\text{{\rm -deg\,}}(f,\Omega)=\Psi\left[ G\text{{\rm -deg}}(f,\Omega) \right].
\end{equation}
\end{itemize}

\vskip.3cm


\subsection{$G$-Equivariant Gradient Degree}

Let $V$ be an orthogonal $G$-representation. Denote by $C^2_G(V,\mathbb{R})$
the space of $G$-invariant real $C^2$-functions on $V$. Let $\varphi \in
C^2_G(V,\mathbb{R})$ and $\Omega\subset V$ be an open bounded invariant set
such that $\nabla \varphi(x)\not=0$ for $x\in \partial \Omega$. Then the
pair $(f,\Omega)$ is called \textit{$G$-gradient admissible}. Denote by $\mathcal{M}^G_{\nabla}(V,V)$ the set of all $G$-gradient $\Omega$-admissible
pairs in $\mathcal{M}^G(V,V)$ and put $\mathcal{M}^G_{\nabla}:= \bigcup_V 
\mathcal{M}^G_{\nabla}(V,V)$. In an obvious way, one can define a $G$-gradient $\Omega$-admissible homotopy between two $G$-gradient $\Omega$%
-admissible maps. Finally, given a $G$-gradient $\Omega$-admissible homotopy
class, one should have a concept of its ``nice representatives''. The
corresponding concept of \textit{special $\Omega$-Morse functions} was
elaborated by K.H. Mayer in \cite{MA}. 

\vs
\begin{definition}\rm\label{def:gene}
A $G$-gradient $\Omega$-admissible map $f:=\nabla \vp$ is called a {\it special } 
 $\Omega$-Morse function if 
 \begin{itemize}
 \item[(i)] $f|_{\Omega}$ is of class $C^1$;
\item[(ii)] $f^{-1}(0)\cap \Omega$ is composed of regular zero orbits;
\item[(iii)] for each $(H)$ with $f^{-1}(0)\cap\Omega_{(H)}\ne\emptyset$, 
there exists a tubular neighborhood $\mathcal N(U,\ve)$ such that $f$ is 
$(H)$-normal on $\mathcal N(U,\ve)$.
\end{itemize}
\end{definition}
\vs

\begin{theorem}
\label{thm:Ggrad-properties} There exists a unique map $\nabla_G\text{\textrm{-deg\,}}:\mathcal{M}_\nabla^G\to U(G)$, which assigns to every $(\nabla \varphi,\Omega) \in \mathcal{M}^G_{\nabla}$ an element $\nabla_G\text{\textrm{-deg\,}}(\nabla \varphi,\Omega)\in U(G)$, called the \textit{$G $-gradient degree} of $\nabla \varphi$ on $\Omega$, 
\begin{equation}  \label{eq:grad-deg}
\nabla_G\text{{\rm -deg\,}}(\nabla
\varphi,\Omega)=\sum_{(H_i)\in\Phi(\Gamma)} n_{H_i}(H_i)=n_{H_1}(H_1)+\dots
+ n_{H_m}(H_m),
\end{equation}
satisfying the following properties:

\begin{itemize}
\item[($\protect\nabla$1)\hspace{-.095cm}] \textbf{(Existence)} If $\nabla_G\text{{\rm -deg\,}}(\nabla \varphi,\Omega)\not=0$, i.e. there is in \eqref{eq:grad-deg} a non-zero coefficient $n_{H_i}$, then $\exists_{x\in\Omega }$ such that $\nabla \varphi (x)=0$ and $(G_x)\geq (H_i)$.

\item[($\protect\nabla$2)\hspace{-.095cm}] \textbf{(Additivity)} Let $\Omega_1$ and $\Omega _2$ be two disjoint open $G$-invariant subsets of $\Omega $ such that $(\nabla\varphi)^{-1}(0)\cap \Omega \subset \Omega _1\cup\Omega_2.$ Then, 
\begin{equation*}
\nabla_G\text{{\rm -deg\,}}(\nabla \varphi,\Omega)=\nabla_G\text{{\rm -deg\,}}(\nabla \varphi,\Omega_1) + \nabla_G\text{{\rm -deg\,}}(\nabla\varphi,\Omega_2).
\end{equation*}

\item[($\protect\nabla$3)\hspace{-.095cm}] \textbf{(Homotopy)} If $\nabla_v\Psi:[0,1]\times V\to V$ is a $G$-gradient $\Omega$-admissible
homotopy, then 
\begin{equation*}
\nabla_G\text{{\rm -deg\,}}(\nabla_v\Psi(t,\cdot),\Omega)=\text{ constant}.
\end{equation*}

\item[($\protect\nabla$4)\hspace{-.095cm}] \textbf{(Normalization)} Let $\varphi\in C^2_G(V,\mathbb{R})$ be a special $\Omega$-Morse function such
that $(\nabla\varphi)^{-1}(0)\cap \Omega=G(v_0)$ and $G_{v_0}=H$. Then, 
\begin{equation*}
\nabla_G\text{{\rm -deg\,}}(\nabla \varphi,\Omega)= (-1)^{{m}^-(\nabla^2\varphi(v_0))}\cdot (H),
\end{equation*}
where ``${m}^-(\cdot)$'' stands for the total dimension of
eigenspaces for negative eigenvalues of a (symmetric) matrix.

\item[($\protect\nabla$5)\hspace{-.095cm}] \textbf{(Multiplicativity)} For
all $(\nabla\varphi_1,\Omega_1)$, $(\nabla\varphi_2,\Omega_2) \in \mathcal{M}^G_\nabla$, 
\begin{equation*}
\nabla_G\text{{\rm -deg\,}}(\nabla\varphi_1\times\nabla\varphi_2,\Omega_1\times\Omega_2)=\nabla_G\text{{\rm-deg\,}}(\nabla
\varphi_1,\Omega_1)\ast \nabla_G\text{{\rm-deg\,}}(\nabla
\varphi_2,\Omega_2)
\end{equation*}
where the multiplication `$\ast$' is taken in the Euler ring $U(G)$.

\item[($\protect\nabla$6)\hspace{-.095cm}] \textbf{(Suspension)} If $W$ is
an orthogonal $G$-representation and $\mathcal{B}$ an open bounded invariant
neighborhood of $0\in W$, then 
\begin{equation*}
\nabla_G\text{{\rm -deg\,}}(\nabla \varphi\times \mbox{\rm Id}_W,\Omega\times \mathcal{B}) = \nabla_G\text{{\rm-deg\,}}(\nabla
\varphi,\Omega).
\end{equation*}

\item[($\protect\nabla$7)\hspace{-.095cm}] \textbf{(Hopf Property)} Assume $B(V)$ is the unit ball of an orthogonal $\Gamma$-representa\-tion $V$ and
for $(\nabla\varphi_1,B(V)),(\nabla\varphi_2,B(V)) \in \mathcal{M}^{G}_{\nabla}$, one has 
\begin{equation*}
\nabla_G\text{{\rm-deg\,}}(\nabla\varphi_1,B(V)) = \nabla_G\text{{\rm-deg\,}}(\nabla\varphi_2,B(V)).
\end{equation*}
Then, $\nabla\varphi_1$ and $\nabla\varphi_2$ are $G$-gradient $B(V)$-admissible homotopic.

\item[($\protect\nabla$8)\hspace{-.095cm}] \textbf{(Functoriality Property)}
Let $V$ be an orthogonal $G$-representation, $f : V \to V$ a $G$-gradient $\Omega$-admissible map, and $\psi:G_o\hookrightarrow G$ an embedding of Lie groups (here we assume that $G$ and $G_o$ have the same dimensions).
Then, $\psi$ induces a $G_o$-action on $V$ such that $f$ is an $\Omega$-admissible $G_o$-gradient map, and the following equality holds 
\begin{equation}  \label{eq:funct-G}
\Psi[ \nabla_G\text{{\rm-deg\,}}(f,\Omega)]=\nabla_{G_o}\text{{\rm-deg}}(f,\Omega),
\end{equation}
where $\Psi:U(G)\to U(G_o)$ is the homomorphism of Euler rings induced by $\psi$.

\item[($\protect\nabla$9)\hspace{-.095cm}] \textbf{(Reduction Property)} Let 
$V$ be an orthogonal $G$-representation, $f : V \to V$ a $G$-gradient $\Omega $-admissible map, then 
\begin{equation}  \label{eq:red-G}
\pi_0\left[ \nabla_G\text{{\rm-deg\,}}(f,\Omega) \right] = G\text{{\rm-deg\,}}(f,\Omega).
\end{equation}
where the ring homomorphism $\pi_0:U(G)\to A(G)$ is given by \eqref{eq:pi_0-homomorphism}.
\end{itemize}
\end{theorem}
\begin{proof} {\bf (Functoriality Property)} ($\protect\nabla$8): It is sufficient to assume that $f$ is special $\Omega$-Morse function such
that $(\nabla\varphi)^{-1}(0)\cap \Omega=G(v_0)$ and $G_{v_0}=H$. Then clearly the orbit $G_o(v_0)$ has the same dimension as $G(v_0)$, which implies that they share the same slice $S$ at the point $v_0$. Since $(G_o)_{v_0}=G_o\cap H=:H_o$, it follows that $S^{H_o}\supset S^H$, which implies that $f$ is also special $\Omega$-Morse function with respect to the group action $G_o$, and therefore 
\[
\nabla_{G_o}\text{{\rm -deg\,}}(\nabla \varphi,\Omega)= (-1)^{{m}^-(\nabla^2\varphi(v_0))}\cdot \Psi(H).
\]
\end{proof}
\begin{remark}\rm
Let us point out that the {\bf Functoriality Property} for the gradient degree was {\it overstated} in \cite{survey} (see property ($\protect\nabla$8), page 30). This statement is not true for general Lie groups $G$ and $G_o$.  Indeed, consider for example $G=S^1$ and $G_o=\bz_n$, $n\ge 3$, and let $V:=\bc$ with the action of $S^1$ and $\bz_3$ by complex multiplication. Then  for $f:=-\id: V\to V$  we have 
\[
\nabla_{G_o}\text{{\rm -deg\,}}(\nabla \varphi,\Omega)=(\bz_n)-(\bz_1),\quad \nabla_{G}\text{{\rm -deg\,}}(\nabla \varphi,\Omega)=(S^1)-(\bz_1).
\]
On the other hand, $\Psi(S^1)=(\bz_n)$, $\Psi(\bz_1)=0$, so 
\[
(\bz_n)=\Psi\left[  \nabla_{G}\text{{\rm -deg\,}}(\nabla \varphi,\Omega)  \right]\not=\nabla_{G_o}\text{{\rm -deg\,}}(\nabla \varphi,\Omega).
\]
\end{remark}
\vs
\begin{remark}\label{rem:Psi}
\rm
We would like to point out that the Euler homomorphism described in  Functoriality Property ($\protect\nabla$8) may take a generator $(H)$ into a linear combination (with positive coefficients) of more then one generator from $\Phi(G_o\times \Gamma)$. Indeed, consider $G_o:=A_4\times S^1 \le S_4\times S^1=:G$, then $\Psi(\bz_3^t)=(\bz_3^{t_1})+(\bz_3^{t_2})$ (see \cite{AED} for the explanation of notation). 
\end{remark}
\vs
\subsection{$G$-Equivariant Gradient Degree in Hilbert
Spaces} Let $\mathscr H$ be a Hilbert space and
consider a $C^1$-differentiable $G$-invariant functional and $f:\mathscr{H}%
\to \mathbb{R}$ given by $f(x)=\frac 12 \|x\|^2-\varphi(x)$, $x\in \mathscr{H}
$. Then 
\begin{equation*}
\nabla f(x)=x-\nabla \varphi(x), \quad x\in \mathscr{H}.
\end{equation*}
We will say that the functional $\varphi$ is admissible if $\nabla \varphi:\mathscr{H}\to \mathscr{H}$ is a completely continuous map. Suppose $\Omega\subset \mathscr{H}$ is a $G$-invariant bounded open set such that the
map $\nabla f:\mathscr{H}\to \mathscr{H}$ is $\Omega$-admissible, i.e. 
\begin{equation*}
\forall_{x\in \Omega}\;\;\;\nabla f(x)= x-\nabla \varphi(x)\not=0.
\end{equation*}
By a \textit{$G$-equivariant approximation scheme} $\{P_n\}_{n=1}^\infty$ in 
$\mathscr{H}$, we mean a sequence of $G$-equivariant orthogonal projections $P_n:\mathscr{H}\to\mathscr{H}$, $n=1$, $2$, \dots, such that:

\begin{itemize}
\item[(a)] the subspaces $\mathscr{H}^n:=P_n(\mathscr{H})$, $n=1,2,$ \dots,
are finite-dimensional;

\item[(b)] $\mathscr{H}^n\subset \mathscr{H}^{n+1}$, $n=0,1,2,$ \dots;

\item[(c)] $\displaystyle \lim_{n\to \infty} P_nx=x$ for all $x\in \mathscr{H}
$.
\end{itemize}

\vskip.3cm

\begin{lemma}
\label{lem:lem1} Suppose that $\varphi :\mathscr{H}\to \mathbb{R}$ is a $G$-equivariant $C^1$-differentiable admissible functional and let $\Omega\subset \mathscr{H}$ be a $G$-invariant bounded open set such that $\nabla f$, where $f(x):=\frac 12\|x\|^2-\varphi(x)$, $x\in \mathscr{H}$, is $\Omega$-admissible. Let $\{P_n\}_{n=1}^\infty$ be a $G$-equivariant
approximation scheme in $\mathscr{H}$. Define $f_n:\mathscr{H}_n\to \mathbb{R}$
by $f_n(x):=\frac 12\|x\|^2-\varphi(x)$, $x\in \mathscr{H}_n$. Then for
sufficiently large $n\in \mathbb{N}$, the maps $\nabla
f_n(x):=x-P_n\nabla\varphi(x)$, $x\in \mathscr{H}$, are $\Omega_n$-admissible, where $\Omega_n:=\Omega\cap \mathscr{H}_n$.
\end{lemma}

\begin{proof} Suppose for contradiction that there exists a sequence $\{x_{n_k}\}\subset \partial \Omega$, such that $x_{n_k}=P_{n_k}\nabla\vp(x_{n_k})$. Since the map $\nabla \vp:\mathscr H\to \mathscr H$ is completely continuous, we can assume without loss of generality that the sequence $\{\nabla \vp(x_{n_k})\}$ converges to $x_*$ as $k\to \infty$. Then
\begin{align*}
\|x_*-P_{n_k}\nabla \vp(x_{n_k})\|&\le \|x_*-P_{n_k}x_*\|+\|P_{n_k}(x_*-\nabla\vp(x_{n_k}))\|\\
&\le  \|x_*-P_{n_k}x_*\|+\|P_{n_k}\|\cdot \|x_*-\nabla\vp(x_{n_k})\|\\
&= \|x_*-P_{n_k}x_*\|+\|x_*-\nabla\vp(x_{n_k})\|\to 0 \;\; \text{ as }\; k\to \infty.
\end{align*}
which implies that $x_{n_k}=P_{n_k}\nabla\vp(x_{n_k})\to x_*$ as $k\to\infty$, and therefore, by continuity of $\nabla \vp$, $x_*-\nabla \vp(x_*)=0$, Since $\partial \Omega$ is closed, $x_*\in \partial \Omega$, which is a contradiction with the assumption that $\nabla f$ is $\Omega$-admissible.
\end{proof}

\begin{lemma}
\label{lem:lem2} Suppose that $\varphi :\mathscr{H}\to \mathbb{R}$ is a $G$-equivariant $C^1$-differentiable admissible functional and let $\Omega\subset \mathscr{H}$ be a $G$-invariant bounded open set such that $\nabla f$, where $f(x):=\frac 12\|x\|^2-\varphi(x)$, $x\in \mathscr{H}$, is $\Omega$-admissible. Let $\{P_n\}_{n=1}^\infty$ be a $G$-equivariant
approximation scheme in $\mathscr{H}$. Define $f_n:\mathscr{H}_n\to \mathbb{R}$
by $f_n(x):=\frac 12\|x\|^2-\varphi(x)$, $x\in \mathscr{H}_n$. Then for
sufficiently large $n\in\mathbb{N}$, the maps $\mathfrak{h}_n:[0,1]\times 
\mathscr{H}_{n+1}\to \mathscr{H}_{n+1}$, defined by 
\begin{equation*}
\mathfrak{h}_n(x):=x-tP_n\nabla\varphi(x)-(1-t)P_{n+1}\nabla\varphi(x),\quad
x\in \mathscr{H}_{n+1},
\end{equation*}
are are $\Omega_{n+1}$-admissible homotopies, where $\Omega_{n+1}:=\Omega\cap \mathscr{H}_{n+1}$.
\end{lemma}

\begin{proof}
Suppose for contradiction that there exists sequence $\{x_{n_k}\}\subset \partial \Omega$ and $\{t_{n_k}\}\subset [0,1]$, such that $x_{n_k}=t_{n_k} P_{n_k}\nabla\vp(x_{n_k})+(1-t_{n_k} )P_{n_k+1}\nabla\vp(x_{n_k})$. Since the map $\nabla \vp:\mathscr H\to \mathscr H$ is completely continuous, we can assume without loss of generality that the sequence $\{\nabla \vp(x_{n_k})\}$ converges to $x_*$ and $t_{n_k}$ converges to $t_*$  as $k\to \infty$. Then
$t_{n_k} P_{n_k}\nabla\vp(x_{n_k})+(1-t_{n_k} )P_{n_k+1}\nabla\vp(x_{n_k})\to x_*$ as $k\to\infty$, and therefore $x_{n_k}\to x_*$. Therefore, by continuity of $\nabla\vp$,
\[
x_*=t_*\nabla\vp(x_*)+(1-t_*)\nabla\vp(x_*)=\nabla \vp(x_*).
\]
 Since $\partial \Omega$ is closed, $x_*\in \partial \Omega$, which is a contradiction with the assumption that $\nabla f$ is $\Omega$-admissible.
\end{proof}
\vskip.3cm

Consider therefore a $G$-equivariant $C^1$-differentiable admissible
functional $\varphi :\mathscr {H}\to \mathbb{R}$ and let $\Omega\subset 
\mathscr{H}$ be a $G$-invariant bounded open set such that $\nabla f$, where $f(x):=\frac 12\|x\|^2-\varphi(x)$, $x\in \mathscr{H}$, is $\Omega$-admissible. In order to define the $G$-equivariant gradient degree $\nabla_G\text{\textrm{-deg\,}}(\nabla f,\Omega)$, we consider a $G$-equivariant
approximation scheme $\{P_n\}_{n=1}^\infty$ in $\mathscr{H}$, define $f_n:\mathscr{H}_n\to \mathbb{R}$ by $f_n(x):=\frac 12\|x\|^2-\varphi(x)$, $x\in 
\mathscr{H}_n$. Then by Lemma \ref{lem:lem1}, $\nabla f_n:\mathscr{H}_n\to 
\mathscr{H}_n$ is $\Omega_n$-admissible for $n$ sufficiently large, where $\Omega_n:=\Omega\cap \mathscr{H}_n$, and therefore the $G$-equivariant
degrees $\nabla_G\text{\textrm{-deg\,}}(\nabla f_n,\Omega_n)$ are
well-defined for $n$ sufficiently large. On the other hand, by Lemma \ref{lem:lem2} and the Suspension Property of the $G$-equivariant gradient
degree, for sufficiently large $n$ we have 
\begin{equation*}
\nabla_G\text{\textrm{-deg\,}}(\nabla f_n,\Omega_n)=\nabla_G\text{\textrm{-deg\,}}(\nabla f_{n+1},\Omega_{n+1})
\end{equation*}
which implies that we can put 
\begin{equation}  \label{eq:grad-H}
\nabla_G\text{\textrm{-deg\,}}(\nabla f,\Omega):= \nabla_G\text{\textrm{-deg\,}}(\nabla f_n,\Omega_n),\;\; \text{ where } \; n\; \text{ is
sufficiently large}.
\end{equation}
One can easily verify that this construction doesn't depend on the choice of
a $G$-approximation scheme in the space $\mathscr{H}$. We should mention that
the ideas behind the usage of the approximation methods to define
topological degree can be rooted to \cite{BP}.

\vskip.3cm The $G$-equivariant gradient degree defined by \eqref{eq:grad-H}
has all the standard properties of a topological degree, i.e. existence,
additivity, homotopy and multiplicativity. 

\vs
\subsection{Computations of the Gradient $G$-Equivariant Degree}

 Similarly to the case of the Brouwer degree, the gradient equivariant degree
can be computed using standard linearization techniques. Therefore, it is
important to establish computational formulae for linear gradient operators.

Let $V$ be an orthogonal (finite-dimensional) $G$-representation and suppose
that $A:V\to V$ is a $G$-equivariant symmetric isomorphism of $V$, i.e. $A:=\nabla\varphi$, where $\varphi(x)=\frac 12 Ax\bullet x$. Consider the $G$%
-isotypical decomposition of $V$ 
\begin{equation*}
V=\bigoplus_{i} V_i, \quad V_i \;\; \text{modeled on $\mathcal{V} _i$}.
\end{equation*}
We assume here that $\{ \mathcal{V} _i\}_i$ is the complete list of
irreducible $G$-representations.

Let $\sigma(A)$ denote the spectrum of $A$ and $\sigma_-(A)$ the negative
spectrum of $A$, i.e. 
\begin{equation*}
\sigma_-(A):=\{\lambda\in \sigma(A): \lambda<1\},
\end{equation*}
and let $E_\mu(A)$ stands for the eigenspace of $A$ corresponding to $\mu\in
\sigma(A)$. Put $\Omega:=\{x\in V: \|x\|<1\}$. Then $A$ is $\Omega$-admissibly homotopic (in the class of gradient maps) to a linear operator $A_o:V\to V$ such that 
\begin{equation*}
A_o(v):= 
\begin{cases}
-v & \;\text{ if } \; v\in E_\mu(A),\; \mu\in \sigma_-(A), \\ 
v & \; \text{ if } \; v\in E_\mu(A),\; \mu\in \sigma(A)\setminus \sigma_-(A).
\end{cases}
\end{equation*}
In other words, $A_o|_{E_\mu(A)}=-\text{\textrm{Id\,}}$ for $\mu
\in\sigma_-(A)$ and $A_o|_{E_\mu(A)}=\text{\textrm{Id\,}}$ for $\mu \in
\sigma(A)\setminus \sigma_-(A)$. Suppose that $\mu\in \sigma_-(A)$, then
denote by $m_i(\mu)$ the integer 
\begin{equation*}
m_i(\mu):= \text{\textrm{dim\,}} (E_\mu(A)\cap V_i)/\text{\textrm{dim\,}} 
\mathcal{V} _i,
\end{equation*}
which is called the \textit{$\mathcal{V} _i$-multiplicity} of $\mu$. Since $\nabla_G\text{\textrm{-deg\,}}(\text{\textrm{Id\,}},\mathcal{V} _i)=(G)$ is
the unit element in $U(G)$, we immediately obtain, by Multiplicativity
property ($\nabla$5), the following formula 
\begin{equation}\label{eq:grad-lin}
\nabla_G\text{\textrm{-deg\,}}(A,\Omega)=\prod_{\mu\in \sigma_-(A)}\prod_{i} 
\left[\nabla_G\text{\textrm{-deg\,}}(-\text{\textrm{Id\,}},B(\mathcal{V} _i))
\right]^{m_i(\mu)},
\end{equation}
where $B(W)$ is the unit ball in $W$.

\vskip.3cm

\begin{definition}\rm 
Assume that $\mathcal{V}_i$ is an irreducible $G$-representation. Then the $G $-equivariant gradient degree: 
\begin{equation*}
\text{\textrm{Deg\,}}_{\mathcal{V}_i}:=\nabla_G\text{\textrm{-deg\,}}(-\text{\textrm{Id\,}},B(\mathcal{V}_i))\in U(G)
\end{equation*}
is called the \textit{gradient $G$-equivariant basic degree} for $\mathcal{V}_i$.
\end{definition}

Clearly, the basic degrees $\text{\textrm{Deg}}_{\mathcal{V}_i}$ are
invertible elements in $U(G)$. \vskip.3cm

\begin{remark}\rm
For $\widetilde G=\Gamma\times S^1$, where $\Gamma$ is a finite
group, the computation of basic gradient degrees $\wt{\text{\rm Deg}}_{\mathcal{V}_o}$ (here $\mathcal V_o$ stands for an irreducible $\Gamma\times S^1$-representation) can be completely reduced
to the computation of basic degrees $\deg_{\mathcal{V}_i}$ without parameter 
and twisted basic degrees $\deg_{\mathcal{V}_{j,l}} $ with one free parameter. Namely, we have the following types of irreducible $\widetilde G$-representations: (i) irreducible $\Gamma$-representation $\mathcal{V}_i$,
where $S^1$ acts trivially; (ii) the representations $\mathcal{V}_{j,l}$, $l\in \mathbb{N}$, where $\mathcal{V}_{j,l}$ as a $\Gamma$-representation is
equivalent to a complex irreducible $\Gamma$-representation $\mathcal{U}_j$,
and the $S^1$ action on $\mathcal{V}_{j,l}$ is given by \textit{$l$-folding}, i.e. $e^{i\alpha} v:=e^{il\alpha}\cdot v$, $v\in \mathcal{V}_{j,l}$, $e^{i\alpha}\in S^1$ and `$\cdot$' stands for complex multiplication. Then: 

\begin{itemize}
\item[(i)] $\wt{\text{\rm Deg}}_{\mathcal{V}_i} = \wt{\deg}_{\mathcal{V}_i}$; 

\item[(ii)] $\wt{\text{\rm Deg}}_{\mathcal{V}_{j,l}}=(\widetilde G)-\wt{\text{\rm deg}}_{\mathcal{V}_{j,l}}$, 
\end{itemize}
where $\wt{\text{\rm deg}}_{\mathcal{V}_i}=\widetilde G\text{\rm -deg\,}(-\text{\rm Id\,},B(\mathcal{V}_i))\in A(\widetilde G)$ and $\deg_{\mathcal{V}_{j,l}}$ is the so-called \textit{twisted basic degree} (see \cite{AED} for
more details and definitions). More precisely, the basic degree 
\begin{equation*}
\wt{\text{\rm deg}}_{\mathcal{V}_i} =(\widetilde G)+n_{L_1}(L_1)+\dots+n_{L_n}(L_n),
\end{equation*}
can be computed from the recurrence formula 
\begin{equation}  \label{eq:coeff-jo}
n_{L_{k}}=\frac{(-1)^{n_{k}}-\sum_{L_{k}<L_l} n(L_{k},L_k)\cdot n_{L_l}\cdot
|W(L_l)|}{|W(L_{k})|},\quad n_k=\text{dim\,}\mathcal V_i^{L_k}
\end{equation}
and the twisted twisted degree 
\begin{equation*}
\wt{\deg}_{\mathcal{V}_{j,l}}=n_{H_1}(H_1)+n_{H_2}(H_2)+\dots + n_{H_m}(H_m),
\end{equation*}
can be computed from the recurrence formula 
\begin{equation}  \label{eq:bdeg-nL}
n_{H_k}=\frac{\frac 12 \text{dim}\,\, \mathcal{V}_{j,l}^{H_k}-
\sum_{H_k<H_s}n_{H_s}\, n(H_k,H_s)\, |W(H_s)/S^1|}{\left| \frac{W(H_k)}{S^1}\right|}.
\end{equation}
One can also find in \cite{AED} complete lists of these basic degrees for
several groups $\wt G=\Gamma\times S^1$. 
\end{remark}

\vskip.3cm

\paragraph{The Case $G:=\Gamma\times O(2)$ with $\Gamma$ being a finite group:}

Consider a symmetric $G$-equivariant linear isomorphism $T:V\to V$, where $V$
is an orthogonal $G$-representation, i.e. $T=\nabla \varphi$ for $\varphi(v)=\frac 12 (Tv\bullet T)$, $v\in V$, where ``$\bullet$'' stands for
the inner product. We will show how to compute $\nabla_G\text{-deg\,}(T,B(V)) $.

\vskip.3cm Assume that $\{\mathcal{V}_i\}_{i=0}^r$ denotes the collection of
all real $\Gamma$-irreducible representations and 
\begin{equation*}
\{ \mathcal{U}_0,\mathcal{U}_{\frac 12}, \mathcal{U}_1,\mathcal{U}_l,\mathcal{U}_{l+1},\dots\}
\end{equation*}
denotes the collection of all real irreducible $O(2)$-representations, where
(following \cite{AED}) $\mathcal{U}_0\simeq \mathbb{R}$ is the trivial
representation of $O(2)$, $\mathcal{U}_{\frac 12}\simeq \mathbb{R}$ is the
one-dimensional irreducible real representation, on which $O(2)$ acts
through the homomorphism $O(2)\to O(2)/SO(2)\simeq {\mathbb{Z}}_2$, and $\mathcal{U}_l\simeq {\mathbb{C}}$, $l\in \mathbb{N}$, is the two-dimensional
irreducible real representation of $O(2)$ with the action of $O(2)$ given by 
\begin{gather*}
e^{i\theta} z:= e^{il\theta}\cdot z, \quad e^{i\theta}=\left[ 
\begin{array}{cc}
\cos \theta & -\sin\theta \\ 
\sin \theta & \cos \theta
\end{array}
\right]\in SO(2), \\
\kappa z:= \overline z,\quad \kappa ==\left[ 
\begin{array}{cc}
1 & 0 \\ 
0 & -1
\end{array}
\right],
\end{gather*}
where `$\cdot$' denotes complex multiplication.

\vskip.3cm There are three types of irreducible $G$-representations $\mathcal{V}_o$: (i) those irreducible $G$-representations on which $O(2)$
acts trivially, i.e. $\mathcal{V}_o$ is in fact an irreducible $\Gamma$-representation (there is no $O(2)$-action) which we will denote by $%
\mathcal{V}_i$, $i=0,1,2,\dots,r$, (ii) those irreducible $G$-representations on which $SO(2)$ acts trivially and $SO(2)\kappa$ acts by
multiplication by $-1$, which we will denote by $\mathcal{V}^-_i$, $i=0,1,2,\dots,r$, and finally we have (iii) those irreducible $G$%
-representations on which $S^1=SO(2)$ acts non-trivially. Therefore, each of
such irreducible $G$-representations $\mathcal{V}_o$ admits a complex
structure induced by the action of $S^1$ and there exists $l\in \mathbb{N}$
such that for all $z\in S^1$, $v\in \mathcal{V}_o$, $zv= z^l\cdot v$ (here `$\cdot $' denotes the complex multiplication). One can show that there exists
an irreducible $\Gamma$-representation $\mathcal{V}_i$ such that $\mathcal{V}_o=\mathcal{V}_i\otimes_{\mathbb{R}} U_l$. In such a case, we will denote
this irreducible $G$-representation by $\mathcal{V}_{i,l}$, $i=0,1,2,\dots,
r $, $l=1,2,\dots$. The irreducible $G$-representation $\mathcal{V}_{i,l}$
can be easily described: as $\Gamma$-representation it is $\mathcal{V}_{j,l}=\mathcal{V}_i\oplus \mathcal{V}_i$ with $z\in \mathcal{V}_{i,l}$ written as $z=(x,y)^T$, $x,y\in \mathcal{V}_i$, where $S^1$-action is given by 
\begin{equation*}
e^{i\alpha} z:=\left[ 
\begin{array}{cc}
\cos (l\alpha)\, \text{\textrm{Id\,}}_{\mathcal{V}_i} & -\sin (l\alpha)\, 
\text{\textrm{Id\,}}_{\mathcal{V}_i} \\ 
\sin (l\alpha)\, \text{\textrm{Id\,}}_{\mathcal{V}_i} & \cos (l\alpha)\, 
\text{\textrm{Id\,}}_{\mathcal{V}_i}
\end{array}
\right] \left[ 
\begin{array}{c}
x \\ 
y
\end{array}
\right]
\end{equation*}
and $\kappa z=\kappa (x,y)^T=(y,x)^T$, for $(x,y)^T\in \mathcal{V}_{i,l}$.

\vskip.3cm Let us describe the algorithm for the computation of the basic
gradient degrees $\text{\textrm{Deg\,}}_{\mathcal{V}_i}$, $\text{\textrm{Deg\,}}_{\mathcal{V}^-_i}$ and $\text{\textrm{Deg\,}}_{\mathcal{V}_{i,l}}$.
Notice that the degrees $\text{\textrm{Deg\,}}_{\mathcal{V}_i}$ and $\text{\textrm{Deg\,}}_{\mathcal{V}^-_i}$ belong to $A(G)$, therefore they can be
computed using the recurrence formula \eqref{eq:RF-0}. In the case of the
gradient basic degree $\text{\textrm{Deg\,}}_{\mathcal{V}_{i,l}}$, consider
two elements $\text{\textrm{Deg\,}}^0_{\mathcal{V}_{i,l}}\in A(G)$ and $%
\text{\textrm{Deg\,}}^1_{\mathcal{V}_{i,l}}\in \mathbb{Z}[\Phi_1(G)]$ such
that 
\begin{equation}\label{eq:basic-1}
\text{\textrm{Deg\,}}_{\mathcal{V}_{i,l}}=\text{\textrm{Deg\,}}^0_{\mathcal{V}_{i,l}}+\text{\textrm{Deg\,}}^1_{\mathcal{V}_{i,l}}.
\end{equation}
Then, 
\begin{equation}\label{eq:basic-2}
\text{\textrm{Deg\,}}^0_{\mathcal{V}_{i,l}}=\pi_0[\text{\textrm{Deg\,}}_{\mathcal{V}_{i,l}}]= n_1(L_1)+\dots +n_m(L_m),
\end{equation}
and the coefficients $n_j$, $j=1,2,\dots,m$ can be effectively computed from
the recurrence formula \eqref{eq:RF-0}. Since the coefficients $x_j$ in 
\begin{equation}\label{eq:basic-3}
\text{\textrm{Deg\,}}^1_{\mathcal{V}_{i,l}}= x_1(H_1)+\dots +x_s(H_s),
\end{equation}
are unknown, by Functoriality Property we can apply the Euler ring homomorphisms $\Psi: U(\Gamma\times
O(2))\to U(\Gamma\times S^1)$. More precisely, since $\mathcal V_{i,l}$ is also an irreducible $\Gamma\times S^1$-representation, we have 
\[
\Psi[\text{\textrm{Deg\,}}_{\mathcal{V}_{i,l}}]=\wt{\text{\textrm{Deg\,}}}_{\mathcal{V}_{i,l}},
\]
where $\wt{\text{\textrm{Deg\,}}}_{\mathcal{V}_{i,l}}$ denotes the corresponding gradient basic degree for $\wt G$. Then, by applying $\Psi$ to \eqref{eq:basic-1}, \eqref{eq:basic-1}
and \eqref{eq:basic-3}, we obtain
\begin{equation}\label{eq:basic-4}
2x_1(H_1)+\dots +2x_s(H_s)=\wt{\text{\textrm{Deg\,}}}_{\mathcal{V}_{i,l}}-n_1\Psi(L_1)+\dots +n_m\Psi(L_m),
\end{equation}
which is a system of simple linear equations, which can be easily solved (see Example \ref{ex:s4o2} below).

\vs 
Consequently, we obtain the following result:
\vs
\begin{theorem} For a given group $G:=\Gamma\times O(2)$ all the basic degrees can be effectively computed by applying Euler ring homomorphism $\Psi:U(\Gamma\times O(2))\to U(\Gamma\times S^1)$ and the gradient  basic degrees $\wt{\text{Deg}}_{\mathcal V_o}$ for the group $\wt G:=\Gamma\times S^1$.
\end{theorem}\vskip.3cm

\begin{example}\rm\label{ex:s4o2}
{\bf  Gradient Basic Degrees for Irreducible $S_4\times O(2)$-Representations:}
Let us  illustrate the usage of the Euler homomorphism $\Psi : U(\Gamma\times O(2))\to U(\Gamma\times S^1)$ to compute partially the
gradient basic degrees for $S_4\times O(2)$-actions. The full computations of these degrees, although elementary, are beyond the format of this paper. 
\vs
There are exactly five real irreducible representations of $S_4$:
the one-dimensional trivial representation $\mathcal{W}_0$, the
one-dimensional representation $\mathcal{W} _1$ induced by the homomorphism $S_4\to S_4/A_4\simeq {\mathbb{Z}}_2$, the two-dimensional representation $\mathcal{W} _2$ corresponding to the homomorphism $S_4\to S_4/V_4=D_3\subset
O(2)$, and two three-dimensional representations, one $\mathcal{W} _3$ --
the natural representation of $S_4$ and $\mathcal{W} _4$ -- the tensor
product $\mathcal{W} _1\otimes \mathcal{W}_3$.

\vskip.3cm Then we have the following irreducible orthogonal $S_4\times O(2)$-representations:

\begin{itemize}
\item[(i)] The representations $\mathcal{V} _i\simeq \mathcal{W} _i$, $i=0,1,2,3,4.$

\item[(ii)] The representations $\mathcal{V} _i^-\simeq \mathcal{W}
_i\otimes \mathcal{U} _{\frac 12}$, $i=0,1,2,3,4$,

\item[(iii)] The representations $\mathcal{V} _{l,i}\simeq \mathcal{W}
_i\otimes \mathcal{U} _l$, $i=0,1,2,3,4$, $l\in {\mathbb{N}}$.
\end{itemize}

For the representations of the type (i), we have the following gradient  basic
degrees: 
\begin{align*}
\text{Deg}_{\mathcal{V} _0}&= -(S_4\times O(2)), \\
\text{Deg}_{\mathcal{V} _1}&= (S_4\times O(2)) -(A_4\times O(2)), \\
\text{Deg}_{\mathcal{V} _2}&=(S_4\times O(2))-2(D_4\times O(2)) +(V_4\times O(2)),
\\
\text{Deg}_{\mathcal{V} _3}&=(S_4)-2(D_3\times O(2))-(D_2\times O(2))+3(D_1\times
O(2))-(\mathbb{Z}_1\times O(2)), \\
\deg_{\mathcal{V} _4}&= (S_4\times O(2))-(\mathbb{Z}_4\times
O(2))-(D_1\times O(2))-(\mathbb{Z}_3\times O(2))+(\mathbb{Z}_1\times O(2)).
\end{align*}
For the representations of the type (ii), we have the following gradient  basic
degrees: 
 \begin{align*}
   \text{Deg}_{\mathcal V_0}^-=&\;(\amal{S_4}{O(2)}{}{}{})
        -(\amal{S_4}{SO(2)}{}{}{}),\\
    \text{Deg}_{\mathcal V_1}^-=&\;(\amal{S_4}{O(2)}{}{}{})
        -(\amal{S_4}{O(2)}{D_{1}}{A_4}{}),\\
    \text{Deg}_{\mathcal V_2}^-=&\;(\amal{S_4}{O(2)}{}{}{})
        -(\amal{D_4}{SO(2)}{}{}{})
        -(\amal{D_4}{O(2)}{D_{1}}{V_4}{})
        +(\amal{V_4}{SO(2)}{}{}{}),\\
   \text{Deg}_{\mathcal V_3}^-=&\;(\amal{S_4}{O(2)}{}{}{})
        -(\amal{D_4}{O(2)}{D_{1}}{D_2}{})
        -(\amal{D_3}{SO(2)}{}{}{})
        -(\amal{D_2}{O(2)}{D_{1}}{D_1}{})\\
      & +(\amal{\mathbb{Z}_2}{O(2)}{D_{1}}{\mathbb{Z}_1}{})
        +2(\amal{D_1}{SO(2)}{}{}{})
        -(\amal{\mathbb{Z}_1}{SO(2)}{}{}{}),\\
   \text{Deg}_{\mathcal V_4}^-=&\;(\amal{S_4}{O(2)}{}{}{})
        -(\amal{D_4}{O(2)}{D_{1}}{\mathbb{Z}_4}{})
        -(\amal{D_3}{O(2)}{D_{1}}{\mathbb{Z}_3}{})
        -(\amal{D_2}{O(2)}{D_{1}}{D_1}{})\\
      & +(\amal{\mathbb{Z}_2}{O(2)}{D_{1}}{\mathbb{Z}_1}{})
        +2(\amal{D_1}{O(2)}{D_{1}}{\mathbb{Z}_1}{})
        -(\amal{\mathbb{Z}_1}{SO(2)}{}{}{}).
  \end{align*}
For the representations of the type (iii) ($l\in \mathbb{N}$), the gradient  basic
degrees can be obtained (as an example) using the Euler homomorphism $\Psi$:
 \begin{align*}
    \text{Deg}_{\mathcal V_{0,l}}=&\;(\amal{S_4}{O(2)}{}{}{})
        -(\amal{S_4}{D_{l}}{\mathbb{Z}_{1}}{S_4}{}),\\
    \text{Deg}_{\mathcal V_{1,l}}=&\;(\amal{S_4}{O(2)}{}{}{})
        -(\amal{S_4}{D_{2l}}{\mathbb{Z}_{2}}{A_4}{}),\\
    \text{Deg}_{\mathcal V_{2,l}}=&\;(\amal{S_4}{O(2)}{}{}{})
        -(\amal{S_4}{D_{3l}}{D_{3}}{V_4}{})
        -(\amal{D_4}{D_{l}}{}{}{})
        -(\amal{D_4}{D_{2l}}{\mathbb{Z}_{2}}{V_4}{})\\
      & +(\amal{V_4}{D_{l}}{}{}{})
        +2(\amal{D_4}{D_{l}}{D_{1}}{V_4}{})+x_1(\bz_4\times \bz_l),
        \end{align*}
 where $x_1$ is an unknown. Then we have (see \cite{AED}) 
 \begin{align*}
\Psi(\text{Deg}_{\mathcal V_{2,l}})
 &=(S_4\times S^1)-(A_4^{t,l})-(D_4^{l})-(D_4^{\hat d,l})+(V_4)+(2x_1+2)(\bz_4\times \bz_l)\\
 &=(S_4\times S^1)-(A_4^{t,l})-(D_4^{l})-(D_4^{\hat d,l})+(V_4)\quad \text{(basic degree for $S_4\times S^1$)}.
 \end{align*}       
 Therefore, $x_1=-1$. Similarly, we have 
        \begin{align*}
    \text{Deg}_{\mathcal V_{3,l}}&=\;(\amal{S_4}{O(2)}{}{}{})
        -(\amal{D_4}{D_{2l}}{\mathbb{Z}_{2}}{D_2}{})
        -(\amal{D_3}{D_{l}}{}{}{})
        -(\amal{D_4}{D_{4l}}{D_{4}}{\mathbb{Z}_1}{})\\
      & -(\amal{D_2}{D_{2l}}{\mathbb{Z}_{2}}{D_1}{})
        -(\amal{D_3}{D_{3l}}{D_{3}}{\mathbb{Z}_1}{})
        +(\amal{D_2}{D_{2l}}{D_{2}}{\mathbb{Z}_1}{D_1})
        +(\amal{V_4}{D_{2l}}{D_{2}}{\mathbb{Z}_1}{})\\
      & +(\amal{D_2}{D_{l}}{D_{1}}{D_1}{})
        +(\amal{\mathbb{Z}_2}{D_{2l}}{\mathbb{Z}_{2}}{\mathbb{Z}_1}{})
        +2(\amal{D_1}{D_{l}}{}{}{})
        -(\amal{\mathbb{Z}_2}{D_{l}}{D_{1}}{\mathbb{Z}_1}{})\\
      & -(\amal{\mathbb{Z}_1}{D_{l}}{}{}{})+x_1(\bz_2\times_{\bz_2}\bz_{2l})+x_2(D_1\times \bz_l)+x_3(\bz_1\times \bz_{l}),
      \end{align*}
  where   $x_1$, $x_2$ and $x_3$ are unknown. Similarly, we obtain 
  \begin{align*}
  \Psi( \text{Deg}_{\mathcal V_{3,l}})&=(S_4\times S^1)-(D_4^{d,l})-(D_3^l)-(\bz_4^{c,l})-(D_2^{d,l})-(\bz_3^{t,l}) +(2x_1-1)(\bz_2^{-,l})\\
  &+(2x_2+3)(D_1^{l})+(2x_3-2)(\bz_1^{l})\\
  &=(S_4\times S^1)-(D_4^{d,l})-(D_3^l)-(\bz_4^{c,l})-(D_2^{d,l})-(\bz_3^{t,l}) +(\bz_2^{-,l})\\
  &+(D_1^{l}) \quad \text{(basic degree for $S_4\times S^1$)},
  \end{align*}
 which implies that $x_1=1$, $x_2=-2$ and $x_3=0$. Finally, we have 
      \begin{align*}
    \text{Deg}_{\mathcal V_{4,l}}=&\;(\amal{S_4}{O(2)}{}{}{})
        -(\amal{D_4}{D_{2l}}{\mathbb{Z}_{2}}{\mathbb{Z}_4}{})
        -(\amal{D_3}{D_{2l}}{\mathbb{Z}_{2}}{\mathbb{Z}_3}{})
        -(\amal{D_4}{D_{4l}}{D_{4}}{\mathbb{Z}_1}{})\\
      & -(\amal{D_2}{D_{2l}}{\mathbb{Z}_{2}}{D_1}{})
        -(\amal{D_3}{D_{3l}}{D_{3}}{\mathbb{Z}_1}{})
        +(\amal{D_2}{D_{2l}}{D_{2}}{\mathbb{Z}_1}{D_1})
        +(\amal{D_2}{D_{2l}}{D_{2}}{\mathbb{Z}_1}{\mathbb{Z}_2})\\
      & +(\amal{V_4}{D_{2l}}{D_{2}}{\mathbb{Z}_1}{})
        +2(\amal{D_1}{D_{2l}}{\mathbb{Z}_{2}}{\mathbb{Z}_1}{})
        +(\amal{\mathbb{Z}_2}{D_{2l}}{\mathbb{Z}_{2}}{\mathbb{Z}_1}{})
        -(\amal{\mathbb{Z}_2}{D_{l}}{D_{1}}{\mathbb{Z}_1}{})\\
      & -(\amal{\mathbb{Z}_1}{D_{l}}{}{}{})+x_1(\bz_2\times_{\bz_2}\bz_{2l})+x_2(D_1\times_{\bz_2}\bz_{2l})+x_3(\bz_1\times \bz_l).
  \end{align*}
 Then, we have
 \begin{align*}
\Psi(  \text{Deg}_{\mathcal V_{4,l}})&=(S_4\times S^1)-(D_4^{z,l})-(\bz_4^{c,l})-(D_2^{d,l})-(D_3^{z,l})-(\bz_3^{t,l})+(2x_1+3)(\bz_2^{-,l})\\
&+(2x_2+3)(D_1^{z,l})+(2x_3-2)(\bz_1\times \bz_l)\\
&=(S_4\times S^1)-(D_4^{z,l})-(\bz_4^{c,l})-(D_2^{d,l})-(D_3^{z,l})-(\bz_3^{t,l})+(\bz_2^{-,l})\\
&+(D_1^{z,l})\quad \text{(basic degree for $S_4\times S^1$)},
 \end{align*}
 and obtain $x_1=-1$, $x_2=-1$ and $x_3=0$.
\end{example}
\vskip.3cm

\section{Periodic Solutions to Symmetric Systems of Newtonian Equations}
Let $\Gamma$ be a finite group and assume that $V=\mathbb{R}^N$ is an
orthogonal $\Gamma$-representation, where the group $\Gamma$ acts on the the
vectors $x=(x_1,x_2,\dots,x_N)$ by permuting their coordinates. More
precisely, we assume that there is given a group homomorphism $\sigma
:\Gamma\to S_N$ (here $S_N$ stands for the permutation group for $N$-elements), then 
\begin{equation*}
\gamma(x_1,x_2,\dots,
x_N):=(x_{\sigma(\gamma)1},x_{\sigma(\gamma)2},\dots,x_{\sigma(\gamma)N}),\quad \gamma\in \Gamma.
\end{equation*}
For a given number $p>0$, we are interested in finding non-constant periodic
solution to the following Newtonian second order system of differential
equations: 
\begin{equation*}  \label{eq:Newtonian}
\begin{cases}
\ddot x(t) =-\nabla f(x(t)), \;\;\;\; x(t)\in V \\ 
x(0)=x(p), \; \dot x(0)=\dot x(p),
\end{cases}
\end{equation*}
where $f:V\to \mathbb{R}$ is a $C^2$-differentiable function. 
By rescaling the time, this problem  can be reduced  to system
\eqref{eq:New-L}, where $\lambda:=\frac p{2\pi} $. 

\vs
We will also introduce the following conditions:

\begin{itemize}
\item[(A1)] $f$ is a $\Gamma$-invariant function, i.e. 
\begin{equation*}
\forall_{x\in V}\; \forall_{\gamma\in \Gamma}\quad f(\gamma x)=f(x)
\end{equation*}

\item[(A2)] There exists a constant $R>0$ such that for all $x\in V$
satisfying $\|x\|>R$ we have 
\begin{equation*}
\nabla f(x)\bullet x <0.
\end{equation*}

\item[(A3)]  $\nabla f(0)=0$ and for a given $p>0$  the operator $A:=\nabla^2 f(0):V\to V$
satisfy the condition 
\begin{equation*}
\left\{\frac{p^2\mu^A}{4\pi^2}: \mu^A\in \sigma(A) \right \}\cap \{k^2:
k=0,1,2,\dots\}=\emptyset.
\end{equation*}
\item[(A4)] $\nabla f(x)\not=0$ for all $x\not=0$. 
\item[(A5)] For all $x\in V^\Gamma\setminus \{0\}$ we have $
 f(-x)=f(x)$. 
 \item[(A6)] There exists a symmetric matrix $B\in \text{GL}(N,\br)$ such that 
 \[
 \lim_{|x|\to\infty} \frac{\nabla  f(x)-Bx}{|x|}=0,
 \]
 and 
\begin{equation*}
\left\{\frac{p^2\mu^B} {4\pi^2}: \mu^B\in \sigma(B) \right \}\cap \{k^2:
k=0,1,2,\dots\}=\emptyset.
\end{equation*}

\item[(A7)] $\nabla f(0)=0$  and  $0\notin \sigma(A)$, where  $A:=\nabla^2 f(0):V\to V$.
\end{itemize}

\begin{lemma}
\label{lem:A4} Let $f:V\to \mathbb{R}$ be a $C^1$-differential function
satisfying the condition (A2). Then there exists a constant $R_o>0$
such that for any non-constant periodic solution to the system 
\begin{equation}  \label{eq:A4}
\ddot x(t)=-\rho\nabla f(x(t))+(1-\rho)x(t),
\end{equation}
where $\rho \in[0,1]$ is a fixed number, we have 
\begin{equation*}
R_o\geq \max\left\{\|x\|_\infty, \|\dot x\|_\infty\right\},
\end{equation*}
where $\|y\|_\infty:=\sup\{ |y(t)|: t\in \mathbb{R}\}$, here $y$ stands for
a continuous periodic function.
\end{lemma}

\begin{proof}
Let $x(t)$ be a non-constant periodic solution to \eqref{eq:A4}. We claim that $|r(t)|\le M$ for all $t\in \br$. Assume for contradiction that  the function
\[
r(t)=\frac{1}{2}|x(t)|^2,  t\in\br.
\]
achieves its maximum/minimum at $t_o$ and  $|r(t_o)|>M$. If $\rho =0$ then $x(t)$ is a zero solution. Assume therefore $\rho\in (0,1]$. 
Then
\[0=r'(t_0)=x(t_0)\bullet\dot x(t_0)\]
and  by (A2)
\begin{align*}
0&= r''(t_0)=|\dot x(t_0)|^2+x(t_0)\bullet\ddot x(t_0)\\
&\geq x(t_0)\bullet\ddot x(t_0)=x(t_0)\bullet(-\rho\nabla f(x(t_0))+(1-\rho)x(t_0))>0,
\end{align*}
which is a contradiction. Thus,  $\sup_{t\in\br}|x(t)|\leq M$.
On the other hand,  for all $t\in \br$
\[|\dot x(t)|=\sup_{|v|\leq1}\dot x(t)\bullet v.\]
Hence, assume that $v\in V$ is an arbitrary vector such that $|v|\leq 1$ and  define
\[\psi_v(t)=\dot x(t)\bullet v, \quad t\in \br.\] 
Then, we have 
\[
\frac d{dt} \psi_v(t)=\ddot x(t)\bullet v, \quad t\in \br.
\]
Put  $M':=\sup_{\|x\| \leq R}\|\nabla f(x)\|$. Since  $\vp_v(t):=x(t)\bullet v$,  $t\in \br$ is a scalar  $p$-periodic function such that $\vp'_v(t)=\psi_v(t)$, there exists  $t_1\in \br$ such that 
$\vp_v(t_1)=\max_{t\in\br}\vp_v(t)$, thus $\vp'_v(t_1)=\psi_v(t_1)=0$ and we have
\begin{align*}
|\psi_v(t)|&=\psi_v(t_1)+\left|\int_{t_1}^t \frac d{dt}\psi_v(\tau)d\tau\right|\\
&=\left|\int_{t_1}^t \frac d{dt}\psi_v(\tau)d\tau\right|\le \int_{t_1}^t\left| \frac d{dt}\psi_v(\tau)\right|d\tau\\
&=\int_{t_1}^{t} |\ddot x(\tau)\bullet v| d\tau=\int_{t_1}^t |-\rho \nabla f(x(\tau)+(1-\rho)x(\tau))\bullet v)|d\tau\\
&\leq \int_{t_1}^{t}\left(\rho|\nabla f(x(\tau))+(1-\rho)|x(\tau)|\right)|\cdot |v|d\tau\leq \int_{t_1}^{t}\left(|\nabla f(x(\tau))|+|x(\tau)|\right)d\tau\\
&\leq|t-t_1|(M'+M) \le p(M'+M)=:M_1.
\end{align*}
Therefore, we obtain that
\[ \forall_{|v|\leq1}\:\forall_{t\in\br} \;\; \dot x(t)\bullet v\leq M_1.\]
which implies
\[\forall_{t\in \br}\;\;\;|\dot x(t)|\leq M_1.\]
Consequently,
\[ \forall_{t\in \br} \;\; \max\left\{\sup_{t\in\br}|x(t)|, \sup_{t\in\br}|\dot x(t)|\right\}\le \max\{M, M_1  \}=:R_o.
\]
\end{proof}

\vskip.3cm

\paragraph{\bf Sobolev Spaces of $2\pi$-Periodic Functions:}

\label{sec:H}Let $V=\mathbb{R}^N$ be an orthogonal representation
of a finite group $\Gamma$ and let $\mathscr{H}$ denote the first Sobolev
space of $2\pi$-periodic functions from $\mathbb{R}$ to $V$, i.e. 
\begin{equation*}
\mathscr{H}:=H^1_{2\pi}(\mathbb{R},V)=\{x:\mathbb{R}\to V\;:\; x(0)=x(2\pi), \;
x|_{[0,2\pi]}\in H^1([0,2\pi];V)\},
\end{equation*}
equipped with the inner product 
\begin{equation*}
\langle x,y\rangle:=\int_0^{2\pi}(\dot x(t)\bullet \dot y(t)+x(t)\bullet
y(t))dt,\quad x,y\in H^1_{2\pi}(\mathbb{R},V).
\end{equation*}
Let $O(2)$ denote the group of $2\times 2$-orthogonal matrices.
Notice that $O(2)=SO(2)\cup SO(2)\kappa$, where $\kappa=\left[ 
\begin{array}{cc}
1 & 0 \\ 
0 & -1
\end{array}
\right]$. It is convenient to identify a rotation $\left[ 
\begin{array}{cc}
\cos \tau & -\sin\tau \\ 
\sin \tau & \cos \tau
\end{array}
\right]\in SO(2)$ with $e^{i\tau}\in S^1\subset \mathbb{C}$. Notice that $\kappa e^{i\tau}=e^{-i\tau}\kappa$. Then the space $\mathscr{H}$
is an orthogonal Hilbert representation of $G:=\Gamma\times O(2)$. Indeed,
we have for $x\in \mathscr{H}$ and $\gamma\in \Gamma$, $e^{i\tau}\in S^1$ we
have 
\begin{align*}
((\gamma , e^{i\tau})x)(t) &= \gamma x\left(t+\tau\right) \\
((\gamma ,e^{i\tau}\kappa)x)(t) &= \gamma x\left(-t+\tau\right).
\end{align*}

It is useful to identify a $2\pi$-periodic function $x:\mathbb{R}\to V$ with a
function $\widetilde x:S^1\to V$ via the following commuting diagram: 
\newline
\vglue2cm\hskip4cm 
\rput(0,2){$\mathbb{R}$} 
\rput(3,2){$S^1$} 
\psline{->}(0.5,2)(2.5,2) 
\rput(1.5,0){$V$} 
\psline{->}(0.2,1.7)(1.3,.3) 
\psline{->}(2.8,1.7)(1.8,.3) 
\rput(1.5,2.3){$e$} 
\rput(1,1.){$x$} 
\rput(2,1){$\widetilde x $} 
\rput(5,1){$e(\tau)=e^{i\tau}$}

\vskip.3cm
 \noindent Using this identification, we will write $H^1(S^1,V)$
instead of $H^1_{2\pi}(\mathbb{R},V)$. In addition, notice that for $x\in 
\mathscr{H}$, the isotropy group $G^{\prime }_x$ is finite if and only if $x$
is a non-constant periodic function.
\vs
Consider the $O(2)$-isotypical decomposition of $\mathscr{H}$, which is: 
\begin{equation}  \label{eq:isoS1}
\mathscr{H}=\overline{ \bigoplus_{k=0}^\infty \mathscr{V}_k},\quad \mathscr{V}_k:=\{u_k\cos(k\cdot t)+v_k\sin(k\cdot t): u_k,v_k\in V\}.
\end{equation}
Each of the components $\mathscr{V}_k$ can be identified $G$-equivariantly to the following space  of complex $2\pi$-periodic functions  
\begin{equation*}
{\mathscr{V}}_k\simeq \left\{e^{i kt}(x_k+iy_k):
x_k,\,y_k\in V\right\}, \quad \text{ where }\quad x_k=\frac{u_k+v_k}{2},\;\;
y_k=\frac{u_k-v_k}{2}.
\end{equation*}
One can easily notice that the component ${\mathscr{V}}_0$ is $G$-equivalent to $V$ and for $k>0$ the component  $ {\mathscr{V}}_k$ is
$G$-equivalent to the complexification $V^c:=V\oplus V$ of $V$, where $\kappa$
acts on vectors in $V^c$ by conjugation and for $e^{i\tau}\in SO(2)\simeq
S^1 $, $e^{i\tau}z=e^{ik\tau}\cdot z$, where `$\cdot$' denotes complex
multiplication and $z=x+iy\in V^c$. Assume  next that  
\[
V=V_0\oplus V_1\oplus \dots \oplus V_r
\]
is a real $\Gamma$-isotypical decomposition of $V$ (with $V_j$ modeled on the $j$-th irreducible  real $\Gamma$-representation $\cV_j$)
and
\[
V^c=U_0\oplus U_1\oplus \dots \oplus U_s
\]
is a complex $\Gamma$-isotypical decomposition of $V^c$ (with $U_j$ modeled on the $j$-th irreducible  complex $\Gamma$-representation $\cU_j$). Then the $G$-isotypical decomposition of $\mathscr H$ is given by 
\[
\mathscr H= \bigoplus_{j=0}^r V_k\oplus \overline{\bigoplus_{ k=1}^\infty\bigoplus_{j=0}^s \mathscr {V}_{k,j}  }, 
\]
where $\mathscr V_{k,j}:=\{ e^{kt}z: z\in U_j \}$. Clearly, the isotypical component is modeled on the irreducible $G$-representation $\mathcal W_{k,j}:= \cV_j\otimes \cU_k$.
\vs 
For $k>0$ we put 
\[
\Phi^{\mathfrak f}_k:=\{ (G_x): x\in \mathscr {V}_k\;\text{ and $x$ is a non-constant function}\}\;\;\;\text{ and } \;\;\; \Phi^{\mathfrak f}:=\bigcup_{k>0}\Phi^{\mathfrak f}_k.
\]
Clearly, the order relation in $\Phi(G;\mathscr H)$ induces an order in $\Phi^{\mathfrak f}_k$. Notice that $\Phi^{\mathfrak f}_k$ is finite set.
\vs
\begin{definition}\rm An orbit type $(H)$ in $\mathscr H$ will be called a {\it maximal $\mathfrak f$-orbit type} if and only if $(H)$ is a maximal element in $\Phi^{\mathfrak f}_k$ for some $k>0$. Moreover, two maximal $\mathfrak f$-orbit types $(H)$ and $(H')$ will be called {\it similar} if there exists $(K)\in \Phi^{\mathfrak f}$ such that $(K)\ge (H)$ and $(K)\ge (H')$.
\end{definition}
\vs
Notice that $\mathfrak f$-orbit types in $\mathscr H$ belong to $\Phi_0(G;\mathscr H)$ and the similarity of the $\mathfrak f$-orbit types is an equivalent relation. Thus, we will call $\mathfrak f$-orbit types belonging to different equivalence classes {\it dissimilar}. 
\vs
\paragraph{\bf Variational Reformulation of \eqref{eq:New-L}: }

Using the standard arguments the system \eqref{eq:New-L} can be written as the following variational equation
\begin{equation}\label{eq:var1}
\nabla_u J(\lambda,u)=0, \quad (\lambda,u)\in \br\times \mathscr H,
\end{equation}
where $J:\br\times \mathscr H \to \br$ is defined by
\begin{equation}\label{eq:var-1}
J(\lambda,u)=\int_0^{2\pi} \left[ \frac 12 |\dot u(t)|^2-\lambda^2 V(u(t))  \right]dt.
\end{equation}
The gradient $\nabla_u J(\lambda,u)$  can be explicitly expressed as
\begin{equation}\label{eq:grad-J}
\nabla_u J(\lambda,u)=u-jL^{-1}\left(\lambda^2 N_{\nabla f}(u)+u   \right), \quad u\in \mathscr H,
\end{equation}
where 
\begin{align*}
j:H^2(S^1,V)\to H^1(S^1,V),\quad ju=u\\
L:H^2(S^1,V)\to L^2(S^1,V), \quad Lu=-\ddot u+u\\
N_{\nabla f}:C(S^1,V)\to L^2(S^1,V),\quad N_{\nabla f}(u)(t):=\nabla f(u(t)),
\end{align*}
which implies that $\nabla_u J$ is a compact vector field on $\br\times \mathscr H$. Moreover, if $\nabla f$ is continuously differentiable at $0$, with the Hessian matrix $\nabla^2f(0)$ then $\nabla_u J(\lambda,u)$ is also differentiable at $0$ and 
\[
\nabla^2 J(\lambda,0)=\id -jL^{-1}\left(\lambda^2 N_{\nabla^2 f(0)}+\id   \right),
\]
where
\[
N_{\nabla^2 f(0)}u(t)=\nabla^2f(0)u(t), \quad u\in C(S^1,V).
\]
We put 
\begin{equation}\label{eq:A}
\mathscr A:= \id -jL^{-1}\left(\lambda^2 N_{\nabla^2 f(0)}+\id   \right).
\end{equation}
Notice that if $p>0$ is a fixed number and $\lambda=\frac p{2\pi}$, then $\mathscr A:\mathscr H\to \mathscr H$ is an isomorphism if and only if the condition (A3) is satisfied.
\vs
\subsection{Existence Results for \eqref{eq:New-L} for a Fixed $p>0$}
\paragraph{\bf Asymptotically Linear System:} In this paragraph we will assume that $f$ satisfies the assumptions (A1), (A3), (A4) and (A6). Since $p>0$ is fixed we assume that $\lambda:=\frac p{2\pi}$ is also fixed and consider  $J:\mathscr H \to \br$ given by \eqref{eq:var-1}. 
\vs
We put 
\begin{equation}\label{eq:B}
\mathscr B:= \id -jL^{-1}\left(\lambda^2 N_{B}+\id   \right),
\end{equation}
where the matrix $B$ is given in (A6). By applying standard arguments (see for example \cite{RR,RY1}) one can easily obtain the following 
\vs
\begin{proposition}\label{pro:e-R} Under the assumptions (A1), (A3), (A6) and (A7), there exists a sufficiently small  $\ve>0$ and a sufficiently large $R>0$ such that 
\begin{itemize}
\item[(a)] 
$\nabla J$ and $\mathscr A$ are $\Omega_0$-admissibly $G$-homotopic (here $\Omega_0:=B_\ve(0)$ in $\mathscr H$),
\item[(b)] $\nabla J$ and $\mathscr B$ are $\Omega_\infty$-admissibly $G$-homotopic (here $\Omega_\infty:=B_R(0)$ in $\mathscr H$),
\end{itemize}
\end{proposition} 
\vs
Using Proposition \ref{pro:e-R}, we can define $\Omega:=\Omega_\infty\setminus \overline{\Omega_0}$ so we have (by properties of the gradient degree)
\begin{align*}
\nabla_G\text{-deg}(\nabla J,\Omega)&=\nabla_G\text{-deg} (\nabla J,\Omega_\infty)-\nabla_G\text{-deg}(\nabla J,\Omega_0)\\
&=\nabla_G\text{-deg} (\mathscr B,\Omega_\infty)-\nabla_G\text{-deg}(\mathscr A,\Omega_0).
\end{align*}
Consequently we obtain the following result
\vs
\begin{theorem}\label{th:exist1}
 Under the assumptions (A1), (A3), (A6) and (A7), if 
 \[
 \nabla_G\text{\rm-deg}(\nabla J,\Omega)=n_1(H_1)+n_2(H_2)+\dots+ n_k(H_k), \quad n_j\not=0, \; j=1,,2,\dots,k,
 \]
then  for every $j=1,2,\dots,k$, the system \eqref{eq:New-L} has a non-constant periodic solution $x_j$ such that $(G_{x_j})\ge (H_j)$. Moreover, if $(H_j)$ is a maximal $\mathfrak f$-orbit type in $\mathscr H$ then the orbit $G(x_j)$ contains $\left| G/H_j \right|_{S^1}$ periodic non-constant solutions (where $\left| G/H_j \right|_{S^1}$  stands for the number of $S^1$-orbits in $G/H_j$).
\end{theorem}
\vs
Notice that we have 
\begin{align*}
\sigma(\mathscr A)&:=\bigcup_{k=0}^\infty \sigma^k(\mathscr A), \quad \sigma^k(\mathscr A):=\left\{1-\frac {p^2\mu^A+4\pi^2}{4\pi^2(k^2+1)}  :  \mu^A\in \sigma(A) \right\},\\
\sigma(\mathscr B)&:=\bigcup_{k=0}^\infty \sigma^k(\mathscr B), \quad \sigma^k(\mathscr B):=\left\{1-\frac {p^2\mu^B+4\pi^2}{4\pi^2(k^2+1)}  :  \mu^B\in \sigma(B) \right\}.
\end{align*}
Moreover,  for $\xi:=1-\frac {p^2\mu^A+4\pi^2}{4\pi^2(k^2+1)} $ (resp.  $\xi=1-\frac {p^2\mu^B+4\pi^2}{4\pi^2(k^2+1)} $) we have that $m_{k,i}(\xi)=m_i(\mu^A)$ (resp.  $m_{k,i}(\xi)=m_i(\mu^A)$). Put,
\[
\sigma^k_0:=\{ \xi \in \sigma^k(\mathscr A): \xi<0\} \;\;\text{ and }\;\; \sigma^k_\infty:=\{ \xi \in \sigma^k(\mathscr B): \xi<0\}.
\] 
and 
Therefore, by \eqref{eq:grad-lin} we have (by assumption (A7))
\begin{align*}
 \nabla_G\text{\rm-deg}(\nabla J,\Omega)&= \mathfrak a *\left[\prod_{k=1}^\infty \prod_{\xi\in \sigma^k_\infty}\prod_{j=0}^s\left(\text{\rm Deg}_{\mathcal W_{kj}}\right)^{m_{k,j}(\xi)}
-\prod_{k=1}^\infty \prod_{\xi\in \sigma^k_0}\prod_{j=0}^s\left(\text{\rm Deg}_{\mathcal W_{kj}}\right)^{m_{k,j}(\xi)}\right]\\
\mathfrak a&:= \prod_{\xi\in \sigma^0_\infty}\prod_{j=0}^r\left(\text{\rm Deg}_{\mathcal V_{j}}\right)^{m_{0,j}(\xi)}=\prod_{\xi\in \sigma^0_0}\prod_{j=0}^r\left(\text{\rm Deg}_{\mathcal V_{j}}\right)^{m_{0,j}(\xi)}
\end{align*}
\vs
\begin{example}{\bf Asymptotically Linear System:}\rm \label{ex:1} \rm
 Let us provide an example of a function
    $f:V\to\br$ satisfying the assumptions
    (A1), (A3), (A6) and (A7). Recall that
    $V:=\br^N$ is an orthogonal
    $\Gamma$-representation and consider
    two symmetric $\Gamma$-equivariant
    $(N\times N)$-matrices such that
    $A-B$ (or $B-A$) are positive matrices.
    Then, for $x\in V$ we define
    \begin{align}
      f(x):=\begin{cases}
        \frac{1}{2}(Bx\bullet x)+
            \sqrt{(A-B)x\bullet x+1}\quad\mbox{if }A-B>0,\\
        \frac{1}{2}(Bx\bullet x)+
            \sqrt{(B-A)x\bullet x+1}\quad\mbox{if }A-B<0.
      \end{cases}
    \end{align}
    In order to assure that the assumption (A7)
    is satisfied, notice that, if $\nabla f(x)=0$
    for some $x\neq 0$, then $Bx+\alpha(A-B)x=0$
    for $\alpha=(\abs{(A-B)x\bullet x}+1)^{-\frac{1}{2}}<1$.
    Therefore, (A7) is satisfied if
    \begin{align*}
      \forall_{\alpha\in(0,1)}\quad
          0\notin\sigma(B+\alpha(A-B)).
    \end{align*}
    As a {\bf numerical example}, we consider
    $p=2\pi$ and $\Gamma=S_4$ acting on $V=\br^4$
    by permuting the coordinates of vectors.
    The space $V$ has the $S_4$-representations
    \begin{align}
      V=V_0\oplus V_3,
    \end{align}
    where $V_0$ is a one-dimensional trivial
    $S_4$-representation and $V_3$ is
    a 3-dimensional standard $S_4$-representation
    (we use here the convention introduced in [6]).
    Assume that $f:V\to\br$ satisfies (A1), (A6)
    and both matrices $A$ and $B$ are of the
    type
    \begin{align*}
      C:=\begin{bmatrix}
        c & d & d & d \\
        d & c & d & d \\
        d & d & c & d \\
        d & d & d & c
      \end{bmatrix}.
    \end{align*}
    The spectrum $\sigma(C)$ consists of 
    two eigenvalues $\mu_0=c+3d$ and
    $\mu_3=c-d$. here $E(\mu_0)=V_0$
      and $E(\mu_3)=V_3$. That means
    $m_0(\mu_0)=m_3(\mu_3)=1$.
    We choose the following values of $c$
    and $d$ for the matrices $A$ and $B$
    \begin{align*}
      A:\quad & c=18.8,\,d=-1.5, \\
      B:\quad & c=2.1,\,d=1.2.
    \end{align*}
    Then
    \begin{align*}
      \sigma(A)=\set{\mu_0^A=14.3,\mu_3^A=20.3}\quad
          \mbox{and}\quad\sigma(B)=
          \set{\mu_0^B=5.7,\mu_3^B=0.9},
    \end{align*}
    which implies the following list for $\sigma_0$ and
    $\sigma_\infty$ (see Table \ref{tbl:list_sigma}).

    \begin{table}[h]
      \newcolumntype{Y}{>{\centering\arraybackslash$}X<{$}}
      \newcolumntype{C}{>{$}c<{$}}
      \newcolumntype{L}{>{$}l<{$}}
      \centering
      \begin{tabularx}{\textwidth}{@{}Cp{.3cm}YYp{.3cm}YY@{}}
        \toprule
        & & \multicolumn{2}{c}{$\sigma_0$} &
            & \multicolumn{2}{c}{$\sigma_\infty$} \\
        \cmidrule{3-4}\cmidrule{6-7}
        k & & \sigma^{k,1}(\mathscr{A}) &
            \sigma^{k,3}(\mathscr{A}) & &
            \sigma^{k,1}(\mathscr{B}) &
            \sigma^{k,3}(\mathscr{B}) \\
        \midrule
        0 & & -14.3 & -20.3 & & -5.7 & -0.9 \\
        1 & & -6.65 & -9.65 & & -2.35 & \\
        2 & & -2.06 & -3.26 & & -0.34 & \\
        3 & & -0.53 & -1.13 & & & \\
        4 & & & -0.25 & & & \\
        \bottomrule
      \end{tabularx}
      \caption{List of $\sigma_0$ and $\sigma_\infty$}
      \label{tbl:list_sigma}
    \end{table}

  \noindent
  It follows that
  \begin{align*}
    \eqdeg{\nabla_G}(\nabla J,\Omega)=\;&
      {\boldsymbol{(\amal{D_4}{D_{2}}{\mathbb{Z}_{2}}{D_2}{})}}- 
      {\boldsymbol{(\amal{D_2}{D_{2}}{\mathbb{Z}_{2}}{D_1}{})}}+ 
      (\amal{D_1}{D_{2}}{\mathbb{Z}_{2}}{\mathbb{Z}_1}{})+ 
      (\amal{D_3}{D_{1}}{\mathbb{Z}_{1}}{D_3}{})\\ 
      &-(\amal{D_1}{D_{1}}{\mathbb{Z}_{1}}{D_1}{})+ 
      {\boldsymbol{(\amal{D_4}{D_{4}}{D_{4}}{\mathbb{Z}_1}{})}}+ 
      (\amal{D_3}{D_{3}}{D_{3}}{\mathbb{Z}_1}{})+ 
      3(\amal{D_2}{D_{2}}{D_{2}}{\mathbb{Z}_1}{D_1})\\ 
      &+5(\amal{D_2}{D_{2}}{D_{2}}{\mathbb{Z}_1}{\mathbb{Z}_2})- 
      (\amal{V_4}{D_{2}}{D_{2}}{\mathbb{Z}_1}{})- 
      (\amal{D_4}{D_{1}}{D_{1}}{D_2}{})- 
      84(\amal{D_1}{D_{1}}{D_{1}}{\mathbb{Z}_1}{})\\ 
      &+146(\amal{\mathbb{Z}_2}{D_{1}}{D_{1}}{\mathbb{Z}_1}{})- 
      {\boldsymbol{(\amal{D_4}{D_{4}}{\mathbb{Z}_{2}}{D_2}{})}}- 
      {\boldsymbol{(\amal{D_2}{D_{4}}{\mathbb{Z}_{2}}{D_1}{})}}+ 
      (\amal{\mathbb{Z}_2}{D_{4}}{\mathbb{Z}_{2}}{\mathbb{Z}_1}{})\\ 
      &-(\amal{D_3}{D_{2}}{\mathbb{Z}_{1}}{D_3}{})+ 
      2(\amal{D_1}{D_{2}}{\mathbb{Z}_{1}}{D_1}{})- 
      (\amal{\mathbb{Z}_1}{D_{2}}{\mathbb{Z}_{1}}{\mathbb{Z}_1}{})+ 
      {\boldsymbol{(\amal{D_4}{D_{8}}{D_{4}}{\mathbb{Z}_1}{})}}\\ 
      &-(\amal{D_3}{D_{6}}{D_{3}}{\mathbb{Z}_1}{})- 
      (\amal{D_2}{D_{4}}{D_{2}}{\mathbb{Z}_1}{\mathbb{Z}_2})+ 
      (\amal{D_4}{D_{2}}{D_{1}}{D_2}{})+ 
      (\amal{D_2}{D_{2}}{D_{1}}{D_1}{})\\ 
      &-(\amal{D_1}{D_{2}}{D_{1}}{\mathbb{Z}_1}{})+ 
      2(\amal{\mathbb{Z}_2}{D_{2}}{D_{1}}{\mathbb{Z}_1}{})- 
      {\boldsymbol{(\amal{D_4}{D_{6}}{\mathbb{Z}_{2}}{D_2}{})}}+ 
      {\boldsymbol{(\amal{D_2}{D_{6}}{\mathbb{Z}_{2}}{D_1}{})}}\\ 
      &-(\amal{D_1}{D_{6}}{\mathbb{Z}_{2}}{\mathbb{Z}_1}{})- 
      {\boldsymbol {(\amal{S_4}{D_{3}}{\boldsymbol{\mathbb{Z}}_{1}}{S_4}{})}}+ 
      (\amal{D_3}{D_{3}}{\mathbb{Z}_{1}}{D_3}{})+ 
      (\amal{D_2}{D_{3}}{\mathbb{Z}_{1}}{D_2}{})\\ 
      &-2(\amal{D_1}{D_{3}}{\mathbb{Z}_{1}}{D_1}{})+ 
      (\amal{\mathbb{Z}_1}{D_{3}}{\mathbb{Z}_{1}}{\mathbb{Z}_1}{})- 
      {\boldsymbol{(\amal{D_4}{D_{12}}{D_{4}}{\mathbb{Z}_1}{})}}+ 
      (\amal{D_3}{D_{9}}{D_{3}}{\mathbb{Z}_1}{})\\ 
      &-(\amal{D_2}{D_{6}}{D_{2}}{\mathbb{Z}_1}{D_1})+ 
      (\amal{D_2}{D_{6}}{D_{2}}{\mathbb{Z}_1}{\mathbb{Z}_2})+ 
      (\amal{V_4}{D_{6}}{D_{2}}{\mathbb{Z}_1}{})+ 
      (\amal{D_4}{D_{3}}{D_{1}}{D_2}{})\\ 
      &-{\boldsymbol{(\amal{D_4}{D_{8}}{\mathbb{Z}_{2}}{D_2}{})}}+ 
      {\boldsymbol{(\amal{D_2}{D_{8}}{\mathbb{Z}_{2}}{D_1}{})}}- 
      (\amal{D_1}{D_{8}}{\mathbb{Z}_{2}}{\mathbb{Z}_1}{})+ 
      (\amal{D_3}{D_{4}}{\mathbb{Z}_{1}}{D_3}{})\\ 
      &+(\amal{D_2}{D_{4}}{\mathbb{Z}_{1}}{D_2}{})- 
      2(\amal{D_1}{D_{4}}{\mathbb{Z}_{1}}{D_1}{})+ 
      (\amal{\mathbb{Z}_1}{D_{4}}{\mathbb{Z}_{1}}{\mathbb{Z}_1}{})- 
      {\boldsymbol{(\amal{D_4}{D_{16}}{D_{4}}{\mathbb{Z}_1}{})}}\\ 
      &+(\amal{D_3}{D_{12}}{D_{3}}{\mathbb{Z}_1}{})- 
      (\amal{D_2}{D_{8}}{D_{2}}{\mathbb{Z}_1}{D_1})+ 
      (\amal{D_2}{D_{8}}{D_{2}}{\mathbb{Z}_1}{\mathbb{Z}_2})+ 
      (\amal{V_4}{D_{8}}{D_{2}}{\mathbb{Z}_1}{})\\ 
      &-(\amal{D_2}{D_{4}}{D_{1}}{D_1}{})- 
      (\amal{D_1}{D_{4}}{D_{1}}{\mathbb{Z}_1}{}), 
  \end{align*}
  where the orbit types in bold are maximal ones.

\end{example}

\vs
\begin{example} {\bf System Satisfying Nagumo Condition:} \rm  Let us assume that the function $f:V\to \br$ satisfies the conditions (A1), (A2), (A3) and (A5).  Again we assume that  $p>0$ is fixed, i.e.  we assume that $\lambda:=\frac p{2\pi}$ is a  fixed number. Clearly, the space $\mathscr H$ is a Hilbert representation of $G:=\Gamma\times \bz_2\times O(2)$ (where $\bz_2$ acts on $x\in \mathscr H$ by simple multiplication) and the functional   $J:\mathscr H \to \br$ given by \eqref{eq:var-1} is $G$-invariant. One can easily show that, by Lemma \ref{lem:A4}, there exists a sufficiently $R>0$ such that $\nabla J$ is $B_R(0)$-admissible $G$-homotopic to $\id$. 
\vs
Consider the projection $\pi_2:\Gamma\times \bz_2\times O(2)\to \bz_2$ onto the $\bz_2$ factor. Clearly, if $x\in \mathscr H$ is a constant function then $G_x=K\times \{1\}\times O(2)$, thus $\pi_2(G_x)=\bz_1\subset \bz_2$. On the other hand, if a non-constant function $x\in \mathscr H$ satisfies the condition $x\left(\frac \pi 2+t  \right)=-x\left( \frac \pi 2-t \right)$ for $t\in \br$ then $\pi_2(G_x)=\bz_2$, i.e. the homomorphism $\pi_2:G_x\to \bz_2$ is surjective. This justifies the following definition:
\vs
\begin{definition}\label{def:a-r}\rm
An orbit type $(H)$ in $\mathscr H$ is called {\it $\frac \pi 2$-anti-reflective} (or simply {\it anti-reflective}) if $\pi_2:H\to \bz_2$ is surjective. We will also call the anti-reflective orbit types {\it $\mathfrak {ar}$-types}.
\end{definition}
\vs
Using Proposition \ref{pro:e-R}, we can define $\Omega_0:=B_\ve(0)$, $\Omega_\infty:=B_R(0)$  and  $\Omega:=\Omega_\infty\setminus \overline{\Omega_0}$.
\vs
\begin{remark}\label{rem:ar}\rm   Notice that if $(H)$ is an $\mathfrak {ar}$-orbit type in $\Omega$ and $(K)\ge (H)$ then $(K)$ is also $\mathfrak{ar}$-orbit type.
\end{remark}

\vs

 Then  we have (by properties of the gradient degree)
\begin{align*}
\nabla_G\text{-deg}(\nabla J,\Omega)&=\nabla_G\text{-deg} (\nabla J,\Omega_\infty)-\nabla_G\text{-deg}(\nabla J,\Omega_0)\\
&=\nabla_G\text{-deg} (\id,\Omega_\infty)-\nabla_G\text{-deg}(\mathscr A,\Omega_0)\\
&=(G)-\nabla_G\text{-deg}(\mathscr A,\Omega_0)
\end{align*}
where
\[
\nabla_G\text{-deg}(\mathscr A,\Omega_0)=\mathfrak a* \prod_{k=1}^\infty \prod_{\xi\in \sigma^k_0}\prod_{j=0}^s\left(\text{\rm Deg}_{\mathcal W_{kj}}\right)^{m_{k,j}(\xi)}.
\]

Consequently we obtain the following result
\vs
\begin{theorem}\label{th:exist1}
 Under the assumptions (A1), (A2), (A3) and (A5), if 
 \[
 \nabla_G\text{\rm-deg}(\nabla J,\Omega)=n_1(H_1)+n_2(H_2)+\dots+ n_k(H_k), \quad n_j\not=0, \; j=1,,2,\dots,k,
 \]
then  for every $j=1,2,\dots,k$, such that $(H_j)$ is an $\mathfrak {ar}$-orbit type, the system \eqref{eq:New-L} has a non-constant periodic solution $x_j$ such that $(G_{x_j})\ge (H_j)$. Moreover, if $(H_j)$ is a maximal  $\mathfrak f$-orbit type in $\mathscr H$ then the orbit $G(x_j)$ contains $\left| G/H_j \right|_{S^1}$ periodic non-constant solutions (where $\left| G/H_j \right|_{S^1}$  stands for the number of $S^1$-orbits in $G/H_j$).
\end{theorem}
\vs

  As a numerical example, we consider $p=2\pi$ and $\Gamma=D_4$
    acting on $V=\br^4$ by permuting the coordinates of
    vectors, i.e.,
    \begin{align*}
      \gamma(x_1,x_2,x_3,x_4)^T=(x_4,x_1,x_2,x_3)^T\quad\mbox{and}\quad
      \kappa(x_1,x_2,x_3,x_4)^T=(x_4,x_3,x_2,x_1)^T.
    \end{align*}
    The space $V$ has $(D_4\times\bz_2)$-isotypical decomposition
    \begin{align*}
      V\;&=V_0^-\oplus V_1^-\oplus V_2^-\\
      &=V_0\otimes\sgnrep{\bz_2}\oplus V_1\otimes\sgnrep{\bz_2}\oplus V_2\otimes\sgnrep{\bz_2},
    \end{align*}
    where $V_0$ is the one-dimensional trivial $D_4$-representation,
    $V_1$ is the 2-dimensional $D_4$-representation
    (we use here the convention introduced in []),
    $V_2$ is the one-dimensional $D_4$-representation on which
    $\widetilde{D}_2$ acts trivially and $\sgnrep{\bz_2}$ is the
    one-dimensional faithful $\bz_2$-representation.
    To be more explicit,
    \begin{align*}
      &V_0^-=\myspan\set{(1,1,1,1)^T},\quad
          V_2^-=\myspan\set{(1,-1,1,-1)^T},\\
      &V_1^-=\myspan\set{(1,0,-1,0)^T,(0,1,0,-1)^T}.
    \end{align*}
    Assume that $f:V\to\br$ satisfies (A1), (A6) and
    \begin{align*}
      A:=\begin{bmatrix}
        c & d & 0 & d \\
        d & c & d & 0 \\
        0 & d & c & d \\
        d & 0 & d & c
      \end{bmatrix}.
    \end{align*}
    Clearly, $\sigma(A)=\set{\mu_0=c+2d,\mu_1=c,\mu_2=c-2d}$,
    where the corresponding eigenspaces are $V_0^-$, $V_1^-$ and $V_2^-$.
    To demonstrate the computation, take $c=3.1$, $d=1.2$.  Then,
    \begin{align*}
      \sigma(A)=\set{\mu_0=5.5,\mu_1=3.1,\mu_2=0.7},
    \end{align*}
    which implies the following list for $\sigma_0$ (see \Cref{tbl:list2_sigma}).

    \begin{table}[h]
      \newcolumntype{Y}{>{\centering\arraybackslash$}X<{$}}
      \newcolumntype{C}{>{$}c<{$}}
      \newcolumntype{L}{>{$}l<{$}}
      \centering
      \begin{tabularx}{\textwidth}{@{}Cp{.3cm}YYY@{}}
        \toprule
        k & & \sigma^{k,0}(\mathscr{A}) &
            \sigma^{k,1}(\mathscr{A}) &
            \sigma^{k,2}(\mathscr{A}) \\
        \midrule
        0 & & -5.5 & -3.1 & -0.7 \\
        1 & & -2.25 & -1.05 & \\
        2 & & -0.3 & & \\
        \bottomrule
      \end{tabularx}
      \caption{List of $\sigma_0$}
      \label{tbl:list2_sigma}
    \end{table}

    It follows that
    \begin{align*}
      \eqdeg{\nabla_G}(\nabla J,\Omega)=\;
          &(\amal{D_4}{O(2)}{}{}{})
          +(\amal{D_4^\td}{O(2)}{}{}{})
          +(\amal{\tD_2^\td}{O(2)}{}{}{})\\
          &-(\amal{\tD_2}{O(2)}{}{}{})
          +(\amal{D_2^d}{O(2)}{}{}{})
          -(\amal{D_1^d}{O(2)}{}{}{})\\
          &-(\amal{D_1}{O(2)}{}{}{})
          -(\amal{\bz_2^z}{O(2)}{}{}{})
          +(\amal{\bz_1}{O(2)}{}{}{})\\
          &+(\amal{D_4^2}{D_{2}}{\bz_{2}}{D_4}{})
          -(\amal{D_4^\td}{D_{2}}{\bz_{2}}{\tD_2}{})
          +\boldsymbol{(\amal{D_2^2}{D_{2}}{\bz_{2}}{D_2^d}{})}\\
          &+\boldsymbol{(\amal{\tD_2^2}{D_{2}}{\bz_{2}}{\tD_2^\td}{})}
          -(\amal{\tD_1^2}{D_{2}}{\bz_{2}}{\tD_1}{})
          -(\amal{\tD_2^\td}{D_{2}}{\bz_{2}}{\tD_1}{})\\
          &-(\amal{\tD_2^\td}{D_{2}}{\bz_{2}}{\bz_2^z}{})
          -(\amal{\tD_2}{D_{2}}{\bz_{2}}{\tD_1}{})
          -(\amal{D_1^2}{D_{2}}{\bz_{2}}{D_1}{})\\
          &-(\amal{D_2^d}{D_{2}}{\bz_{2}}{D_1}{})
          -(\amal{D_2^d}{D_{2}}{\bz_{2}}{\bz_2^z}{})
          -(\amal{D_2^z}{D_{2}}{\bz_{2}}{D_1^d}{})\\
          &-(\amal{\bz_2^2}{D_{2}}{\bz_{2}}{\bz_2^z}{})
          +2(\amal{D_1^d}{D_{2}}{\bz_{2}}{\bz_1}{})
          +(\amal{\tD_1^\td}{D_{2}}{\bz_{2}}{\bz_1}{})\\
          &+(\amal{\tD_1}{D_{2}}{\bz_{2}}{\bz_1}{})
          +(\amal{\bz_2^z}{D_{2}}{\bz_{2}}{\bz_1}{})
          +(\amal{\bz_2}{D_{2}}{\bz_{2}}{\bz_1}{})\\
          &+(\amal{\bz_1^2}{D_{2}}{\bz_{2}}{\bz_1}{})
          -(\amal{\tD_2^\td}{D_{1}}{}{}{})
          -(\amal{D_2^d}{D_{1}}{}{}{})\\
          &+(\amal{D_1^d}{D_{1}}{}{}{})
          +(\amal{D_1}{D_{1}}{}{}{})
          +2(\amal{\tD_1}{D_{1}}{}{}{})\\
          &+2(\amal{\bz_2^z}{D_{1}}{}{}{})
          -4(\amal{\bz_1}{D_{1}}{}{}{})
          +\boldsymbol{(\amal{D_4^2}{D_{4}}{D_{4}}{\bz_2^z}{})}\\
          &-(\amal{D_4^d}{D_{4}}{D_{4}}{\bz_1}{})
          -(\amal{D_4}{D_{4}}{D_{4}}{\bz_1}{})
          -(\amal{D_2^2}{D_{2}}{D_{2}}{D_1}{D_1^2})\\
          &-(\amal{D_2^2}{D_{2}}{D_{2}}{D_1}{D_2})
          -(\amal{D_2^2}{D_{2}}{D_{2}}{\bz_2^z}{\bz_2^2})
          -(\amal{\tD_2^2}{D_{2}}{D_{2}}{\tD_1}{\tD_1^2})\\
          &-(\amal{\tD_2^2}{D_{2}}{D_{2}}{\tD_1}{\tD_2})
          -(\amal{\tD_2^2}{D_{2}}{D_{2}}{\bz_2^z}{\bz_2^2})
          +(\amal{\tD_1^2}{D_{2}}{D_{2}}{\bz_1}{\bz_1^2})\\
          &+(\amal{\tD_2^\td}{D_{2}}{D_{2}}{\bz_1}{\tD_1^\td})
          +(\amal{\tD_2^\td}{D_{2}}{D_{2}}{\bz_1}{\tD_1})
          +(\amal{\tD_2^z}{D_{2}}{D_{2}}{\bz_1}{\bz_2})\\
          &+(\amal{\tD_2}{D_{2}}{D_{2}}{\bz_1}{\bz_2})
          +(\amal{D_1^2}{D_{2}}{D_{2}}{\bz_1}{\bz_1^2})
          +(\amal{D_2^d}{D_{2}}{D_{2}}{\bz_1}{D_1^d})\\
          &+(\amal{D_2^d}{D_{2}}{D_{2}}{\bz_1}{D_1})
          +(\amal{D_2^z}{D_{2}}{D_{2}}{\bz_1}{D_1^d})
          +(\amal{D_2^z}{D_{2}}{D_{2}}{\bz_1}{\bz_2})\\
          &+(\amal{D_2}{D_{2}}{D_{2}}{\bz_1}{\bz_2})
          +(\amal{\bz_2^2}{D_{2}}{D_{2}}{\bz_1}{\bz_2})
          +(\amal{\bz_2^2}{D_{2}}{D_{2}}{\bz_1}{\bz_1^2})\\
          &-(\amal{D_4^2}{D_{1}}{D_{1}}{D_4}{})
          +(\amal{D_4^\td}{D_{1}}{D_{1}}{\tD_2}{})
          +(\amal{\tD_1^2}{D_{1}}{D_{1}}{\tD_1}{})\\
          &+2(\amal{\tD_2^\td}{D_{1}}{D_{1}}{\tD_1}{})
          +(\amal{\tD_2^\td}{D_{1}}{D_{1}}{\bz_2^z}{})
          +(\amal{\tD_2}{D_{1}}{D_{1}}{\tD_1}{})\\
          &+(\amal{D_1^2}{D_{1}}{D_{1}}{D_1}{})
          +2(\amal{D_2^d}{D_{1}}{D_{1}}{D_1}{})
          +(\amal{D_2^d}{D_{1}}{D_{1}}{\bz_2^z}{})\\
          &-4(\amal{D_1^d}{D_{1}}{D_{1}}{\bz_1}{})
          -(\amal{D_1}{D_{1}}{D_{1}}{\bz_1}{})
          -2(\amal{\tD_1^\td}{D_{1}}{D_{1}}{\bz_1}{})\\
          &-2(\amal{\tD_1}{D_{1}}{D_{1}}{\bz_1}{})
          -4(\amal{\bz_2^z}{D_{1}}{D_{1}}{\bz_1}{})
          -(\amal{\bz_2}{D_{1}}{D_{1}}{\bz_1}{})\\
          &-(\amal{\bz_1^2}{D_{1}}{D_{1}}{\bz_1}{})
          +\boldsymbol{(\amal{D_4^2}{D_{4}}{\bz_{2}}{D_4}{})}
          -(\amal{D_4^\td}{D_{4}}{\bz_{2}}{\tD_2}{})\\
          &-(\amal{\tD_2^\td}{D_{4}}{\bz_{2}}{\tD_1}{})
          -(\amal{D_2^d}{D_{4}}{\bz_{2}}{D_1}{})
          +(\amal{D_1^d}{D_{4}}{\bz_{2}}{\bz_1}{})\\
          &+(\amal{\bz_2^z}{D_{4}}{\bz_{2}}{\bz_1}{})
          -(\amal{D_4}{D_{2}}{}{}{})
          +(\amal{\tD_2}{D_{2}}{}{}{})\\
          &+(\amal{D_1}{D_{2}}{}{}{})
          -(\amal{\bz_1}{D_{2}}{}{}{}),
    \end{align*}
    where the underlined orbit types are maximal ones.

\end{example}
\vs
\subsection{Bifurcation-Type of Existence Results for \eqref{eq:New-L}}  In this subsection we are interested in studying the existence of branches of periodic (non-constant) solutions to \eqref{eq:int1} emerging from the zero solution. We assume that the function $f$ satisfies the conditions  (A1) and (A7). By rescaling the time this problem can be reformulated as \eqref{eq:New-L}, where $\lambda>0$ is considered as an additional parameter.  Next, we reformulate the equation \eqref{eq:New-L} in the functional space $\mathscr H$ as the bifurcation problem
\begin{equation}\label{eq:bif1}
\nabla J(\lambda,x)=0,\quad \lambda>0,
\end{equation}
 where the functional $J:\br_+\times \mathscr H\to \br$ is given by \eqref{eq:grad-J}, i.e. we are looking for branches of non-constant periodic solutions $(\lambda, x_\lambda)$ bifurcating from $(\lambda_o,0)$ (for some $\lambda_o>0$).  
 \vs
 Put $G:=\Gamma\times \times O(2)$. Then clearly, $J$ is a $G$-invanriat functional. 
 \vs
 Define  the operators $\mathscr A(\lambda):\mathscr H\to \mathscr H$ by 
 \[
 \mathscr A(\lambda):= \id -jL^{-1}\left(\lambda^2 N_{\nabla^2 f(0)}+\id   \right).
 \]
 and put $\Lambda:=\{\lambda>0: \;\; \mathscr A(\lambda):\mathscr H\to \mathscr H \;\text{ is not an isomorphism}\}$, i.e.
 \[
 \Lambda=\left\{\lambda=\frac{k}{\sqrt{\mu^A}} : k\in \bn,\; \mu^A\in \sigma(A),\; \mu^A>0  \right\}.
 \]
Then, the necessary condition for a point $(\lambda_o,0)$ to be a bifurcation point for \eqref{eq:bif1} is that $\lambda_o\in \Lambda$, which implies that
for some $k_o>0$ there exists $\mu_o\in \sigma(A)$ ($\mu_o>0$), so $\lambda^2_o=\frac{k_o^2+1}{\mu_o}$. Then by choosing a sufficiently small $\delta>0$ such that 
$[\lambda_o^-,\lambda_o^+]\cap \Lambda=\{\lambda_o\}$, where $\lambda_o^\pm :=\lambda_o\pm \delta$. Then we put  
\begin{align*}
\sigma_{\lambda^\pm_o} &= \sigma_-(\mathscr A(\lambda_o^{\pm})):=\{\xi\in \sigma(\mathscr A(\lambda_o^{\pm})): \xi<0\}\\
&=\bigcup_{k=1}^\infty \left\{ \xi=1-\frac{(\lambda_o^\pm)^2 \mu+1}{k^2+1} : \mu\in \sigma(A) ,\; \xi<0\right\}\\
&=:\bigcup _{k=1}^\infty \sigma^k_{\lambda^\pm_o}. 
\end{align*}
Then, clearly 
\[
\sigma_{\lambda_o^+}=\sigma_{\lambda_o^-} \cup\left\{\xi_o \right\}, \quad \text{  where }\; \xi_o:= 1- \frac {(\lambda_o^+)^2+1\mu_o}{k^2_o+1} .
\]
For each $\lambda_o\in \Lambda)$, we choose a sufficiently small $\ve>0$ and put $\Omega_o:=B_\ve(0)$. Then we define the $G$-equivariant bifurcation invariant $\omega(\lambda_o)\in U(G)$ by
\[
\omega(\lambda_o):= \nabla_G\text{\rm-deg}(\nabla J(\lambda_o^-,\cdot ),\Omega_p)- \nabla_G\text{\rm-deg}(\nabla J(\lambda_o^+,\cdot ),\Omega_p).
\]
Notice that
\begin{align*}
\omega(\lambda_o)&=\nabla_G\text{\rm-deg}(\mathscr A(\lambda_o^-),\Omega_o)-\nabla_G\text{\rm-deg}(\mathscr A(\lambda_o^+),\Omega_o)\\
&=  \mathfrak a *\prod_{k\not=k_o}^\infty \prod_{\xi\in \sigma^k_{\lambda_o^-}}\prod_{j=0}^s\left(\text{\rm Deg}_{\mathcal W_{kj}}\right)^{m_{k,j}(\xi)}*\left[ (G)-\prod_{j=0}^s \left(\text{\rm Deg}_{\mathcal W_{k_oj}}\right)^{m_{k_o,j}(\xi_o)}\right],
\end{align*}
where 
\[
\mathfrak a:= \prod_{\xi\in \sigma^0_{\lambda_o^-}}\prod_{j=0}^r\left(\text{\rm Deg}_{\mathcal V_{j}}\right)^{m_{0,j}(\xi)}.
\]
\vs
The following result is a direct consequence of the properties of the equivariant degree theory.
\vs
\begin{theorem}\label{th:bif1}  Assume that the function $f$ satisfies the conditions  (A1) and (A7) and $\lambda_o\in \Lambda$. Suppose that 
\[
\omega(\lambda_o)=n_1(H_1)+n_2(H_2)+\dots+ n_k(H_k), \quad n_j\not=0, \; j=1,,2,\dots,k.
\]
Then  for every $j=1,2,\dots,k$, the system \eqref{eq:New-L} has a branch of non-constant $2\pi$-periodic solutions $\{(\lambda_s,x_s)\}$  bifurcating from $(\lambda_o,0)$ such that $(G_{x_s})\ge (H_j)$. This is equivalent to the existence of a branch of non-constant $p_s$-periodic solutions $\{x_s\}$ of the system \eqref{eq:int1} emerging from $0$ with the limit period $p_o=2\pi \lambda_o$).  Moreover, if $(H_j)$ is a maximal $\mathfrak f$-orbit type in $\mathscr H$ then the orbit $G(x_s)$ contains $\left| G/H_j \right|_{S^1}$ periodic non-constant solutions (where $\left| G/H_j \right|_{S^1}$  stands for the number of $S^1$-orbits in $G/H_j$).
\end{theorem}
\vs
\begin{remark}\label{rem:glob-bif}\rm
Under additional assumption (A4), consider the set  
\[
\mathcal S:=\{(\lambda,x)\in \mathbb R_+\times \mathscr H: \nabla J(\lambda,x)=0\;\text{ and } x\not=0\}.
\]
Then by the well-known results (see for example \cite{DR}) we have that   $(0,0)\notin \overline{\mathcal S}$ (we also sketched the proof in the Appendix). Therefore, 
\[
\br\times \{0\} \cap \overline{\mathcal S}\subset \Lambda\times \{0\}.
\]
Then, we have the following well-known (see for example \cite{GolRyb}) Rabinowitz-type (cf. \cite{Rab}) global `bifurcation' result: {\it Suppose $\mathcal C$ is a connected bounded component of $\overline{\mathcal S}$. Then $\mathcal C$ is compact 
\[\Lambda\times \{0\} \cap \mathcal C=\{(\lambda_1,,0),(\lambda_2,0),\dots,(\lambda_k,0)\}\]
and 
\[
\omega(\lambda_1)+\omega(\lambda_2)+\dots+\omega(\lambda_k)=0.
\]
}
\end{remark}
\vs
\begin{remark} \rm 
In this paper for the sake of simplicity we consider only the case of a non-degenerate stationary solutions for \eqref{eq:int1}. However, similar methods can be applied to study the existence of  periodic solutions of autonomous Hamiltonian systems emanating from degenerate stationary solutions (see for example \cite{DR,RaRy,FRRuan}).
\end{remark}
\vs
In summary, under the assumptions (A1) and (A7), the existence of branches of non-constant periodic solutions for the system \eqref{eq:int1}  emanating from the origin and their symmetric properties 
are topologically classified by the set $\Lambda=\{\lambda_1,\lambda_2,\dots, \lambda_k\}$ (which provides the information about the limit periods) and the values of the equivariant invariants $\{\omega(\lambda_1),\omega(\lambda_2),\dots, \omega(\lambda_k)\}$, which describe the possible symmetries of these branches of periodic solutions.
\vs
\begin{example}\rm 
As a {\bf numerical example} we consider  $\Gamma=S_4$ acting on $V=\br^4$ by permuting the coordinates of vectors.  This representation was already analyzed in Example \ref{ex:1}, including  the $S_4$-isotypical decomposition
$V=V_0\oplus  V_3$. 
We assume that $f:V\to \br$ satisfies (A1), (A7) and the matrices $A$ is of the  type
\[
A:=\left[ \begin{array}{cccc}c&d&d&d\\
d&c&d&d\\
d&d&c&d\\
d&d&d&c
 \end{array}        \right]
\]
where $c=18.8$ and $d=-1.5$.
Then 
$\sigma(A)=\{\mu_0^A=14.3,\, \mu_3^A=20.3\}$
which implies
that 
\[
\sigma(\mathscr A(\lambda))=\left\{ 1-\frac{\lambda^2\cdot 14.3+1 }{k^2+1}, \, 1-\frac{\lambda^2\cdot 20.3+1}{k^2+1}: k=0,1,2,\dots  \right\}
\]
and 
\begin{align*}
\Lambda&=\left\{\lambda_3^k:= \frac{k}{\sqrt{20.3}} ,\; \lambda_0^k:= \frac{k}{\sqrt{14.3}} : k\in \bn\right\}\\
&=\{\lambda_3^1=0.22194838,\, \lambda_0^1=0.26444294,\, \lambda_3^2=0.44389676,\, \lambda_0^2=0.52888589,\, \dots   \}.
\end{align*}
By using our computational database,  the exact values of the $G$-equivariant bifurcation invariants $\omega(\lambda_3^k)$ and $\omega(\lambda_0^k)$ can be easily evaluated. For example we have
\begin{align*}
\omega(\lambda_3^1)&=(S_4\times O(2))-\text{Deg}_{\mathcal V_{3,1}}\\
&= \boldsymbol{ (\amal{D_4}{D_{2}}{\mathbb{Z}_{2}}{D_2}{})}
        +(\amal{D_3}{D_{1}}{}{}{})
        +\boldsymbol{(\amal{D_4}{D_{4}}{D_{4}}{\mathbb{Z}_1}{})}\\
      & -\boldsymbol{(\amal{D_2}{D_{2}}{\mathbb{Z}_{2}}{D_1}{})}
        +(\amal{D_3}{D_{3}}{D_{3}}{\mathbb{Z}_1}{})
        -(\amal{D_2}{D_{2}}{D_{2}}{\mathbb{Z}_1}{D_1})
        -(\amal{V_4}{D_{2}}{D_{2}}{\mathbb{Z}_1}{})\\
      & -(\amal{D_2}{D_{1}}{D_{1}}{D_1}{})
        -(\amal{\mathbb{Z}_2}{D_{2l}}{\mathbb{Z}_{2}}{\mathbb{Z}_1}{})
        -2(\amal{D_1}{D_{1}}{}{}{})
        +(\amal{\mathbb{Z}_2}{D_{1}}{D_{1}}{\mathbb{Z}_1}{})\\
      & +(\amal{\mathbb{Z}_1}{D_{1}}{}{}{}),      
\end{align*}
where the maximal orbit types are marked in bold. 
\end{example}

\vs
\subsection{Appendix: Minimal period of Non-constant periodic solutions}\label{sec:App}
Suppose that the function $f$ satisfies the assumptions (A1), (A4) and (A7). Let us use the following result of Vidossich (cf. \cite{Vid})
\begin{lemma}\label{lem:Vid}
Let $\mathbb E$ be a Banach space, $V:\br \to \mathbb E$ a $p$-periodic function with the following properties:
\begin{itemize}
\item[(i)] $V$ is integrable and $\int_0^p V(t)dt=0$

\item[(ii)] There exists $U\in L^1([0,p/2];\br_+)$ such that $\|V(t)-V(s)\|\le U(t-s)$ for almost all $t$, $s$ with $0\le s\le t\le p$ and $t-s\le p/2$.
\end{itemize}
Then 
\[
p\sup_{t\in \br}\|V(t)\|\le 2\int_0^{\frac p2} U(t)dt.
\]
\end{lemma}
\vs

Consider a non-constant $p$-periodic solution $x(t)$ to \eqref{eq:int1} and assume that $M>0$ is a constant such that  $\|x(t)\|\le M$.  Put $V(t)=\dot x(t)$. Then we have for $o\le s\le t\le p$
\begin{align*}
|\dot x(t)-\dot x(s)|&\le \sup_{|x|\le M}\|\nabla^2 f(x)\|\int_s^t|\dot x(\tau)|\tau d\tau\\
&\le 2p \sup_{|x|\le M}\|\nabla^2 f(x)\|\cdot \sup_{\tau\in \br} |\dot x(\tau)|\cdot (t-s).
\end{align*}
Put $K:=\sup_{|x|\le M}\|\nabla^2 f(x)\|$ and $U(t):=2pK \sup_{\tau\in \br} |\dot x(\tau)|t$. Then by Lemma \ref{lem:Vid} we have
\[
p \sup_{\tau\in \br} |\dot x(\tau)|\le 2\int_0^{\frac p2} U(t)dt=4pK \sup_{\tau\in \br} |\dot x(\tau)|\frac {p^2}4,
\]
which implies
\[
p^2\ge \frac 1K.
\]

\end{document}